\documentclass{article}
\usepackage{arxiv}
\usepackage{comment}
\usepackage[utf8]{inputenc} 
\usepackage[T1]{fontenc}    
\usepackage{hyperref}       
\usepackage{url}            
\usepackage{booktabs}       
\usepackage{amsmath}
\usepackage{amssymb}
\usepackage{amsfonts}       
\usepackage{nicefrac}       
\usepackage{microtype}      
\usepackage{xcolor}
\usepackage{graphicx}
\usepackage{stmaryrd}
\usepackage{bm}
\usepackage[normalem]{ulem}
\usepackage[tickmarkheight=0.1cm,textsize=tiny]{todonotes}
\usepackage{algorithmicx}
\usepackage{algorithm}
\usepackage{algpseudocode}
\usepackage{caption}
\usepackage{subcaption}
\usepackage{lineno}
\usepackage{orcidlink}
\usepackage[inline,final]{showlabels}

\usepackage{xpatch}
\makeatletter
\xpatchcmd{\algorithmic}
  {\ALG@tlm\z@}{\leftmargin\z@\ALG@tlm\z@}
  {}{}
\makeatother

\hypersetup{
    colorlinks=true, 
    linkcolor=blue, 
    urlcolor=red, 
    citecolor=magenta,
    linktoc=all 
}

\newlength\myindent 
\newcommand{\re}{\mathbb{R}} 
\newcommand{\en}{\mathbb{E}} 
\newcommand{\thresh}{\mathbb{T}} 
\newcommand{\slope}{\delta}
\newcommand{\uspx}{\bw^{\ast, +x}}
\newcommand{\uspmx}{\bw^{\ast, \pm x}}
\newcommand{\uspmy}{\bw^{\ast, \pm y}}

\newcommand{\uspxy}{\bw^{\ast, +xy}}

\newcommand{\unphpx}{\bw^{\nph, +x}}
\newcommand{\unphmx}{\bw^{\nph, -x}}
\newcommand{\unphpmx}{\bw^{\nph, \pm x}}
\newcommand{\unphpmy}{\bw^{\nph, \pm y}}
\newcommand{\unphsx}{\bw^{\nph, \ast x}}
\newcommand{\unphsy}{\bw^{\nph, \ast y}}
\newcommand{\unphpy}{\bw^{\nph, +y}}
\newcommand{\unphmy}{\bw^{\nph, -y}}

\newcommand{\unpx}{\bw^{n, +x}}
\newcommand{\unmx}{\bw^{n, -x}}
\newcommand{\unpy}{\bw^{n, +y}}
\newcommand{\unmy}{\bw^{n, -y}}
\newcommand{\unpmx}{\bw^{n, \pm x}}
\newcommand{\unpmy}{\bw^{n, \pm y}}
\newcommand{\thetapx}{\mathcal{\theta}_{ij}^{+x}}

\newcommand{\thetapy}{\mathcal{\theta}_{ij}^{+y}}

\newcommand{\utilow}{\tilde{\bw}^{\text{low},n+1}}

\newcommand{\poly}{\mathbb{P}} 

\newcommand{\emh}{{e-\frac{1}{2}}}
\newcommand{\eymh}{{e_y-\frac{1}{2}}}
\newcommand{\eph}{{e+\frac{1}{2}}}
\newcommand{\exph}{{e_x+\frac{1}{2}}}
\newcommand{\iph}{{i+\frac{1}{2}}}
\newcommand{\imh}{{i-\frac{1}{2}}}
\newcommand{\ipmh}{{i\pm\frac{1}{2}}}
\newcommand{\jmh}{{j-\frac{1}{2}}}
\newcommand{\jph}{{j+\frac{1}{2}}}
\newcommand{\jpmh}{{j\pm\frac{1}{2}}}

\newcommand{\nph}{{n+\frac{1}{2}}}
\newcommand{\Nmh}{{N-\frac{1}{2}}}
\newcommand{\Nph}{{N+\frac{1}{2}}}
\newcommand{\half}{\frac{1}{2}}
\newcommand{\ud}{\textrm{d}}
\newcommand{\pd}[2]{\frac{\partial #1}{\partial #2}}
\newcommand{\od}[2]{\frac{\ud #1}{\ud #2}}

\newcommand{\mum}{\mu_-}
\newcommand{\mup}{\mu_+}
\newcommand{\mumx}{\mu_{-x}}
\newcommand{\mupx}{\mu_{+x}}
\newcommand{\mupmx}{\mu_{\pm x}}
\newcommand{\mumy}{\mu_{-y}}
\newcommand{\mupy}{\mu_{+y}}
\newcommand{\mupmy}{\mu_{\pm y}}

\newcommand{\ee}{\mathbf{e}}
\newcommand{\ex}{e_x}
\newcommand{\ey}{e_y}
\newcommand{\bo}{\mathbf{0}}
\newcommand{\expmh}{e_x \pm\half}
\newcommand{\exmh}{\ex - \half}
\newcommand{\eypmh}{\ey \pm \half}

\newcommand{\eexpmh}{{(\expmh, \ey)}}

\newcommand{\eexmh}{{(\exmh, \ey)}}
\newcommand{\eeypmh}{{(\ex, \eypmh)}}

\newcommand{\eeymh}{{(\ex, \eymh)}}

\newcommand{\avg}[1]{\overline{#1}}
\newcommand{\au}{\avg{u}}

\newcommand{\myvector}[1]{\mathsf{#1}}
\newcommand{\vu}{\myvector{u}}

\newcommand{\vR}{\myvector{R}}

\DeclareMathOperator{\sign}{sign}
\DeclareMathOperator{\minmod}{minmod}

\newcommand{\Uad}{\mathcal{U}_{\textrm{ad}}}

\newtheorem{remark}{Remark}
\newtheorem{definition}{Definition}

\usepackage{graphicx,latexsym,tabularx,xcolor}

\newcommand{\nobracket}{}

\newcommand{\tmop}[1]{\ensuremath{\operatorname{#1}}}

\newenvironment{proof}{\noindent\textbf{Proof\ }}{\hspace*{\fill}$\Box$\medskip}
\newtheorem{lemma}{Lemma}

\newtheorem{theorem}{Theorem}

\newcommand{\bw}{u}
\newcommand{\bv}{v}
\newcommand{\ff}{f}
\newcommand{\fg}{g}

\newcommand{\correction}[1]{{#1}}

\title{Admissibility preserving subcell limiter for Lax-Wendroff flux reconstruction}

\author{
Arpit~Babbar \orcidlink{0000-0002-9453-370X} \\
Centre for Applicable Mathematics\\
Tata Institute of Fundamental Research\\
Bangalore -- 560065\\
\texttt{arpit@tifrbng.res.in} \\
\And
Sudarshan~Kumar~Kenettinkara \orcidlink{0009-0005-1186-4167}\\
School of Mathematics\\
Indian Institute of Science Education and Research\\
Trivandrum -- 695551\\
\texttt{sudarshan@iisertvm.ac.in} \\
\And
Praveen~Chandrashekar \orcidlink{0000-0003-1903-4107}\thanks{Corresponding author}\\
Centre for Applicable Mathematics\\
Tata Institute of Fundamental Research\\
Bangalore -- 560065\\
\texttt{praveen@math.tifrbng.res.in}
}
\begin{document}
\maketitle
\begin{abstract}
Lax-Wendroff Flux Reconstruction (LWFR) is a single-stage, high order, quadrature free method for solving hyperbolic conservation laws. We develop a subcell based limiter by blending LWFR with a lower order scheme, either first order finite volume or MUSCL-Hancock scheme. While the blending with a lower order scheme helps to control \correction{spurious} oscillations, it may not guarantee admissibility of discrete solution, e.g., positivity property of quantities like density and pressure. By exploiting the subcell structure and admissibility of lower order schemes, we devise a strategy to ensure that the blended scheme is admissibility preserving for the mean values and then use a scaling limiter to obtain admissibility of the polynomial solution. For MUSCL-Hancock scheme on non-cell-centered subcells, we develop a slope limiter, time step restrictions and suitable blending of higher order fluxes, that ensures admissibility of lower order updates and hence that of the \correction{cell averages}. By using the MUSCL-Hancock scheme on subcells and Gauss-Legendre points in flux reconstruction, we improve small-scale resolution compared to the subcell-based RKDG blending scheme with first order finite volume method and Gauss-Legendre-Lobatto points. We demonstrate the performance of our scheme on compressible Euler's equations, showcasing its ability to handle shocks and preserve small-scale structures.
\end{abstract}
\keywords{Conservation laws \and hyperbolic PDE \and Lax-Wendroff \and flux reconstruction \and Shock Capturing \and Admissibility preservation}
\section{Introduction}

The current state of memory-bound HPC hardware~\cite{attig2011,subcommittee2014} makes a strong case for development of \correction{high order} discrete methods in computational fluid dynamics (CFD). By incorporating more higher order terms, these methods can achieve greater numerical accuracy per degree of freedom while minimizing memory usage and data transfers. In particular, \correction{high order} methods have higher arithmetic intensity and are thus less likely to be memory-bound. However, lower order methods are still routinely applied in practical applications, in part due to their robustness. This work is in direction of using high order methods while retaining the robustness properties of lower order methods.

Discontinuous Galerkin (DG) is a Spectral Element Method first introduced by Reed and Hill~\cite{reed1973} for neutron transport equations and developed for fluid dynamics equations by Cockburn and Shu and others, see~\cite{cockburn2000} and the references therein. The DG method uses an approximate solution which is a polynomial within each element and is allowed to be discontinuous across interfaces. The neighbouring DG elements are coupled only through the numerical flux and thus bulk of computations are local to the element, minimizing data transfers.

Flux Reconstruction (FR) is also a class of discontinuous Spectral Element Methods introduced by Huynh~\cite{Huynh2007}. FR method is obtained by using the numerical flux and correction functions to construct a continuous flux approximation and collocating the differential form of the equation. Thus, FR is quadrature free and all local operations can be vectorized. The choice of the correction function affects the accuracy and stability of the method~\cite{Huynh2007,Vincent2011a,Vincent2015,Trojak2021}; by properly choosing the correction function and solution points, FR method can be shown to be equivalent to some discontinuous Galerkin and spectral difference schemes~\cite{Huynh2007,Trojak2021}. \correction{In~\cite{Cicchino2022}, a nonlinearly stable FR scheme was constructed in split form where a key idea was application of correction functions to the volume terms.}

FR and DG are procedures for discretizing the equation in space and can be used to obtain a system of ODEs in time, i.e., a semi-discretization of the PDE. The standard approach to solve the semi-discretization is to use a high order multi-stage Runge-Kutta method. In this approach, the spatial discretization has to be performed in every RK stage and thus the expensive operations of MPI communication and limiting have to be performed multiple times per time step.

An alternative approach is to use a single-stage solver, see~\cite{babbar2022} for a more in-depth review. ADER (Arbitrary high order DERivative) schemes are one class of single-stage evolution methods which originated as ADER Finite Volume (FV) schemes~\cite{Titarev2002,Titarev2005} and are also used as ADER-DG schemes~\cite{Dumbser2008,Dumbser2014}. Another class of single-stage evolution methods are the Lax-Wendroff schemes which were originally proposed in the finite difference framework in~\cite{Qiu2003} with the WENO approximation of spatial derivatives~\cite{Shu1989} and later extended to DG framework in~\cite{Qiu2005b}. These schemes were based on computation of flux Jacobian which leads to a problem dependent and expensive procedure especially for higher dimension equations with many variables. A Jacobian free method, called \textit{approximate Lax-Wendroff}, using finite differences in time was developed in~\cite{Zorio2017} and further studied in several other works~\cite{Burger2017,Carrillo2021, Carrillo2021a}. In ~\cite{Lee2021}, the flux Jacobians were directly computed as finite difference derivatives with
a parameter $\epsilon$ for accuracy of finite difference. In~\cite{Burger2017}, the approximate Lax-Wendroff procedure was used in the DG framework and the performance benefit of Jacobian free methods was observed.

In~\cite{babbar2022}, the present authors proposed a Lax-Wendroff Flux Reconstruction (LWFR) scheme which used the approximate Lax-Wendroff procedure of~\cite{Zorio2017} to obtain high order accuracy. The numerical flux was carefully constructed in~\cite{babbar2022}; the dissipative part of the numerical flux was computed with the time averaged solution (called D2 dissipation) leading to an upwind flux in the linear case and improved CFL numbers at no additional computational cost. The central part of the numerical flux was computed by performing the approximate Lax-Wendroff procedure at the faces (EA scheme) rather than using the extrapolated time averaged flux from solution points (AE scheme). It was observed that the EA scheme improved accuracy of the LWFR scheme and some tests showed optimal order of convergence only with the EA scheme. Thus, in this work, we use the LWFR scheme with numerical flux computed with D2 dissipation and EA scheme.

Although arbitrary high order accuracy in smooth test cases and Wall Clock Time (WCT) performance improvement over RKFR (Runge-Kutta Flux Reconstruction) was observed for the LWFR scheme proposed in~\cite{babbar2022}, robustness and maintaining \correction{high order} accuracy in presence of discontinuities remained to be addressed. Solutions to hyperbolic conservation laws contain shocks in many practical applications and it is well known that \correction{high order} schemes produce spurious oscillations in those cases. These oscillations can lead not only to incorrect solutions but can also easily generate nonphysical solutions like gases with negative density or pressure. Thus, these schemes require limiters which reduce the \correction{high order} scheme to a robust lower order scheme in non-smooth regions. In~\cite{babbar2022}, a TVB limiter was used which reduces the scheme to first order or linear in FR elements using a minmod function (Section~\ref{sec:tvd}). The TVB limiter is inadequate for the following reasons - it doesn't preserve any subcell information other than the element mean and trace values, and it is not provably admissibility preserving for Lax-Wendroff schemes even when used with the scaling limiter of Zhang and Shu~\cite{Zhang2010b}. The second issue has been considered in~\cite{moe2017, Xu2022} by modifying the numerical flux to obtain admissibility in means making the scaling limiter applicable. In~\cite{moe2017}, admissibility in means is obtained by limiting the numerical flux. In~\cite{Xu2022}, a third order maximum-principle satisfying Lax-Wendroff DG scheme is constructed using the direct DG numerical flux from~\cite{Chen2016}. We now give a brief literature review of the schemes which deal with the first issue of TVB limiter which is preservation of subcell information, see~\cite{Vilar2019, Dumbser2014} for a more in-depth review.

Moment limiters~\cite{Biswas1994,Burbeau2001, Krivodonova2007} can be seen as an extension of TVB limiters where coefficients in an orthonormal basis (moments) are limited in a decreasing sequence, from higher to lower degree. The hierarchical nature of moment limiters enables preservation of subcell information. Another popular strategy is the (H)WENO limiting
procedure~\cite{Qiu2005,Balsara2007}, where the DG polynomial is substituted in troubled regions by a reconstructed (H)WENO polynomial that is computed by a WENO procedure using subcell and neighboring cells information. There are also the methods of artificial viscosity where a second order diffusion term is added in elements where the solution is non-smooth, preserving the subcell information as the high order polynomial solution is still used. In~\cite{Persson2006}, an artificial viscosity model was introduced for the Runge-Kutta (RK) Discontinuous Galerkin (DG) method to add dissipation to the \correction{high order} method based on a modal smoothness indicator. The indicator of~\cite{Persson2006} was further refined and detailed in~\cite{klockner2011}.

There have also been several schemes which limit the solution by breaking the element into subcells which offers some advantages over artificial viscosity methods, including problem independence over boundary conditions and no additional time step restrictions, even when high dissipation is required, as noted in~\cite{henemann2021}. In~\cite{Peraire2012}, the modal smoothness indicator of~\cite{Persson2006} was used to adapt local basis functions, e.g., switching to finite volume basis in the presence of discontinuities. In~\cite{Peraire2013}, subcells were used to assign different values to artificial viscosity within each element. In~\cite{Sonntag2014,Rosa2018}, after having detected the troubled zones using the modal indicator of~\cite{Persson2006}, cells are subdivided into subcells, and a robust first-order finite volume scheme is performed on the subgrid in troubled cells. In~\cite{henemann2021}, the modal smoothness indicator of~\cite{Persson2006} was used to perform limiting by blending a high order DG scheme with Gauss-Legendre-Lobatto (GLL) points with a lower order finite volume scheme on subcells. In~\cite{Ramirez2020}, the method of~\cite{henemann2021} was extended to compressible magnetohydrodynamics (MHD) and \correction{high order} reconstruction on subcells was used to improve accuracy. In~\cite{ramirez2021}, it was shown that the subcell FV method of~\cite{henemann2021} can be made positivity preserving by an \textit{a posteriori} modification of the blending coefficient. In~\cite{ramirez2022}, the subcell finite volume method of~\cite{henemann2021} with Rusanov's flux~\cite{Rusanov1962} was shown to be equivalent to the sparse Invariant Domain Preserving method of Pazner~\cite{Pazner2021}.

The approaches explained above can be classified as \textit{a priori} limiters. We briefly discuss \textit{a posteriori} limiting techniques \correction{where the solution is updated to time $t^{n+1}$, and low order re-updates are conducted in the elements that fail certain carefully chosen admissibility checks}. One of these is the MOOD technique~\cite{Clain2011,Diot2012,Diot2013} where the local re-updates are computed with reduced order of accuracy until the admissibility checks pass. In~\cite{Dumbser2014,Dumbser2019}, the subcell based technique of~\cite{Sonntag2014,Rosa2018} is applied in an \textit{a posteriori} fashion using $2N+1$ subcells for $N+1$ degrees of freedom per element in the 1-D case, using least squares approximation to convert back to a degree $N$ polynomial. In case the least square transformation leads to violation of admissibility constraints, the subcell solution values are used in the next evolution and thus the scheme is guaranteed to not crash. In~\cite{Vilar2019}, the DG scheme was reformulated as subcell Finite Volume (FV) method with appropriate subcells. An indicator was used to mark troubled subcells and thus the solution could be modified in a very localized manner, preserving subcell information well.

Other techniques for shock capturing exist that do not fit strictly into the aforementioned categories. In~\cite{dzanic2022}, positivity preservation and shock capturing were achieved by filtering and enforcing the minimum entropy principle, while in~\cite{Lu2021}, a numerical damping term was introduced in the DG scheme to control spurious oscillations.

In this work, we use the \textit{a priori} blending limiter of~\cite{henemann2021} for LWFR as its choice of subcells gives a natural correction to the time averaged numerical flux to obtain admissibility preservation in means. The key idea of the blending scheme is to reduce spurious oscillations by using a low order scheme in regions where the solution is not smooth, as detected by a smoothness indicator. The blending limiter by itself is not guaranteed to control all oscillations and thus unphysical solutions may still be obtained. Thus, we perform additional limiting to obtain a provably admissibility preserving scheme. Special attention is also paid to improving accuracy to capture small scale structures.  We use Gauss-Legendre (GL) solution points and subcells obtained from GL quadrature weights instead of the GLL points and weights used in~\cite{henemann2021}. This is because of their accuracy advantage as observed by us, and as reported in the literature. In the non-linear stability analysis for E-fluxes in~\cite{Jameson2012}, Gauss-Legendre points were found to be the most resistant to aliasing driven instability. In another study on accuracy with different choices of solution points~\cite{Witherden2021}, the optimality of Gauss-Legendre points was again observed. In~\cite{babbar2022}, optimal convergence rates for some non-linear problems were observed only for Gauss-Legendre solution points.

As observed in~\cite{ramirez2021}, accuracy can be improved by performing a \correction{high order} reconstruction on the subcells. Since LWFR is a single-stage method, we improve accuracy by using the single-stage, second order MUSCL-Hancock scheme~\cite{vanleer1984} on the subcells. As explained in~\cite{henemann2021}, \correction{for a DG method of degree $N$,} maintaining conservation requires the subcell sizes to be given by the $N+1$ quadrature weights and the solution points to be the solution points of DG scheme. This implies that the subcells are non-uniform and the finite volumes are neither cell-centered nor vertex centered. Thus, as a first step to ensuring that the blended scheme is admissible, we extend the work of~\cite{Berthon2006} to the non-cell centered grids that occur from demanding conservation in the blending scheme. Enforcing admissibility as in~\cite{Berthon2006} requires an additional slope limiting step and we propose a problem independent procedure to do the same.

In order to maintain conservation, low and high order updates need to use the same numerical flux at the FR element interfaces (see Remark~\ref{rmk:why.same.flux}). This numerical flux has to be chosen by blending between the high order time averaged flux and the low order FV flux. Thus, as the next step to enforce admissibility of the blended Lax-Wendroff scheme, we carefully select the blended numerical flux using a scaling procedure to ensure that the lower order updates at solution points neighboring the interfaces are admissible.

In~\cite{ramirez2021}, the blending limiter of~\cite{henemann2021} has been made admissibility preserving by changing the blending coefficients in an \textit{a posteriori} fashion. \correction{Since our choice of the blended numerical flux implies the admissibility of lower order updates at all solution points, we could take the same approach}. In this work, we instead use the fact that, with the blended numerical flux, admissibility of lower order scheme implies admissibility in the means of the blended scheme and thus the scaling limiter of~\cite{Zhang2010b} can now be used to obtain an admissibility preserving scheme. In~\cite{moe2017}, a correction has been made to the Lax-Wendroff numerical flux enforcing the admissibility in means property and then the scaling limiter~\cite{Zhang2010b} has been used to obtain an admissibility preserving Lax-Wendroff scheme. Our work differs from~\cite{moe2017} as we only target to ensure admissibility of the lower order scheme and the admissibility in means is consequently obtained. This implies that our correction requires less storage and doesn't require additional loops, minimizing memory reads.

The rest of this paper is organized as follows. In Section~\ref{sec:scl}, we review the Lax-Wendroff Flux Reconstruction scheme proposed in~\cite{babbar2022}. In Section~\ref{sec:controlling.oscillations}, we explain the blending limiter extended to Gauss-Legendre solution points including a review of the smoothness indicator used in~\cite{henemann2021} and the\correction{n } MUSCL-Hancock reconstruction performed on the subcells \correction{in Section~\ref{sec:mh}}. \correction{Maintaining} conservation requires that at the faces of FR elements, both the lower and high order schemes must use the same numerical flux (see Remark~\ref{rmk:why.same.flux}). In Section~\ref{sec:admissibility.preservation}, we show how to construct the numerical flux to ensure admissibility preservation in means. In Section~\ref{sec:alg}, we explain our implementation of the Lax-Wendroff blended scheme as an algorithm. The numerical results verifying accuracy and robustness of our scheme with 1-D and 2-D compressible Euler equations are shown in Section~\ref{sec:num.results}. Section~\ref{sec:sum} gives a summary of the proposed blended scheme.

In Appendix~\ref{sec:muscl.admissibility.proof}, we give details on extension of the  proof of admissibility of MUSCL-Hancock scheme of~\cite{Berthon2006} to non-cell centered grids and a description of the 2-D MUSCL-Hancock scheme. In Appendix~\ref{sec:2d.admissibility}, we give details on extending the choice of blended numerical flux to ensure admissibility and accuracy in 2-D.
\section{Lax-Wendroff FR scheme}\label{sec:scl}
Consider a conservation law of the form
\begin{equation}
\bw_t + \ff(u)_x = 0 \label{eq:con.law}
\end{equation}
where $u \in \re^p$ is the vector of conserved quantities, $f(u)$ is the corresponding flux, together with some initial and boundary conditions. The solution that is physically correct is assumed to belong to an \correction{admissibility set}, denoted by $\Uad$. For example in case of compressible flows, the density and pressure (or internal energy) must remain positive. In case of shallow water equations, the water depth must remain positive. In most of the models that are of interest, the \correction{admissibility set} is a convex subset of $\re^p$, and can be written as
\begin{equation} \label{eq:uad.form}
\Uad = \{ u \in \re^p : p_k(u) > 0, \ 1 \le k \le K \}
\end{equation}
where each admissibility constraint $p_k$ is concave if $p_{j} > 0$ for all $j < k$. For Euler's equations, $K=2$ and $p_1, p_2$ are density, pressure functions, respectively; if the density is positive then pressure is a concave function of the conserved variables.

For the numerical solution, we will divide the computational domain $\Omega$ into disjoint elements $\Omega_e$, with
\[
\Omega_e = [x_\emh, x_\eph] \qquad\textrm{and}\qquad \Delta x_e = x_\eph - x_\emh
\]
Let us map each element to a reference element, $\Omega_e \to [0,1]$, by
\[
x \to \xi = \frac{x - x_\emh}{\Delta x_e}
\]
Inside each element, we approximate the solution by degree $N \ge 0$ polynomials belonging to the set $\poly_N$. For this, choose $N+1$ distinct nodes
\begin{equation} \label{eq:xi0.defn}
0 \le \xi_0 < \xi_1 < \cdots < \xi_N \le 1
\end{equation}
which will be taken to be Gauss-Legendre (GL) or Gauss-Lobatto-Legendre (GLL) nodes, and will also be referred to as {\em solution points}. There are associated quadrature weights $w_j$ such that the quadrature rule is exact for polynomials of degree up to $2N+1$ for GL points and upto degree $2N-1$ for GLL points. Note that the nodes and weights we use are with respect to the interval $[0,1]$ whereas they are usually defined for the interval $[-1,1]$. The solution inside an element is given by
\[
x \in \Omega_e: \qquad u_h(\xi,t) = \sum_{j=0}^N u_j^e(t) \ell_j(\xi)
\]
where each $\ell_j$ is a Lagrange polynomial of degree $N$ given by
\[
\ell_j(\xi) = \prod_{i=0, i\ne j}^N \frac{\xi - \xi_i}{\xi_j - \xi_i} \in \poly_N, \qquad \ell_j(\xi_i) = \delta_{ij}
\]
Figure~(\ref{fig:solflux1}a) illustrates a piecewise polynomial solution at some time $t_n$ with discontinuities at the element boundaries. Note that the coefficients $u_j^e$ which are the basic unknowns or {\em degrees of freedom} (dof), are the solution values at the solution points.
\begin{figure}
\begin{center}
\begin{tabular}{cc}
\includegraphics[width=0.40\textwidth]{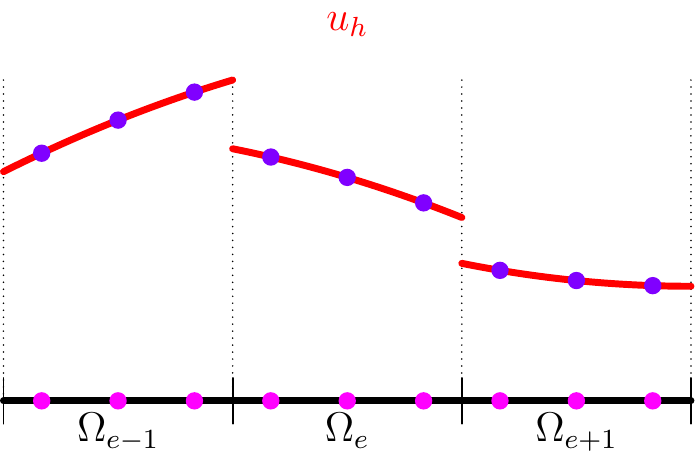} &
\includegraphics[width=0.40\textwidth]{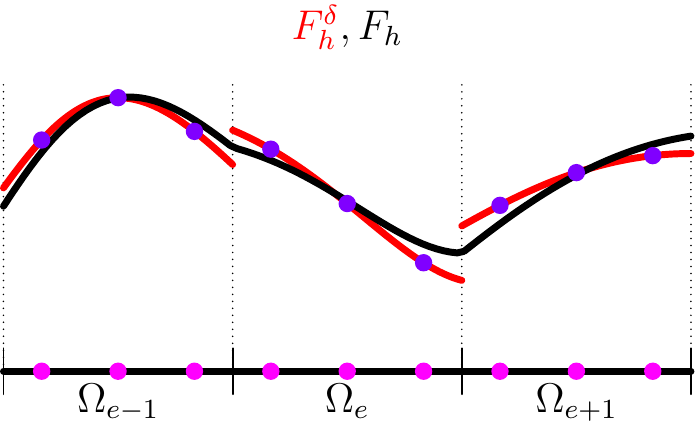} \\
(a) & (b)
\end{tabular}
\end{center}
\caption{(a) Piecewise polynomial solution at time $t_n$, and (b) discontinuous and continuous flux.}
\label{fig:solflux1}
\end{figure}
The update equation is given by
\begin{equation}
(u_j^e)^{n+1} = (u_j^e)^n - \frac{\Delta t}{\Delta x_e} \od{F_h}{\xi}(\xi_j), \qquad 0 \le j \le N
\label{eq:uplwfr}
\end{equation}
which is a single-stage scheme at all orders of accuracy. The quantity $F_h$ is a time average flux which is continuous in space and is computed using the flux reconstruction approach; it is given by
\begin{equation}
\label{eq:frcontflux}
F_h(\xi) = \left[F_\emh - F_h^\delta(0) \right] g_L(\xi) + F_h^\delta(\xi) + \left[F_\eph - F_h^\delta(1) \right] g_R(\xi)
\end{equation}
where $F_h^\delta$ is the discontinuous flux obtained by interpolation at the solution points
\[
F_h^\delta(\xi) = \sum_{j=0}^N F_j^e \ell_j(\xi)
\]
The discontinuous and continuous fluxes are illustrated in Figure~(\ref{fig:solflux1}b). The functions $g_L, g_R$ are some polynomials that are chosen in the FR technique by linear stability analysis and $F_\eph$ are some numerical flux functions. The coefficients in the discontinuous flux $F_h^\delta$ provide approximations to the time average flux,
\[
F_j^e \approx \frac{1}{\Delta t} \int_{t_n}^{t_{n+1}} f( u( \xi_j, t) ) \ud t
\]
and are computed by an approximate Lax-Wendroff procedure that uses finite differencing in time. The reader should consult~\cite{babbar2022} for more details on these aspects of the scheme.

The element mean value is given by
\[
\bar u_e = \sum_{j=0}^N u_j^e w_j
\]
where $w_j$ are the weights associated to the solution points. Then it is easy to show that the scheme is conservative in the sense that
\begin{equation}\label{eq:fravgup}
\bar u_e^{n+1} = \bar u_e^n - \frac{\Delta t}{\Delta x_e} (F_\eph - F_\emh)
\end{equation}
The admissibility preserving property, also known as convex set preservation property since $\Uad$ is convex, of the conservation law can be written as
\begin{equation} \label{eq:conv.pres.con.law}
\bw(\cdot, t_0) \in \Uad \qquad \implies \qquad \bw(\cdot, t) \in \Uad, \qquad t > t_0
\end{equation}
and thus we define an admissibility preserving scheme to be
\begin{definition}\label{defn:adm.pres}
The flux reconstruction scheme is said to be admissibility preserving if
\[
(\bw_j^e)^n \in \Uad \quad  \forall e,j \qquad \implies \qquad (\bw_j^e)^{n+1} \in \Uad \quad  \forall e,j
\]
where $\Uad$ is the admissibility set of the conservation law.
\end{definition}
To obtain an admissibility preserving scheme, we exploit the weaker admissibility preservation in means property defined as
\begin{definition}\label{defn:mean.pres}
The flux reconstruction scheme is said to be admissibility preserving in the means if
\[
(\bw_j^e)^n \in \Uad \quad  \forall e,j \qquad \implies \qquad \au_e^{n+1} \in \Uad \quad \forall e
\]
where $\Uad$ is the admissibility set of the conservation law.
\end{definition}
The focus of this work is to obtain the admissibility preservation in means property for the Lax-Wendroff Flux Reconstruction scheme. Once the scheme is admissibility preserving in means, the scaling limiter of~\cite{zhang2010c} can be used to obtain an admissibility preserving scheme.
\section{On controlling oscillations}\label{sec:controlling.oscillations}
High order methods for hyperbolic problems necessarily produce Gibbs oscillations at discontinuities as shown by Godunov. The cure is to make the schemes to be non-linear even in the case of linear equations. For one dimensional problems, total variation diminishing approach provides a framework to construct non-oscillatory schemes. This is achieved by incorporating some non-linear limiting strategy into the scheme which locally reduces the order of the scheme when a discontinuity is detected. In discontinuous Galerkin type methods, the limiting is performed by modifying the solution in each element so as to ensure a TVD property for the element means. We first recall this strategy following Cockburn and Shu~~\cite{Cockburn1991, Cockburn1989a}.

\subsection{TVD limiter} \label{sec:tvd}
The limiters developed in the context of RKDG schemes~\cite{Cockburn1991,Cockburn1989a} can be adopted in the framework of LWFR schemes. The limiter is applied in a post-processing step after the solution is updated to the new time level. The limiter is thus applied only once for each time step unlike in RKDG scheme where it has to be applied after each RK stage update. Let $u_h(x)$ denote the solution at time $t_{n+1}$ obtained from the LWFR scheme. In element $\Omega_e$, let the average solution be $\au_e$; define the backward and forward differences of the solution and element means by
\[
\Delta^- u_e = \au_e - u_h(x_\emh^+), \qquad \Delta^+ u_e = u_h(x_\eph^-) - \au_e
\]
\[
\Delta^- \au_e = \au_e - \au_{e-1}, \qquad \Delta^+ \au_e = \au_{e+1} - \au_e
\]
We limit the solution by comparing its variation within the element with the difference of the neighbouring element means through a limiter function,
\[
\Delta^- u_e^{m} = \minmod(\Delta^- u_e, \Delta^- \au_e, \Delta^+ \au_e), \qquad
\Delta^+ u_e^{m} = \minmod(\Delta^+ u_e, \Delta^- \au_e, \Delta^+ \au_e)
\]
where
\[
\minmod(a,b,c) = \begin{cases}
s \min(|a|, |b|, |c|), & \textrm{if } s = \sign(a) = \sign(b) = \sign(c) \\
0, & \textrm{otherwise}
\end{cases}
\]
If $\Delta^- u_e^m \ne \Delta^- u_e$ or $\Delta^+ u_e^m \ne \Delta^+ u_e$, then the solution is deemed to be locally oscillatory and we modify the solution inside the element by replacing it as a linear polynomial with a limited slope, which is taken to be the average limited slope. The limited solution polynomial in element $\Omega_e$ is given by
\[
u_h|_{\Omega_e} = \au_e + \frac{\Delta^- u_e^m + \Delta^+ u_e^m}{2} (2 \xi - 1), \qquad \xi \in [0,1]
\]
This limiter is known to clip smooth extrema since it cannot distinguish them from jump discontinuities. A small modification based on the idea of TVB limiters~\cite{Cockburn1991} can be used to relax the amount of limiting that is performed which leads to improved resolution of smooth extrema. The $\minmod$ function is replaced by
\[
\widetilde{\minmod}(a,b,c) = \begin{cases}
a, & |a| \le M \Delta x^2 \\
\minmod(a,b,c), & \textrm{otherwise}
\end{cases}
\]
where the parameter $M$ has to be chosen by the user, which is an estimate of the second derivative of the solution at smooth extrema. In the case of systems of equations, the limiter is applied to the characteristic variables, which is known to yield better control on the spurious numerical oscillations~\cite{Cockburn1989}. Clearly, the performance of this limiter depends on the proper choice of the parameter $M$ which is problem dependent.
\subsection{Blending scheme}\label{sec:blending.scheme}
The TVD-type limiters used in DG methods lose a lot of information when the limiter is active, since the polynomial solution of degree $N$ is replaced either by a solution of degree 1 or a constant solution if a strong discontinuity is detected in an element. This is especially problematic near smooth extrema which may be wrongly detected as a discontinuity. It would be desirable to use more information inside each element while applying some limiting process. Let us write the LWFR update equation~\eqref{eq:uplwfr} as
\[
\vu^{H,n+1}_e = \vu^n_e - \frac{\Delta t}{\Delta x_e} \vR^H_e
\]
where $\vu_e$ is the vector of nodal values in the element. Suppose we also have a lower order and non-oscillatory scheme available to us in the form
\[
\vu^{L,n+1}_e = \vu^n_e - \frac{\Delta t}{\Delta x_e} \vR^L_e
\]
Then a blended scheme is given by
\begin{equation}
\vu^{n+1}_e = (1 - \alpha_e) \vu^{H,n+1}_e + \alpha_e \vu^{L,n+1}_e = \vu^n_e - \frac{\Delta t}{\Delta x_e}[(1-\alpha_e) \vR^H_e + \alpha_e \vR^L_e] \label{eq:blended.scheme}
\end{equation}
where $\alpha_e \in [0,1]$ must be chosen based on some local smoothness indicator. If $\alpha_e = 0$ then we obtain the high order LWFR scheme, while if $\alpha_e=1$ then the scheme becomes the low order scheme that is less oscillatory. In subsequent sections, we explain the details of the lower order scheme and the design of smoothness indicators. The lower order scheme will either be a first order finite volume scheme or a high resolution scheme based on MUSCL-Hancock idea. In either case, the common structure of the low order scheme can be explained as follows.

\begin{figure}
\begin{center}
\includegraphics[width=0.7\textwidth]{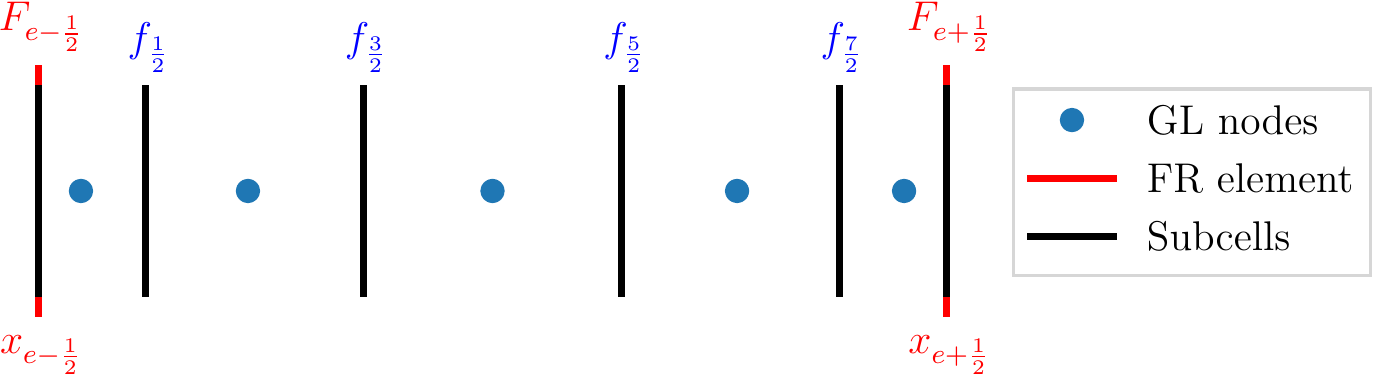}
\caption{Subcells used by lower order scheme for degree $N=4$.}
\label{fig:subcells}
\end{center}
\end{figure}

Let us subdivide each element $\Omega_e$ into $N+1$ subcells associated to the solution points $\{x^e_j, j=0, 1, \ldots, N\}$ of the LWFR scheme. Thus, we will have $N+2$ subfaces denoted by $\{x^e_\jph, j= -1, 0, \ldots, N\}$ with $x^e_{-\half} = x_\emh$ and $x^e_{\Nph} = x_\eph$. For maintaining a conservative scheme, the $j^\text{th}$ subcell is chosen so that
\begin{equation}\label{eq:subcell.defn}
x_\jph^e - x_\jmh^e = w_j \Delta x_e, \qquad 0 \le j \le N
\end{equation}
where $w_j$ is the $j^\text{th}$ quadrature weight associated with the solution points. Figure~\ref{fig:subcells} gives an illustration of the subcells for degree $N=4$ case. The low order scheme is obtained by updating the solution in each of the subcells by a finite volume scheme,
\begin{equation}
\label{eq:low.order.update}
\begin{split}
(u_0^e)^{L,n+1} &= (u_0^e)^n - \frac{\Delta t}{w_0 \Delta x_e}[f_\half^e - F_\emh] \\
(u_j^e)^{L,n+1} &= (u_j^e)^n - \frac{\Delta t}{w_j \Delta x_e}[f_{j+\frac 12}^e - f_{j-\frac 12}^e], \qquad 1 \le j \le N-1 \\
(u_N^e)^{L,n+1} &= (u_N^e)^n - \frac{\Delta t}{w_N \Delta x_e}[F_\eph - f_\Nmh^e]
\end{split}
\end{equation}
The inter-element fluxes $F_\eph$ used in the low order scheme are same as those used in the high order LWFR scheme in equation~\eqref{eq:frcontflux}. Usually, Rusanov's flux~\cite{Rusanov1962} will be used for the inter-element fluxes and in the lower order scheme. The element mean value obtained by the low order scheme satisfies
\begin{equation} \label{eq:low.order.cell.avg.update}
\bar u_e^{L,n+1} = \sum_{j=0}^N (\bw_j^e)^{L,n+1} w_j =\bar u_e^n - \frac{\Delta t}{\Delta x_e} (F_\eph - F_\emh)
\end{equation}
which is identical to the update equation by the LWFR scheme given in equation~\eqref{eq:fravgup}.  The element mean in the blended scheme evolves according to
\begin{eqnarray*}
\bar u_e^{n+1} &=& (1-\alpha_e) (\bar u_e)^{H,n+1} + \alpha_e (\bar u_e)^{L,n+1} \\
&=& (1 - \alpha_e)  \left[ \bar u_e^n - \frac{\Delta t}{\Delta x_e} ( F_\eph - F_\emh ) \right] +
         \alpha_e  \left[ \bar u_e^n - \frac{\Delta t}{\Delta x_e} ( F_\eph - F_\emh ) \right] \\
&=&  \bar u_e^n - \frac{\Delta t}{\Delta x_e} ( F_\eph - F_\emh )
\end{eqnarray*}
and hence the blended scheme is also conservative; all three schemes, i.e., lower order, LWFR and the blended scheme, predict the same mean value.

The inter-element flux $F_\eph$ is used both in the low and high order schemes. To achieve high order accuracy in smooth regions, this flux needs to be high order accurate, however it may produce numerical oscillations near discontinuities when used in the low order scheme. A natural choice to balance accuracy and oscillations is to take
\begin{equation}\label{eq:Fblend}
F_\eph = (1-\alpha_\eph) F_\eph^\text{LW} + \alpha_\eph f_\eph, \qquad \alpha_\eph \in [0,1]
\end{equation}
where $F_\eph^\text{LW}$ is the high order inter-element time-averaged numerical flux of the LWFR scheme~\eqref{eq:frcontflux} and $f_\eph$ is a lower order flux at the face $x_\eph$ shared between FR elements and subcells~(\ref{eq:fo.at.face}, \ref{eq:mh.at.face}). The blending coefficient $\alpha_\eph$ will be based on a local smoothness indicator which will bias the flux towards the lower order flux $f_\eph$ near regions of lower solution smoothness. However, to enforce admissibility in means (Definition~\ref{defn:mean.pres}), the flux has to be further corrected, as explained in Section~\ref{sec:admissibility.preservation}.
\begin{remark}\label{rmk:why.same.flux}
It is essential to use the same inter-element flux in both the low and high order schemes in order to have conservation. Suppose we use numerical fluxes $F_\eph^L, F_\eph^H$ in the low and high order schemes, respectively; then the element mean in the blended scheme will become
\[
\bar u_e^{n+1} = \bar u_e^n - \frac{\Delta t}{\Delta x_e} [ ((1-\alpha_e)F_\eph^H + \alpha_e F_\eph^L) - ((1-\alpha_e)F_\emh^H + \alpha_e F_\emh^L) ]
\]
For conservation the flux leaving element $\Omega_e$ through $x_\eph$ must enter the neighbouring element $\Omega_{e+1}$, i.e.,
\[
(1-\alpha_e)F_\eph^H + \alpha_e F_\eph^L = (1-\alpha_{e+1})F_\eph^H + \alpha_{e+1} F_\eph^L
\]
i.e., $(\alpha_e - \alpha_{e+1}) F_\eph^L = (\alpha_e - \alpha_{e+1}) F_\eph^H$ which must hold for all values of $\alpha_e, \alpha_{e+1}$ and hence we need $F_\eph^L = F_\eph^H$.
\end{remark}
\subsection{Smoothness indicator}\label{sec:smooth.ind}
The numerical approximation of the PDE solution is in the form of piecewise polynomials of degree $N$. The polynomial can be written in terms of an orthogonal basis like Legendre  polynomials. The smoothness of the solution can be assessed by analyzing the decay of the coefficients of the orthogonal expansion, a technique originally proposed by Persson and Peraire~\cite{Persson2006} and subsequently refined by Klöckner et al.~\cite{klockner2011} and Henemann et al.~\cite{henemann2021}. For a scalar problem, the solution $u$ itself can be used to design a smoothness indicator. For a system of PDE, we can use any one or all components of the solution vector. Alternatively, some derived quantity that can indicate the smoothness of all solution components can be chosen. For the Euler equations, a good choice seems to be the product of density and pressure~\cite{henemann2021}.

Let $q = q(u)$ be the quantity used to measure the solution smoothness. We first project this onto Legendre polynomials,
\[
q_h(\xi) = \sum_{j=0}^N \hat q_j L_j(2\xi-1), \quad \xi\in [0,1], \qquad \hat q_j = \int_0^1 q(u_h(\xi)) L_j(2\xi-1) \ud\xi
\]
The Legendre coefficients $\hat q_j$ are computed using the quadrature induced by the solution points,
\[
\hat q_j = \sum_{q=0}^N q(u_q^e) L_j(2\xi_q-1) w_q
\]
Then the energy contained in highest modes relative to the total energy of the polynomial is computed as follows,
\[
\en = \max \left(
\frac{\hat q_{N-1}^2}{\sum_{j=0}^{N-1}\hat q_j^2},
\frac{\hat q_N^2}{\sum_{j=0}^N \hat q_j^2}
\right)
\]
The $N^\text{th}$ Legendre coefficient $\hat q_N$ of a function which is in the Sobolev space $H^2$ decays as $O(1/N^2)$ (see Chapter 5, Section 5.4.2 of~\cite{Canuto2007}). We consider smooth functions to be those whose Legendre coefficients $\hat q_N$ decay at rate proportional to $1/N^2$ so that their squares decay proportional to $1/N^4$~\cite{Persson2006}. Thus, the following threshold for smoothness is proposed in~\cite{henemann2021}
\[
\thresh(N) = a \cdot 10^{-c(N+1)^\frac{1}{4}}
\]
where parameters $a=\half$ and $c=1.8$ are obtained through numerical experiments. To convert the highest mode energy indicator $\en$ and threshold value $\thresh$ into a value in $[0,1]$, the logistic function (Figure~\ref{fig:alpha.func}) is used
\[
\tilde{\alpha}(\en) = \frac{1}{1+\exp\left( -\frac{s}{\thresh} (\en - \thresh) \right)}
\]
\begin{figure}
\begin{center}
\includegraphics[width=0.5\textwidth]{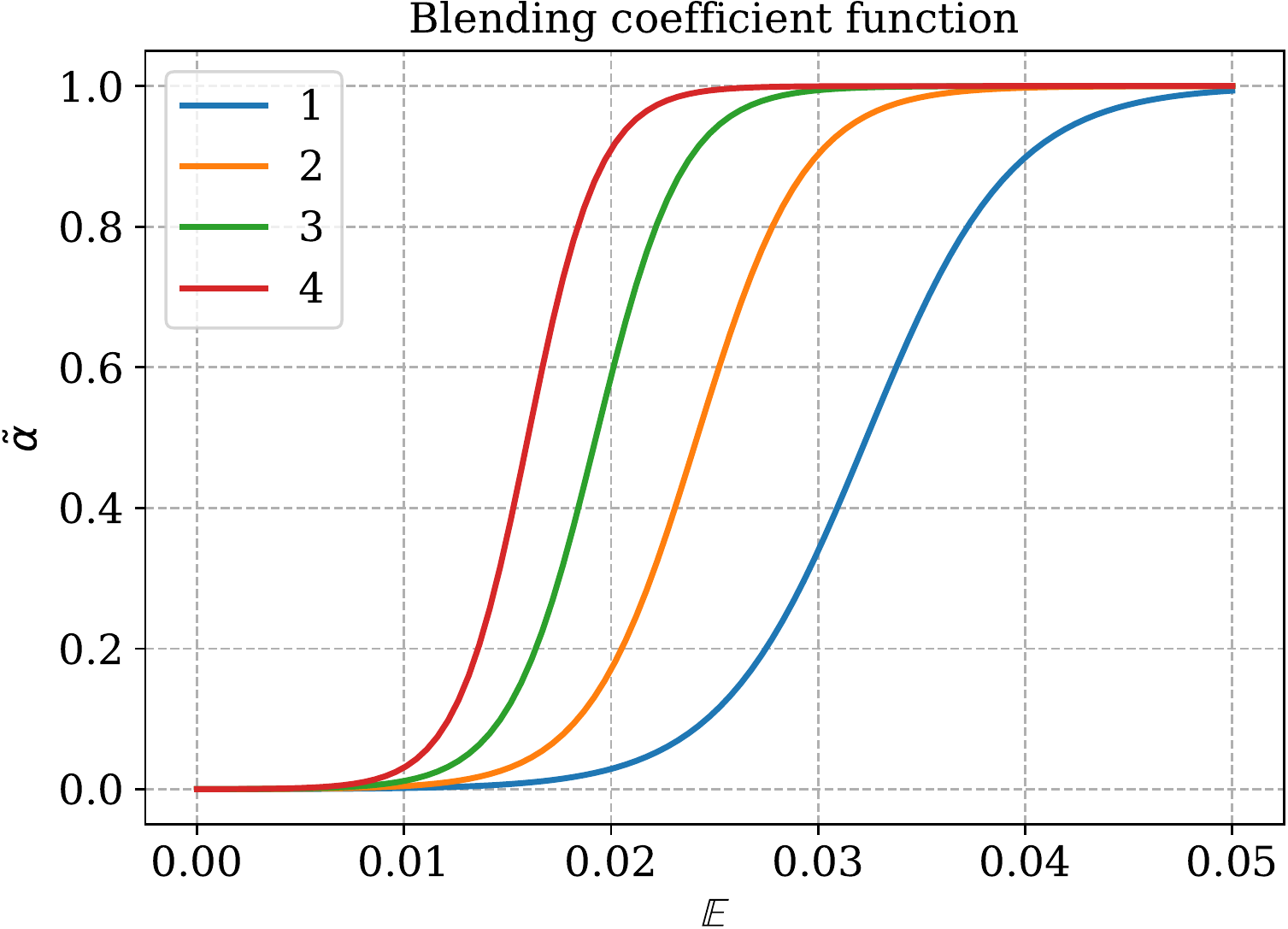}
\caption{Logistic function used to map energy to a smoothness coefficient $\alpha \in [0,1]$ shown for various solution polynomial degrees $N$.}
\label{fig:alpha.func}
\end{center}
\end{figure}\nolinebreak 
The sharpness factor $s$ was chosen to be $s = 9.21024$ so that blending coefficient equals $\alpha = 0.0001$ when highest energy indicator $\en = 0$. In regions where $\tilde\alpha = 0$ or $\tilde\alpha = 1$, computational cost can be saved by performing \correction{only the high or low order schemes}, respectively. Thus, the values of $\alpha$ are clipped as
\[
\alpha_e :=
\begin{cases}
0,\quad & \text{if } \tilde\alpha < \alpha_\text{min} \\
\tilde\alpha, & \text{if } \alpha_\text{min} \le \tilde\alpha \le 1 - \alpha_\text{min} \\
1,      & \text{if } 1 - \alpha_\text{min}< \tilde\alpha
\end{cases}
\]
with $\alpha_\text{min} = 0.001$. In ~\cite{henemann2021}, the maximum value of $\alpha$ was clipped to $\alpha_\text{max}=0.5$, but we use $\alpha_\text{max} = 1$ for the LWFR scheme. There were no significant improvements observed by decreasing $\alpha_\text{max}$ in any of the tests; in some tests like Shu-Osher (Section~\ref{sec:shuosher}), we observed a large number of oscillations when $\alpha_\text{max}=0.5$ was used. Finally, since shocks can spread to the neighbouring cells, smoothening of $\alpha$ is performed as
\begin{equation}
\alpha_e^\text{final} = \max_{\correction{e'} \in \mathcal{E}_e}\left\{ \alpha_e, \half \alpha_\correction{e'} \right\} \label{eq:smoothing}
\end{equation}
where $\mathcal{E}_e$ denotes all elements sharing a face with $e$.
\subsection{First order blending}\label{sec:fo}
The lower order scheme is taken to be a first order finite volume scheme, for which the subcell fluxes in \eqref{eq:low.order.update} are given by
\[
f_\jph^e = f(u_j^e, u_{j+1}^e)
\]
At the interfaces that are shared with FR elements, we define the lower order flux used in computing inter-element flux (Section~\ref{sec:admissibility.preservation}) as
\begin{equation}
f_\eph = f(u_{N}^e,u_0^{e+1}) \label{eq:fo.at.face}
\end{equation}
In this work, the numerical flux $f(\cdot, \cdot)$ is taken to be Rusanov's flux~\cite{Rusanov1962}, which is the same flux used by the \correction{high order} scheme at the element interfaces.
\section{Higher order blending}\label{sec:mh}
The MUSCL-Hancock scheme is a single-stage and second order accurate scheme, originally introduced in~\cite{vanleer1984}, and proven to be robust under appropriate slope restrictions~\cite{Berthon2006}. We can expect better accuracy by blending the LWFR scheme with the MUSCL-Hancock scheme. Following the slope correction procedure of Berthon~\cite{Berthon2006}, the MUSCL-Hancock scheme can mimic the \correction{admissibility set} preservation of the solutions of conservation laws~\eqref{eq:conv.pres.con.law}. The extension of Berthon's work to non-cell centered grids~\eqref{eq:non.cell.centred.defn} which arise in the blending scheme is given in Theorem~\ref{thm:muscl.admissibility.theorem} whose proof is given in Appendix~\ref{sec:muscl.admissibility.proof}. In this section, we give algorithmic details of the 1-D procedure and details of the 2-D procedure can be found in Appendix~\ref{sec:2d.mh}.

Essentially, the MUSCL-Hancock scheme provides a high order estimate of the subcell fluxes $f_\jph^e$ used in the low order scheme \eqref{eq:low.order.update} and we now explain the procedure for estimating these fluxes. To simplify the notation, let us suppress the element index $e$ and set
\[
u_{-2} = u_{N-1}^{e-1}, \qquad u_{-1} = u_N^{e-1}, \qquad \{ u_j = u_j^e, \qquad 0 \le j \le N\}, \qquad u_{N+1} = u_0^{e+1}, \qquad u_{N+2} = u_1^{e+1}
\]
Using the mid-point rule in time to integrate the conservation law~\eqref{eq:con.law} over the space-time element $[x_{\jmh},x_{\jph}] \times [t^n,t^{n+1}]$, we get
\begin{equation}
u_j^{n+1} = u_j^n - \frac{\Delta t}{\Delta x_j} (f_{\jph}^{\nph} - f_{\jmh}^{\nph}) \label{eq:riemann.problem}
\end{equation}
where
\[
f_{\jph}^{\nph}=f(\bw_{j-1}^{\nph,+},u_{j}^{\nph,-})
\]
is obtained from a numerical flux function, usually Rusanov's flux~\cite{Rusanov1962}. The $\bw_j^{\nph,\pm}$ denote the approximations of solutions in subcell $j$ at right, left faces respectively, evolved to time level $\nph$. Aiming to first approximate the solution at $t^n$ on the faces, we create a linear approximation of the solution in each subcell as
\begin{equation}\label{eq:delta.defn}
r_j^n(x) = u_j^n + (x-x_j) \delta_j\correction{,} \qquad
\slope_j = \text{minmod}\left(\beta\Delta_+ u_j,\Delta_c u_j,\beta\Delta_- u_j\right)
\end{equation}
where, for $h_1 = x_{j} - x_{j-1}$, $h_2 = x_{j+1} - x_j$,
\[
\Delta_+ u_j = \frac{u_{j+1}^n-u_j^n}{h_2}, \quad  \Delta_- u_j =  \frac{u_j^n-u_{j-1}^n}{h_1}, \quad \Delta_c u_j = -\frac{h_2}{h_1(h_1+h_2)}u^n_{j-1} + \frac{h_2 - h_1}{h_1 h_2}u^n_j + \frac{h_1}{h_2(h_1 + h_2)} u^n_{j+1}
\]
The $\Delta_{\pm} \bw_j$ are forward and backward approximations of slope respectively, and $\Delta_c \bw_j$ is the second order approximation of the slope. The value $\beta$ is chosen to lie between $1$ and $2$; for $\beta=1$, we reduce to the minmod limiter and $\beta = 2$ corresponds to the limiter of van Leer~\cite{vanleer1977}. A higher value of $\beta$ tips the slope closer to the second order approximation, gaining accuracy but also increasing the risk of spurious oscillations. For all results in this work, the choice of $\beta = 2 - \alpha_e$ is made. Thus, $\beta$ will be close to 2 in regions where smoothness indicator only detects mild irregularities in the solution, while it will be near $1$ in regions with strong discontinuities. With the linear reconstructions, we can define
\begin{equation}
\label{eq:reconstruction}
\bw_j^{n,-} = r^n_j(x_{\jmh}) = u_j^n + \slope_j(x_{\jmh}-x_j),\qquad \bw_j^{n,+} = r^n_j(x_{\jph}) = u_j^n + \slope_j(x_{\jph}-x_j)
\end{equation}
Using the conservation law, we approximate the temporal derivatives as
\[
\partial_tu^n_{j} := -\frac{f(u^{n,+}_j)-f(u^{n,-}_j)}{x_{\jph}-x_{\jmh}}
\]
and finally use Taylor's expansion to evolve the face values in time as
\begin{equation}
\label{eq:evolution}
\bw_j^{\nph,-} = \bw_j^{n,-} + \frac{\Delta t}{2}\partial_tu^n_j, \qquad
\bw_j^{\nph,+} = \bw_j^{n,+} + \frac{\Delta t}{2}\partial_tu^n_j
\end{equation}
At the interfaces shared with the FR elements, the lower order flux used in computing inter-element flux (Section~\ref{sec:admissibility.preservation}) is given by $f_\eph = f_\Nph^\nph$; the dependence on neighbouring states can be made explicit as
\begin{equation}
f_{\eph} = f(u_{N-1}^e,u_{N}^e,u_0^{e+1},u_1^{e+1}) \label{eq:mh.at.face}
\end{equation}

For admissibility of the lower order method, we rely on the following generalization of Berthon~\cite{Berthon2006}, proved in Appendix~\ref{sec:muscl.admissibility.proof}.

\begin{theorem}
\label{thm:muscl.admissibility.theorem}
Consider a conservation law of the form~\eqref{eq:con.law} which preserves the admissibility set $\Uad$~\eqref{eq:conv.pres.con.law}. Let $\left\{\bw_j^n\right\}_{j}$ be the approximate solution at time level $n$ and assume that $\bw_j^n \in \Uad$ for all $j$. Consider conservative reconstructions
\[
\bw_j^{n,+} = \bw_j^n + (x_{j+\frac{1}{2}} - x_j)\slope_j,\qquad \bw_j^{n,-} = \bw_j^n + (x_{j-\frac{1}{2}} - x_j) \slope_j
\]
Define $\bw_j^{*,\pm}$ by
\begin{equation} \label{eq:us.defn.main}
\mum \bw_j^{n,-}+\bw_j^{*,\pm}+\mup\bw_j^{n,+} = 2 \bw_j^{n,\pm}
\end{equation}
where
\[
\mum = \frac{x_{j+\frac{1}{2}}-x_j}{x_{j+\frac{1}{2}}-x_{j-\frac{1}{2}}},\qquad \mup = \frac{x_j-x_{j-\frac{1}{2}}}{x_{j+\frac{1}{2}}-x_{j-\frac{1}{2}}}
\]
Assume that the slope $\slope_j$ is chosen so that
\begin{equation}
\bw_j^{*,\pm} \in \Uad \label{eq:ustar.in.uad}
\end{equation}
Then, under appropriate time step restrictions~(\ref{eq:new.cfl1},\ref{eq:new.cfl2},\ref{eq:new.cfl3.conservative}), the updated solution $\bw_j^{n+1}$ defined by the MUSCL-Hancock procedure~\eqref{eq:riemann.problem} is in $\Uad$.
\end{theorem}
\subsection*{Slope limiting in practice}
A problem-independent procedure for slope limiting to ensure admissibility preservation is proposed, in contrast to the original procedure for Euler's equations in~\cite{Berthon2006} that was extended to the 10-moment problem in~\cite{meena2017}. For the MUSCL-Hancock scheme to be admissibility preserving, the slope $\slope_j$ given by the minmod limiter~\eqref{eq:delta.defn} has to be further limited so that $\bw_j^{\ast,\pm} = \bw_j^n + 2 (x_{\jpmh} - x_j) \slope_j \in \Uad$~\eqref{eq:us.defn.main}. We explain the procedure with a \texttt{for} loop over the admissibility constraints $\{p_k, k=1,\ldots,K\}$.

\begin{minipage}{\textwidth}  
\vspace{8pt}
\hrule
\vspace{1.5pt}
\begin{algorithmic}
\State $\slope_j \gets \text{minmod}\left(\beta\Delta_+ u_j,\Delta_c u_j,\beta\Delta_- u_j\right)$
\State $\bw_j^{\ast,\pm} \gets \bw_j^n + 2 (x_{\jpmh} - x_j) \slope_j$
\For{$k$=1:$K$}
\State $\epsilon_k = \frac{1}{10}p_k(\bw_j^n)$
\State $\theta_\pm \gets \min \left\{ \left| \frac{\epsilon_k - p_k \left( \bw_j^n \right)}{p_k \left( \bw_j^{\ast, \pm} \right) - p_k \left( \bw_j^n \right)} \right|, 1 \right\}$
\State $\theta_k \gets \min \{\theta_+, \theta_- \}$
\State $\slope_j \gets \theta_k \slope_j$
\State $\bw_j^{\ast,\pm} \gets \bw_j^n + 2 (x_{\jpmh} - x_j) \slope_j$
\EndFor
\end{algorithmic}
\hrule
\vspace{5pt}
\end{minipage}

At the $k^\text{th}$ iteration, by concavity of the admissibility constraint $p_k$, the $\bw_j^{\ast,\pm}$ computed with the corrected slope $\slope_j$ will satisfy
\begin{equation}\label{eq:pk.slope.correction}
p_k(\bw_j^{\ast,\pm}) = p_k(\theta_k (\bw_j^{\ast,\pm})^\text{prev} + (1-\theta_k)\bw_j^n)  \ge \theta_k p_k((\bw_j^{\ast,\pm})^\text{prev}) +(1-\theta_k) p_k(\bw_j^n)  \ge \epsilon_k
\end{equation}
so that the $k^\text{th}$ admissibility constraint is satisfied; here  $(\bw_j^{\ast,\pm})^\text{prev}$ denotes $\bw_j^{\ast,\pm}$ before the $k^\text{th}$ correction. The choice of $\epsilon_k = \frac{1}{10}p_k(\bw_j^n)$ was made following~\cite{ramirez2021} to allow only a certain deviation below the \textit{safe solution}, imposing a stricter requirement than positivity. Note that this limiting is performed on the slope used for reconstruction in the MUSCL-Hancock scheme, and not on the updated solution.  The previous admissibility constraints $p_l$ for $l<k$ will also continue to be satisfied by using induction argument and concavity of the constraints,
\[
p_l(\bw_j^{\ast,\pm}) \ge \theta_k p_l((\bw_j^{\ast,\pm})^\text{prev}) + (1-\theta_k) p_l(\bw_j^n)  \ge \theta_k \epsilon_l + (1-\theta_k) \epsilon_l = \epsilon_l
\]
The slope $\slope_j$ obtained at the end of $K$ iterations satisfies all admissibility constraints ensuring $\bw_j^{\ast,\pm} \in \Uad$.

\section{Flux limiter for admissibility preservation} \label{sec:admissibility.preservation}
The first step in obtaining an admissibility preserving blending scheme is to ensure that the lower order scheme preserves the admissibility set $\Uad$. This is always true if all the fluxes in the lower order method are computed with a finite volume method that is proven to be admissibility preserving. But the LWFR scheme uses a time average numerical flux and maintaining conservation requires that we use the same numerical flux at the element interfaces for both lower and higher order schemes (see Remark~\ref{rmk:why.same.flux}). To maintain accuracy and admissibility, we have to carefully choose a blended numerical flux $F_\eph$ as in~\eqref{eq:Fblend} but this choice may not ensure admissibility and further limitation is required. Our proposed procedure for choosing the blended numerical flux will give us an admissibility preserving lower order scheme. After this step, there are two possibilities for obtaining admissibility of the blending scheme. We could follow the procedure of~\cite{ramirez2021} to \textit{a posteriori} modify the blending coefficient $\alpha$ to obtain admissibility relying directly on the admissibility of the lower order scheme. The other option which we take in this work is to note that, as a result of using the same numerical flux in both high and low order schemes, element means of both schemes are the same (Theorem~\ref{thm:lwfr.admissibility}). A consequence of this is that our scheme now preserves admissibility of element means and thus we can use the scaling limiter of~\cite{Zhang2010b}. The latter approach of correcting element means to obtain a positivity preserving Lax-Wendroff scheme has been used in~\cite{moe2017}, where the numerical flux is corrected to directly make element means admissible. In comparison to~\cite{moe2017}, our procedure for ensuring admissibility of element means requires less storage and loops. 

The theoretical basis for flux limiting can be summarised in the following Theorem~\ref{thm:lwfr.admissibility}.

\begin{theorem}\label{thm:lwfr.admissibility}
Consider the LWFR blending scheme~\eqref{eq:blended.scheme} where low and high order schemes use the same numerical flux $F_\eph$ at every element interface. Then the following can be said about admissibility preserving in means property (Definition~\ref{defn:mean.pres}) of the scheme:
\begin{enumerate}
\item element means of both low and high order schemes are same and thus the blended scheme~\eqref{eq:blended.scheme} is admissibility preserving in means if and only if the lower order scheme is admissibility preserving in means;
\item if the finite volume method using the lower order flux $f_\eph$ as the interface flux is admissibility preserving, such as the first-order finite volume method or the MUSCL-Hancock scheme with CFL restrictions and slope correction from Theorem~\ref{thm:muscl.admissibility.theorem}, and the blended numerical flux $F_\eph$ is chosen to preserve the admissibility of lower-order updates at solution points adjacent to the interfaces, then the blending scheme~\eqref{eq:blended.scheme} will preserve admissibility in means.
\end{enumerate}
\end{theorem}
\begin{proof}
By~(\ref{eq:fravgup}, \ref{eq:low.order.cell.avg.update}), element means are the same for both low and high order schemes. Thus, admissibility in means of one implies the same for other, proving the first claim. For the second claim, note that our assumptions imply $(u_j^e)^{L,n+1}$ given by~\eqref{eq:low.order.update} is in $\Uad$ for $0\le j \le N$ implying admissibility in means property of the lower order scheme by~\eqref{eq:low.order.cell.avg.update} and thus admissibility in means for the blended scheme.
\end{proof}

We now explain the procedure of ensuring that the update obtained by the lower order scheme will be admissible. The lower order scheme is computed with a first order finite volume method or MUSCL-Hancock with slope correction from Theorem~\ref{thm:muscl.admissibility.theorem} so that admissibility is already ensured for inner solution points; i.e., we already have
\[
(\bw^e_{j})^{L, n + 1} \in \Uad, \quad  1 \le j \le N-1
\]
The remaining admissibility constraints for the first ($j=0$) and last solution points ($j=N$) will be satisfied by appropriately choosing the inter-element flux $F_{\eph}$. The first step is to choose a candidate for $F_\eph$ which is heuristically expected to give reasonable control on spurious oscillations, i.e.,
\[
F_{\eph} = (1 - \alpha_{\eph}) F^\text{LW}_{\eph} + \alpha_{\eph} f_{\eph}, \quad \alpha_{\eph} = \frac{\alpha_e + \alpha_{e+1}}{2}
\]
where $f_{\eph}$ is the lower order flux at the face $\eph$ shared between FR elements and subcells~(\ref{eq:fo.at.face}, \ref{eq:mh.at.face}), and $\alpha_e$ is the blending coefficient~\eqref{eq:blended.scheme} based on element-wise smoothness indicator (Section~\ref{sec:smooth.ind}).

The next step is to correct $F_{\eph}$ to enforce the admissibility constraints. The guiding principle of our approach is to perform the correction within the face loops, minimizing storage requirements and additional memory reads. The lower order updates in subcells neighbouring the $\eph$ face with the candidate flux are
\begin{equation}
\label{eq:low.order.tilde.update}
\begin{split}
\tilde{\bw}_{0}^{n+1} & = (\bw^{e+1}_0)^n - \frac{\Delta t}{w_0 \Delta x_{e+1}} ( f^{e+1}_{\half} - F_{\eph} ) \\
\tilde{\bw}_{N}^{n+1} &= (\bw^e_N)^n - \frac{\Delta t}{w_N \Delta x_{e}} (F_{\eph} - f^e_{\Nmh})
\end{split}
\end{equation}
To correct the interface flux, we will again use the fact that first order finite volume method and MUSCL-Hancock with slope correction from Theorem~\ref{thm:muscl.admissibility.theorem} preserve admissibility, i.e.,
\begin{align*}
\utilow_0 & = (\bw^{e+1}_{0})^n - \frac{\Delta t}{w_0 \Delta x_{e+1}} ( f^{e+1}_{\half} - f_{\eph} ) \in \Uad \\
\utilow_N &= (\bw^e_{N})^n - \frac{\Delta t}{w_N \Delta x_{e}} (f_{\eph} - f^e_{\Nmh}) \in \Uad
\end{align*}

Let $\{p_k, 1 \le 1 \le K\}$ be the admissibility constraints~\eqref{eq:uad.form} of the conservation law. The numerical flux is corrected by iterating over the admissibility constraints as follows

\begin{minipage}{\textwidth}  
\vspace{8pt}
\hrule
\vspace{1.5pt}
\begin{algorithmic}
\State $F_{\eph} \gets (1 - \alpha_{\eph}) F^\text{LW}_{\eph} + \alpha_{\eph} f_{\eph}$
\For{$k$=1:$K$}
\State $\epsilon_0, \epsilon_N \gets \frac{1}{10}p_k(\utilow_0), \frac{1}{10}p_k(\utilow_N)$
\State $\theta \gets \min \left(\min_{j=0,N} \left| \frac{\epsilon_j - p_k(\tilde{\bw}_j^{n+1})}{p_k(\tilde{\bw}_j^{\text{low},n+1}) -  p_k(\tilde{\bw}_j^{n+1})} \right |, 1 \right)$
\State $F_{\eph} \gets \theta F_{\eph} + (1-\theta) f_{\eph}$
\State $\tilde{\bw}_{0}^{n+1}  \gets (\bw^{e+1}_0)^n - \frac{\Delta t}{w_0 \Delta x_{e+1}} ( f^{e+1}_{\half} - F_\eph )$
\State $\tilde{\bw}_{N}^{n+1} \gets (\bw^e_N)^n - \frac{\Delta t}{w_N \Delta x_{e}} (F_\eph - f^e_{\Nmh})$
\EndFor
\end{algorithmic}
\hrule
\vspace{5pt}
\end{minipage}

By concavity of $p_k$, after the $k^\text{th}$ iteration, the updates computed using flux $F_\eph$ will satisfy
\begin{equation}
p_k(\tilde{\bw}_j^{n+1}) =  p_k(\theta (\tilde{\bw}_j^{n+1})^\text{prev} + (1-\theta)\utilow_j) \ge \theta p_k( (\tilde{\bw}_j^{n+1})^\text{prev}) + (1 - \theta)p_k(\utilow_j) \ge \epsilon_j, \qquad j=0,N
\end{equation}
satisfying the $k^\text{th}$ admissibility constraint;  here $(\tilde{\bw}_j^{n+1})^\text{prev}$ denotes $\tilde{\bw}_j^{n+1}$ before the $k^\text{th}$ correction and the choice of $\epsilon_j = \frac{1}{10}p_k(\utilow_j)$ is made following~\cite{ramirez2021}. After the $K$ iterations, all admissibility constraints will be satisfied and the resulting flux $F_{\eph}$ will be used as the interface flux keeping the lower order updates and thus the element means admissible. Thus, by Theorem~\ref{thm:lwfr.admissibility}, the choice of blended numerical flux gives us admissibility preservation in means. We now use the scaling limiter of~\cite{Zhang2010b} to obtain an admissibility preserving scheme as defined in Definition~\ref{defn:adm.pres}, an overview of the complete scheme can be found in Algorithm~\ref{alg:high.level.lw}. The above procedure is for 1-D conservation laws; the extension to 2-D is performed by breaking the update into convex combinations of 1-D updates and adding additional time step restrictions; the details are given in Appendix~\ref{sec:2d.admissibility}.

\section{Some implementation details} \label{sec:alg}
In Section~\ref{sec:admissibility.preservation}, the procedure for computing the blended numerical flux to achieve admissibility preservation in means for LWFR
(Definition~\ref{defn:mean.pres}) was presented. In this section, we present an overview of the complete LWFR blended scheme which employs the computed blended flux and the scaling limiter of~\cite{Zhang2010b} to obtain an admissibility preserving scheme (Definition~\ref{defn:adm.pres}) in Algorithm~\ref{alg:high.level.lw}.

The residual in~\eqref{eq:blended.scheme} is computed by performing an element loop and a face loop, incorporating blending within each of these loops. Within the element loop, we compute lower order fluxes on the subcell faces not shared by the FR elements. The fluxes for the faces shared by FR elements are computed within the face loop, and subsequently blended with the LW flux. This approach enables direct computation and use of each quantity, without the need for intermediate storage. However, to compute~\eqref{eq:low.order.tilde.update}, admissibility preservation requires storage of lower order fluxes $f^e_{\half}$ and $f^e_{\Nmh}$, which are computed during the element loop.

In Algorithm~\ref{alg:high.level.lw}, we give a high level overview of the LWFR with blending scheme. In practice, some operations could be reduced by computing only high or low order residuals in the cases where $\alpha_e = 0$ or $\alpha_e = 1$, but we did not include this optimization in Algorithm~\ref{alg:high.level.lw} to maintain simplicity in our explanation. The correction procedure of numerical flux for admissibility preservation (Section~\ref{sec:admissibility.preservation}) is performed within the interface iteration. The contribution of numerical flux to the residual is added in a different element loop to avoid race conditions in a multi-threaded loop; only one loop would be needed in a serial code.
After the solution update in Algorithm~\ref{alg:high.level.lw}, the blended flux will ensure that our purely low order update and the element means are admissible. However, the updates at solution points need not be admissible at this stage and must be corrected. The correction at solution points could now be performed as an \textit{a posteriori} modification of the blending coefficients~\cite{ramirez2021} or using the scaling limiter of~\cite{Zhang2010b}; we use the scaling limiter for all results in this work.

\begin{algorithm}
\caption{High-level overview of the Lax-Wendroff with blending scheme}
\label{alg:high.level.lw}
\begin{algorithmic} 
\While{$t < T$}
\State Compute $\{\alpha_e\}$ (Section~\ref{sec:smooth.ind})
\For{$e$ in \texttt{eachelement(mesh)}} \Comment{Assemble element residual}
\State Add LW element residual to rhs scaled with $1-\alpha_e$
\State Add FV subcell residual to rhs scaled with $\alpha_e$
\State Store $f^e_{\half}, f^e_{\Nmh}$~\eqref{eq:low.order.tilde.update}
\EndFor
\For{$\eph$ in \texttt{eachinterface(mesh)}} \Comment{Compute interface flux}
\State Compute $F_\eph^\text{LW}$, $f_\eph$ and blend them into $F_\eph$ (Section~\ref{sec:admissibility.preservation})
\EndFor
\For{$e$ in \texttt{eachelement(mesh)}} \Comment{Assemble face residual}
\State Add contribution of $F_{e \pm \half}$ to high, low order residual scaled with $1-\alpha_e, \alpha_e$ respectively
\EndFor
\State Update solution
\State Apply positivity correction at solution points using \cite{Zhang2010b} or \cite{ramirez2021}
\State $t \gets t + \Delta t$
\EndWhile
\end{algorithmic}
\end{algorithm}


\section{Numerical results} \label{sec:num.results}
We perform various tests to show the robustness and accuracy of the proposed blending scheme. The LWFR results are always obtained with D2 dissipation and EA flux~\cite{babbar2022} with Rusanov's numerical flux using Gauss-Legendre solutions point and Radau correction functions. All numerical simulations were run with the first order blending (Section~\ref{sec:fo}), MUSCL-Hancock blending (Section~\ref{sec:mh}) and TVB limiter with fine-tuned parameter $M$ plotted with legends FO, MH and TVB-M. We also made comparison with the results of first order blending scheme using Gauss-Legendre-Lobatto points of~\cite{henemann2021} implemented in {\tt Trixi.jl}~\cite{Ranocha2022, schlottkelakemper2021purely}. Our code is publicly available at~\cite{tenkai}, and the scripts for reproducing results in this work are available at~\cite{paperrepo}. The user only needs to install \correction{{\tt Julia}~\cite{Bezanson2017}} and the remaining dependencies are automatically handled by {\tt Julia} environments and its package manager.


\subsection{1-D Euler equations}\label{sec:res1dsys}
As an example of system of non-linear hyperbolic equations, consider the one-dimensional Euler equations of gas dynamics given by
\begin{equation}\label{eq:1deuler}
\pd{}{t} \begin{pmatrix}
\rho \\
\rho v \\
E
\end{pmatrix} +
\pd{}{x} \begin{pmatrix}
\rho v \\
p + \rho v^2 \\
(E+p)v
\end{pmatrix} = 0
\end{equation}
where $\rho, v, p$ and $E$ denote the  density, velocity, pressure and total energy of the gas, respectively. For a polytropic gas, an equation of state $E=E(\rho, v, p)$ which leads to a closed system is given by
\begin{equation}\label{eq:state}
E = E(\rho, v, p) = \frac{p}{\gamma -1}+\frac{1}{2} \rho v^2
\end{equation}
where $\gamma > 1$ is the adiabatic constant. Unless \correction{otherwise specified,} it will be taken as $1.4$ which is the value for air. The time step size for polynomial degree $N$ is computed as
\begin{equation}
\Delta t = C_s \min_e\left ( \frac{\Delta x_e}{|\overline{v}_e| + \overline{c}_e} \right) \text{CFL}(N) \label{eq:dt.lw}
\end{equation}
where $e$ is the element index, $\overline{v}_e, \overline{c}_e$ are velocity and sound speed of element mean in element $e$, $\text{CFL}(N)$ is the optimal CFL number obtained by Fourier stability analysis (Table 1 of~\cite{babbar2022}) and $C_s \le 1$ is a safety factor. \correction{Most of the numerical results presented in this work use degree $N=4$ for which $\text{CFL}(N) = 0.069$.} The admissibility preservation of subcell based MUSCL-Hancock imposes a time restriction (Theorem~\ref{thm:muscl.admissibility.theorem}) which depends on several quantities other than \correction{element means}, including some evolved quantities, see equations~(\ref{eq:new.cfl1}, \ref{eq:new.cfl2},
\ref{eq:new.cfl3.conservative}). The CFL coefficient of MUSCL-Hancock admissibility is also smaller than $\text{CFL}(N)$ in~\eqref{eq:dt.lw}, see Remark~\ref{rmk:mh.restriction.for.fr}. However, we have found the time step given by~\eqref{eq:dt.lw} with $C_s=0.98$ to be sufficient for admissibility preservation in all the simulations we have performed. Thus, we do not explicitly impose the CFL restrictions in Theorem~\ref{thm:muscl.admissibility.theorem} as they are more severe and expensive to compute. If the admissibility is violated in any cell, then the time update can be repeated in those cells by lowering the time step by some fraction.
\subsubsection{Shu-Osher problem}\label{sec:shuosher}
This problem was developed in~\cite{Shu1989} to test the numerical scheme's capability to accurately capture a shock wave and its interaction with a smooth density field, which then propagates downstream of the shock. Here, we compute the solution of~\eqref{eq:1deuler} up to time $t=1.8$ with initial condition
\begin{equation}\label{eq:shuosher}
(\rho,v,p)=\begin{cases}
(3.857143, 2.629369, 10.333333), & \mbox{ if } x<-4\\
(1+0.2\sin(5x),0,1), & \mbox{ if }x\geq -4
\end{cases}
\end{equation}
prescribed in the domain $[-5,5]$ with transmissive boundary conditions. The smooth density profile passes through the shock and appears on the other side, and its accurate computation is challenging due to numerical dissipation. Due to presence of both spurious oscillations and smooth extremums, this becomes a good test for testing robustness and accuracy of limiters. We discretize the spatial domain with 400 cells using polynomial degree $N=4$ and compare blending schemes and TVB limiter with parameter $M=300$~\cite{Qiu2005b}. The density component of the approximate solutions computed for the compared limiters are plotted against a reference solution obtained using a very fine mesh, and are given in Figure~(\ref{fig:ShuOsher}). The three limiters show similar performance in Figure~(\ref{fig:ShuOsher}a) on the large scale. The enlarged plot in Figure~(\ref{fig:ShuOsher}b) shows that the MUSCL-Hancock blending scheme is able to capture smooth extrema better than the first order blending and the TVB scheme.
\begin{figure}
\centering
\begin{tabular}{cc}
\includegraphics[width=0.45\textwidth]{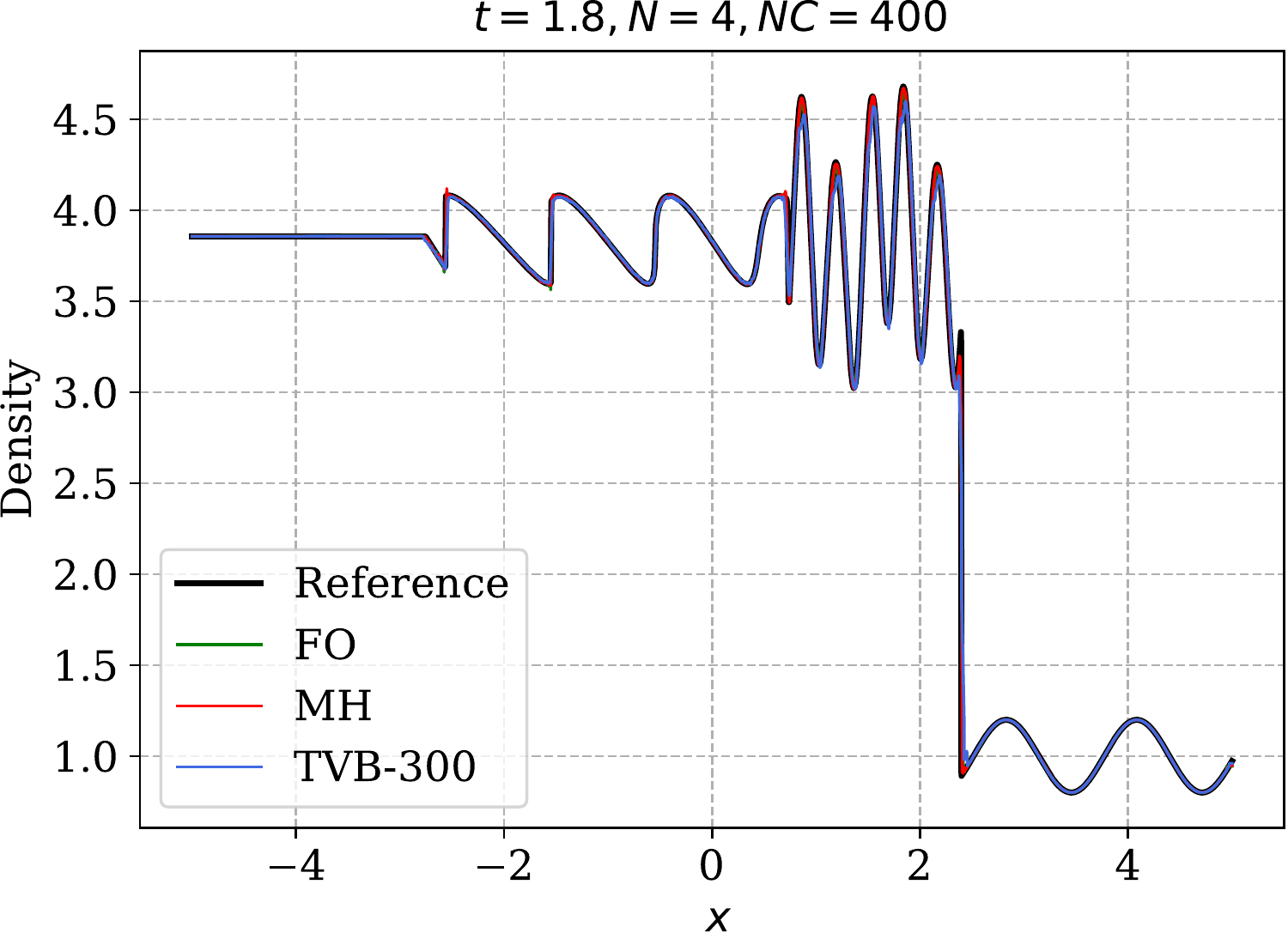} &
\includegraphics[width=0.45\textwidth]{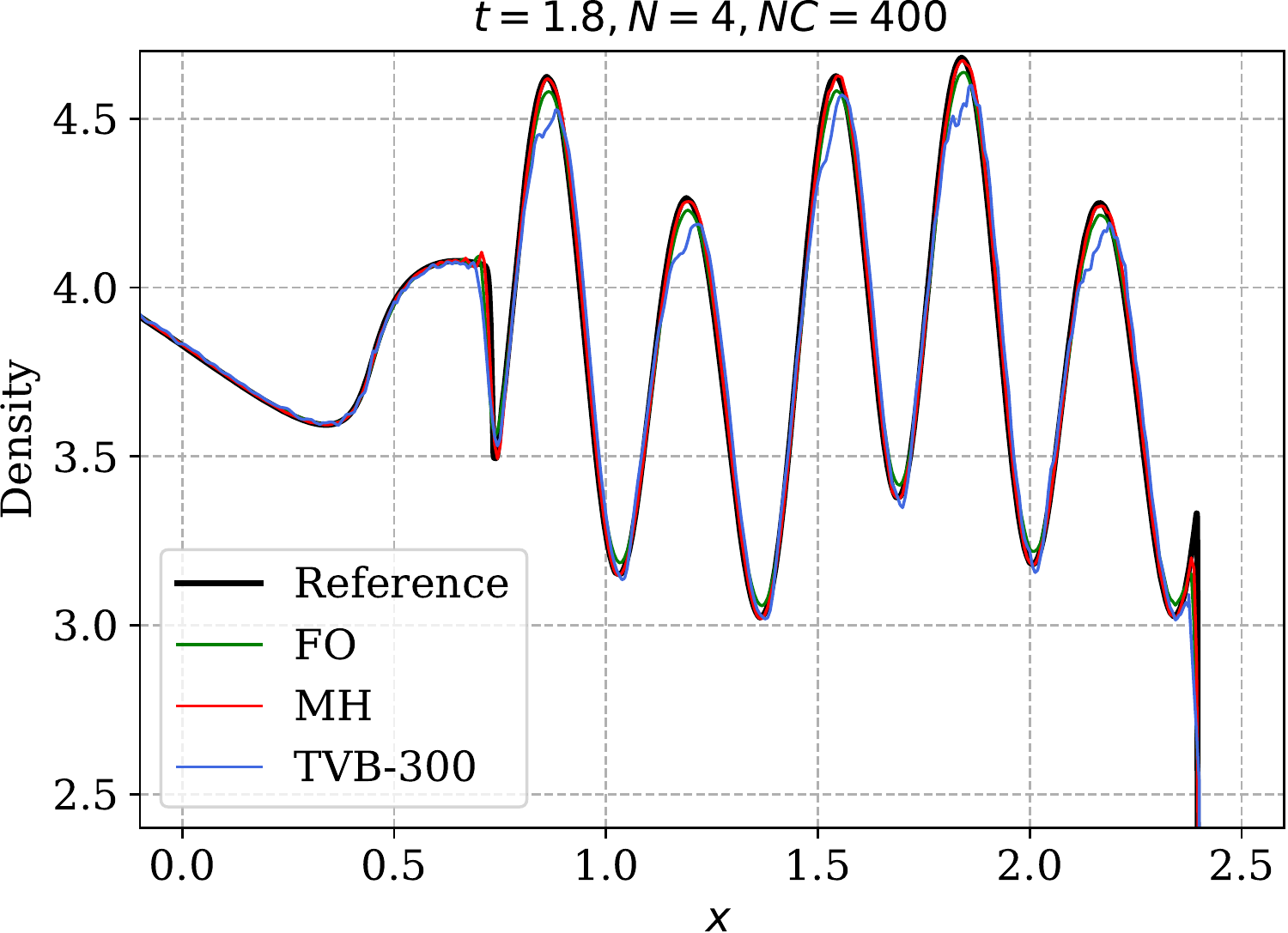} \\
(a) & (b)
\end{tabular}
\caption{Shu-Osher problem, numerical solution with degree $N=4$ using first order (FO) and MUSCL-Hancock (MH) blending schemes, and TVB limited scheme (TVB-300) with parameter $M=300$. (a) Full and (b) zoomed density profiles of numerical solutions are shown up to time $t=1.8$ on a mesh of 400 cells.}
\label{fig:ShuOsher}
\end{figure}

\subsubsection{Blast wave}\label{sec:blast}
The Euler equations \eqref{eq:1deuler} are solved with the initial condition
\begin{equation*}
(\rho,v,p)=\begin{cases}
(1,0, 1000), & \mbox{ if } x<0.1\\
(1,0,0.01), & \mbox{ if }  0.1 < x <0.9\\
(1,0, 100), & \mbox{ if } x> 0.9
\end{cases}
\end{equation*}
in the domain $[0,1]$. This test was originally introduced in~\cite{Woodward1984} to check the capability of the numerical scheme to accurately capture the shock-shock interaction scenario. The boundaries are set as solid walls by imposing the reflecting boundary conditions at $x=0$ and $x=1$. The solution consists of reflection of shocks and expansion waves off the boundary wall and several wave interactions inside the domain. The numerical solutions are inadmissible if the positivity correction is not applied. With a grid of 400 cells using polynomial degree $N=4$, we run the simulation till the time $t=0.038$ where a high density peak profile is produced. As in the previous test, we compare first order (FO) and MUSCL-Hancock (MH) blending schemes, and TVB limiter with parameter $M=300$~\cite{Qiu2005b} (TVB-300). We compare the performance of limiters in Figure~(\ref{fig:blast}) where the approximated density and pressure profiles are compared with a reference solution computed using a very fine mesh. Looking at the peak amplitude and contact discontinuity of the density profile which is also shown in the zoomed inset, it is clear that MUSCL-Hancock blending scheme gives the best resolution, especially when compared with the TVB limiter.
\begin{figure}
\centering
\begin{tabular}{cc}
\includegraphics[width=0.45\textwidth]{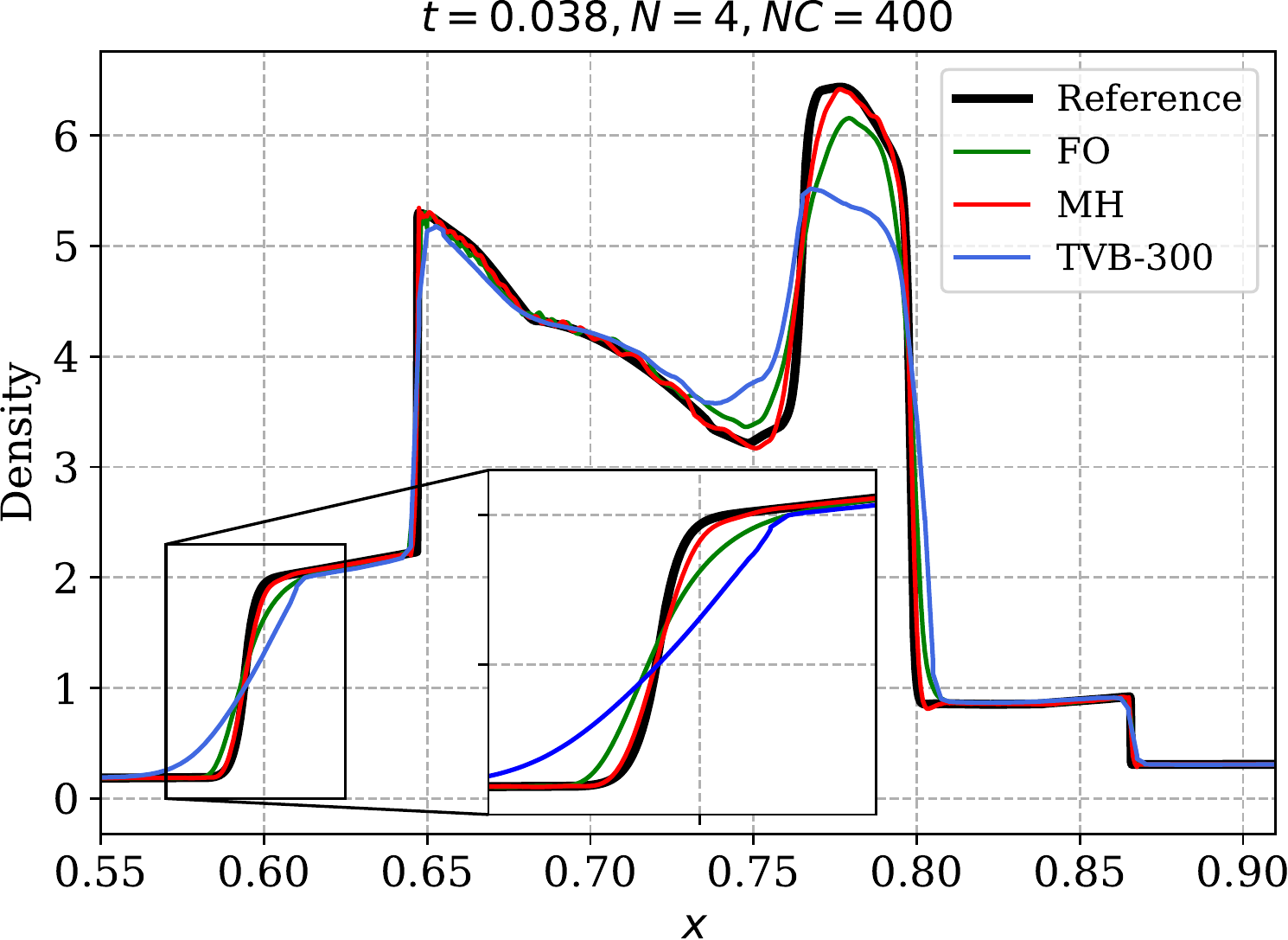} &
\includegraphics[width=0.45\textwidth]{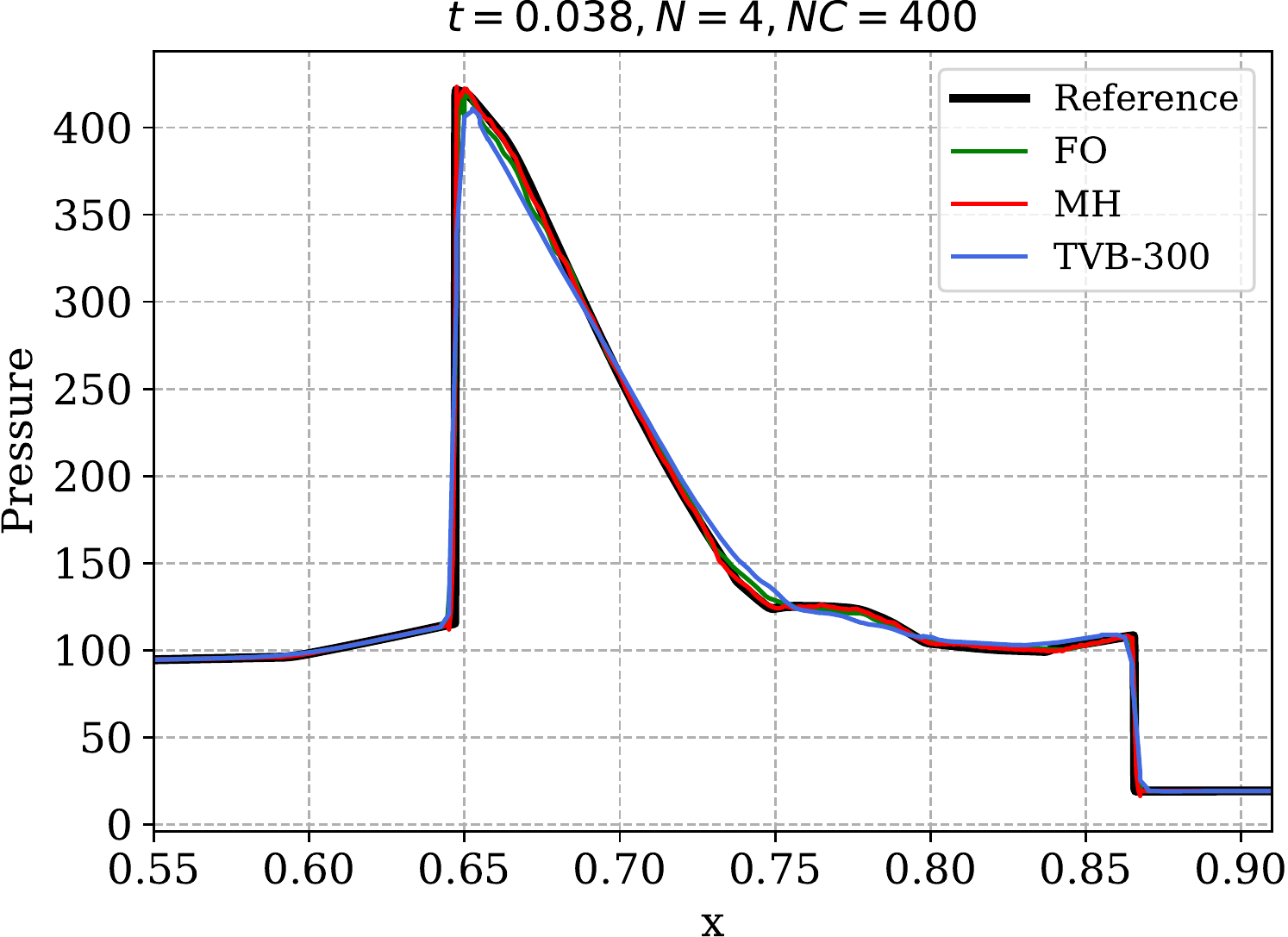} \\
(a) & (b)
\end{tabular}
\caption{Blast wave problem, numerical solution with degree $N=4$ using first order (FO) and MUSCL-Hancock (MH) blending schemes, and TVB limited scheme (TVB-300) with parameter $M=300$. (a) Density, (b) pressure profiles are shown at time $t=0.038$ on a mesh of 400 cells.}
\label{fig:blast}
\end{figure}

\subsubsection{Sedov's blast wave}\label{sec:sedov.blast.1d}
To demonstrate the admissibility preserving property of our scheme, we simulate Sedov's blast wave~\cite{sedov1959}; the test describes the explosion of a point-like source of energy in a gas. The explosion generates a spherical shock wave that propagates outwards, compressing the gas and reaching extreme temperatures and pressures. The problem can be formulated in one dimension as a special case, where the explosion occurs at $x=0$ and the gas is confined to the interval $[-1,1]$ by solid walls. For the simulation, on a grid of 201 cells with solid wall boundary conditions, we use the following initial data~\cite{Zhang2012},
\[
\rho = 1, \qquad
v = 0, \qquad
E(x) = \begin{cases}
\frac{3.2 \times 10^6}{\Delta x}, \qquad & |x| \le \frac{\Delta x}{2} \\
10^{-12}, \qquad & \text{otherwise}
\end{cases}
\]
where $\Delta x$ is the element width. This is a difficult test for positivity preservation because of the high pressure ratios. Nonphysical solutions are obtained if the proposed admissibility preservation corrections are not applied. The density and pressure profiles at $t=0.001$ are obtained using blending schemes are shown in Figure~\ref{fig:sedov.blast}. Results of TVD limiter are not shown as it fails to preserve positivity in this test because the admissibility correction of Lax-Wendroff scheme depends on the blended numerical flux (Section~\ref{sec:admissibility.preservation}).
\begin{figure}
\centering
\begin{tabular}{cc}
\includegraphics[width=0.45\textwidth]{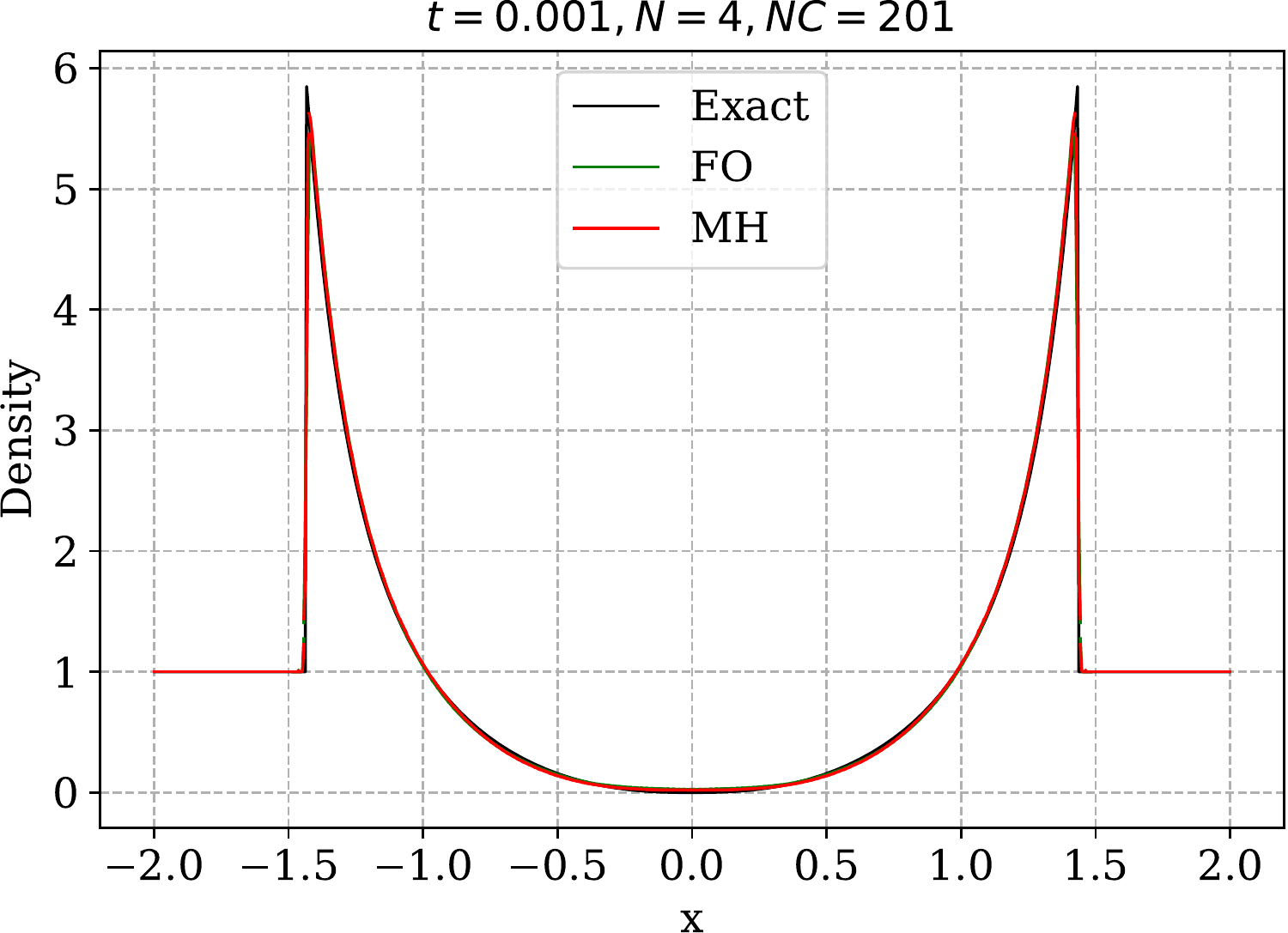} &
\includegraphics[width=0.45\textwidth]{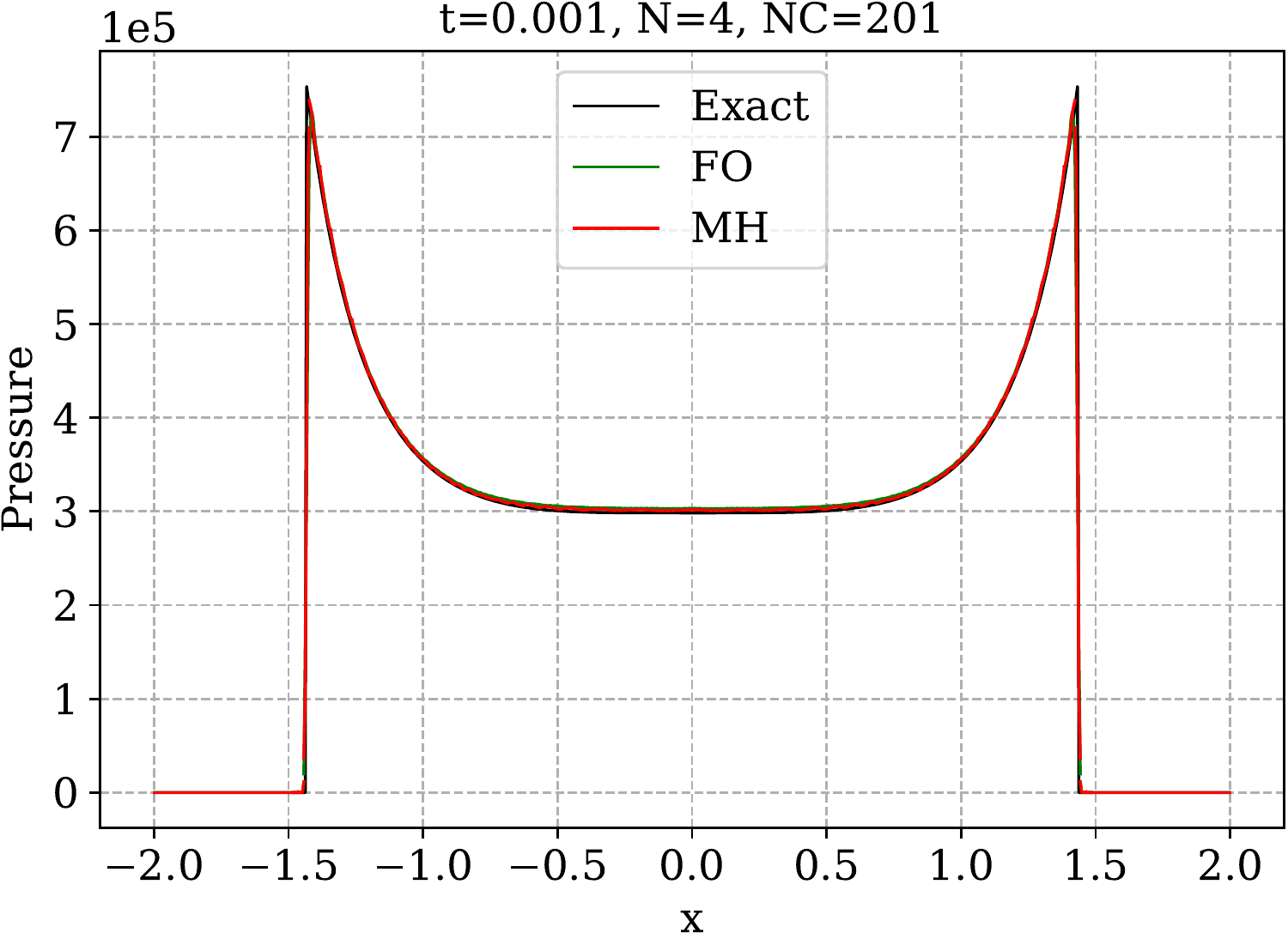} \\
(a) & (b)
\end{tabular}
\caption{Sedov's blast wave problem, numerical solution with degree $N=4$ using first order (FO) and MUSCL-Hancock blending schemes. (a) Density and (b) pressure profiles of numerical solutions are plotted at time $t=0.001$ on a mesh of $201$ cells.}
\label{fig:sedov.blast}
\end{figure}

\subsubsection{Riemann problems}\label{sec:rp1d}
We test two extreme Riemann problems from~\cite{Zhang2010b} to demonstrate admissibility preservation of our scheme. The first is a Riemann problem with no shocks and two rarefactions, which move away from each other leading to a near vacuum state in the exact solution. The low densities make it a challenging test, as the oscillations can easily cause negative density values. We run the simulation on the domain $[-1,1]$ with initial data
\[
(\rho, v, p) = \begin{cases}
(7.0, -1.0, 0.2), \qquad & -1 \le x \le 0 \\
(7.0,  1.0, 0.2), & \text{otherwise}
\end{cases}
\]
The results obtained using blending schemes are shown in Figure~\ref{fig:double.rarefaction} on a mesh of $200$ cells with transmissive boundary conditions at time $t=0.6$.

The second test is a 1D Leblanc shock tube problem with initial data
\[
(\rho, v, p) = \begin{cases}
(2, 0, 10^9), \qquad & -1 \le x \le 0 \\
(0.001, 0, 1), & \text{otherwise}
\end{cases}
\]
The solution has extremely high density and pressure ratios across the shock and the numerical solutions give negative pressure if \correction{the proposed admissibility preservation techniques are} not applied. The log-scaled results obtained using blending schemes are shown in Figure~\ref{fig:leblanc} at time $t=0.001$ on a mesh of $800$ cells with transmissive boundary conditions.
\begin{figure}
\centering
\begin{tabular}{cc}
\includegraphics[width=0.45\textwidth]{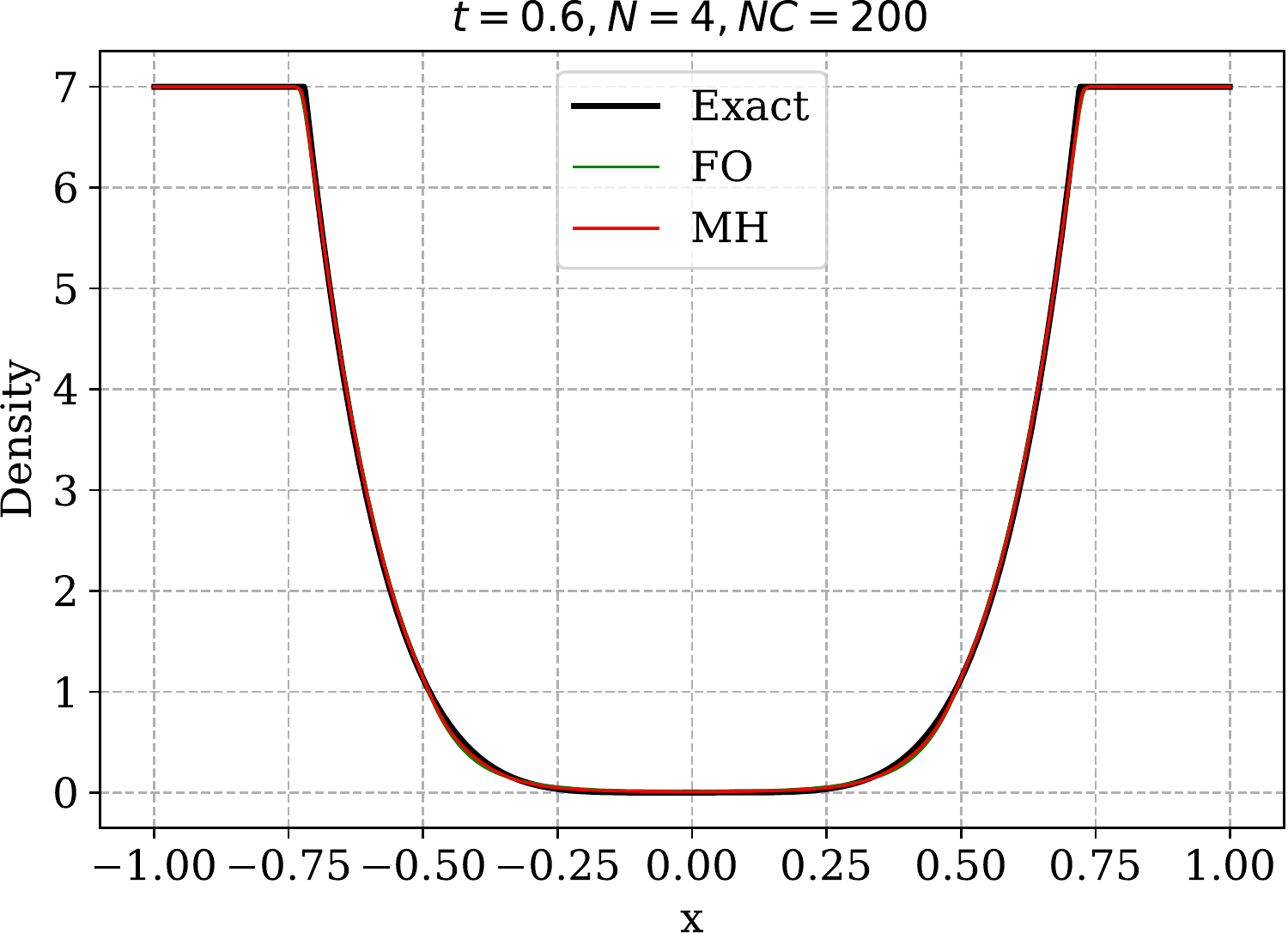} &
\includegraphics[width=0.45\textwidth]{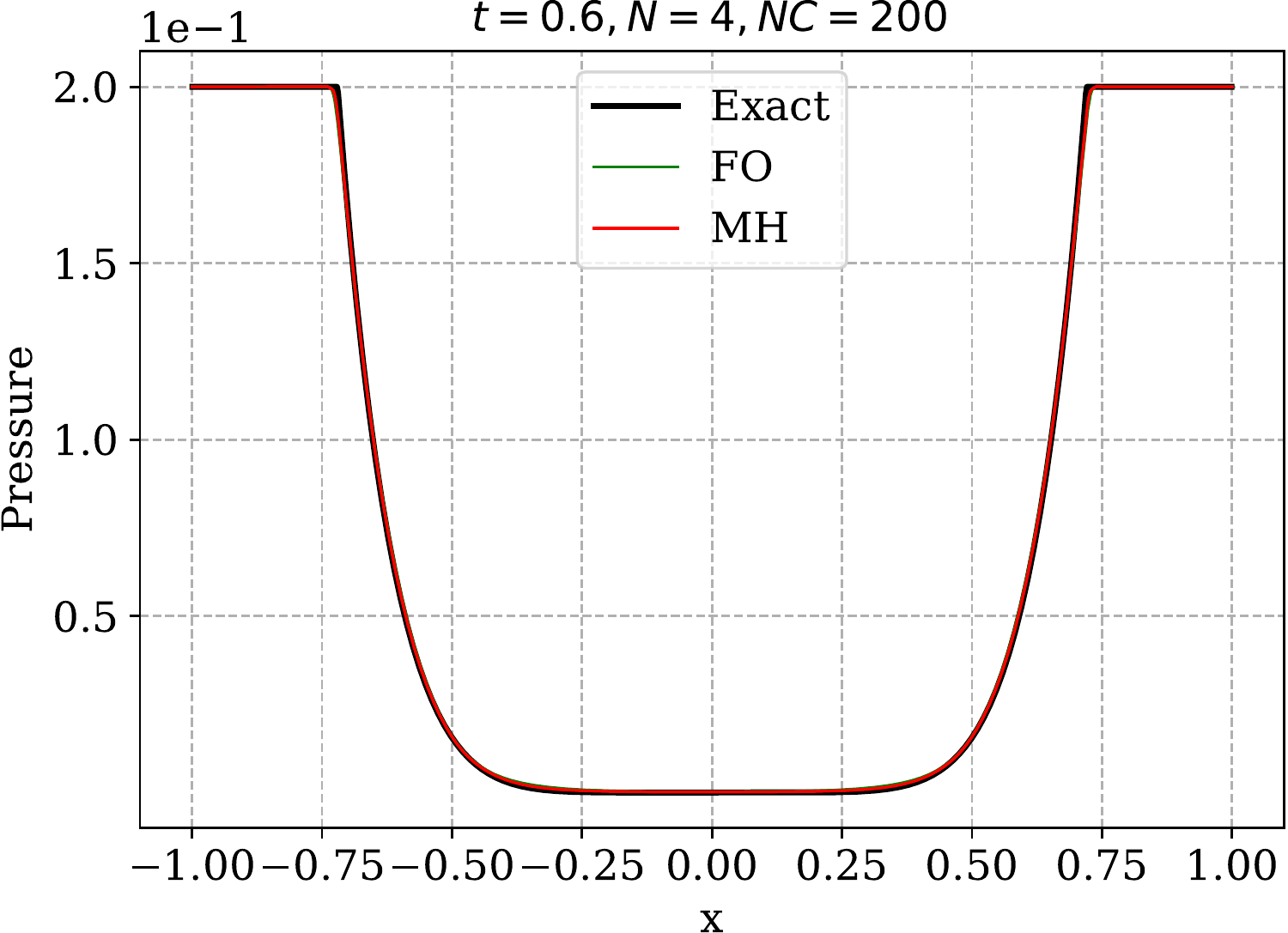} \\
(a) & (b)
\end{tabular}
\caption{Double rarefaction problem, numerical solution with degree $N=4$ using first order (FO) and MUSCL-Hancock (MH) blending. (a) Density and (b) pressure profiles of numerical solutions are plotted at $t=0.6$ on a mesh of $200$ cells.}
\label{fig:double.rarefaction}
\end{figure}

\begin{figure}
\centering
\begin{tabular}{cc}
\includegraphics[width=0.45\textwidth]{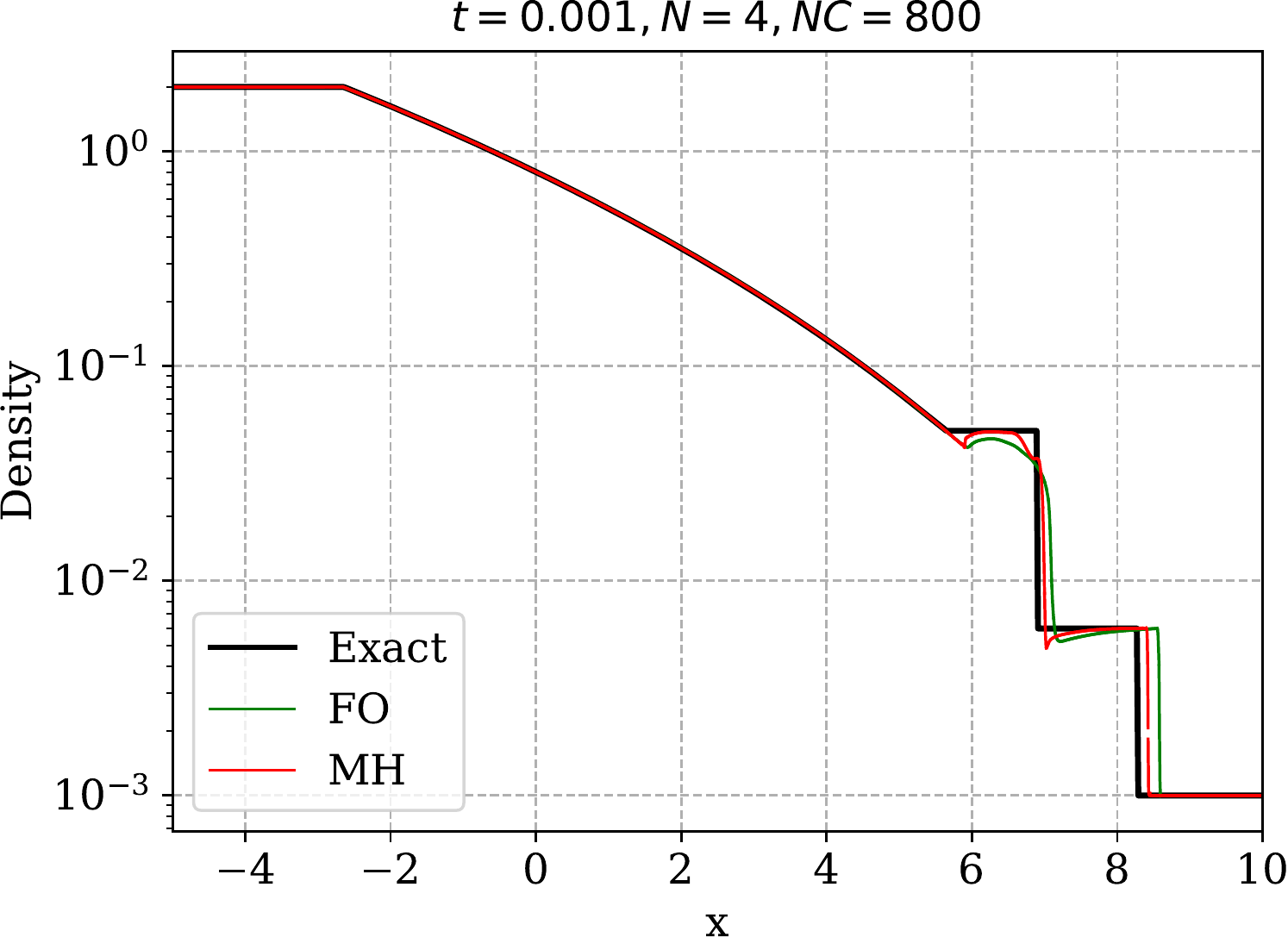} &
\includegraphics[width=0.45\textwidth]{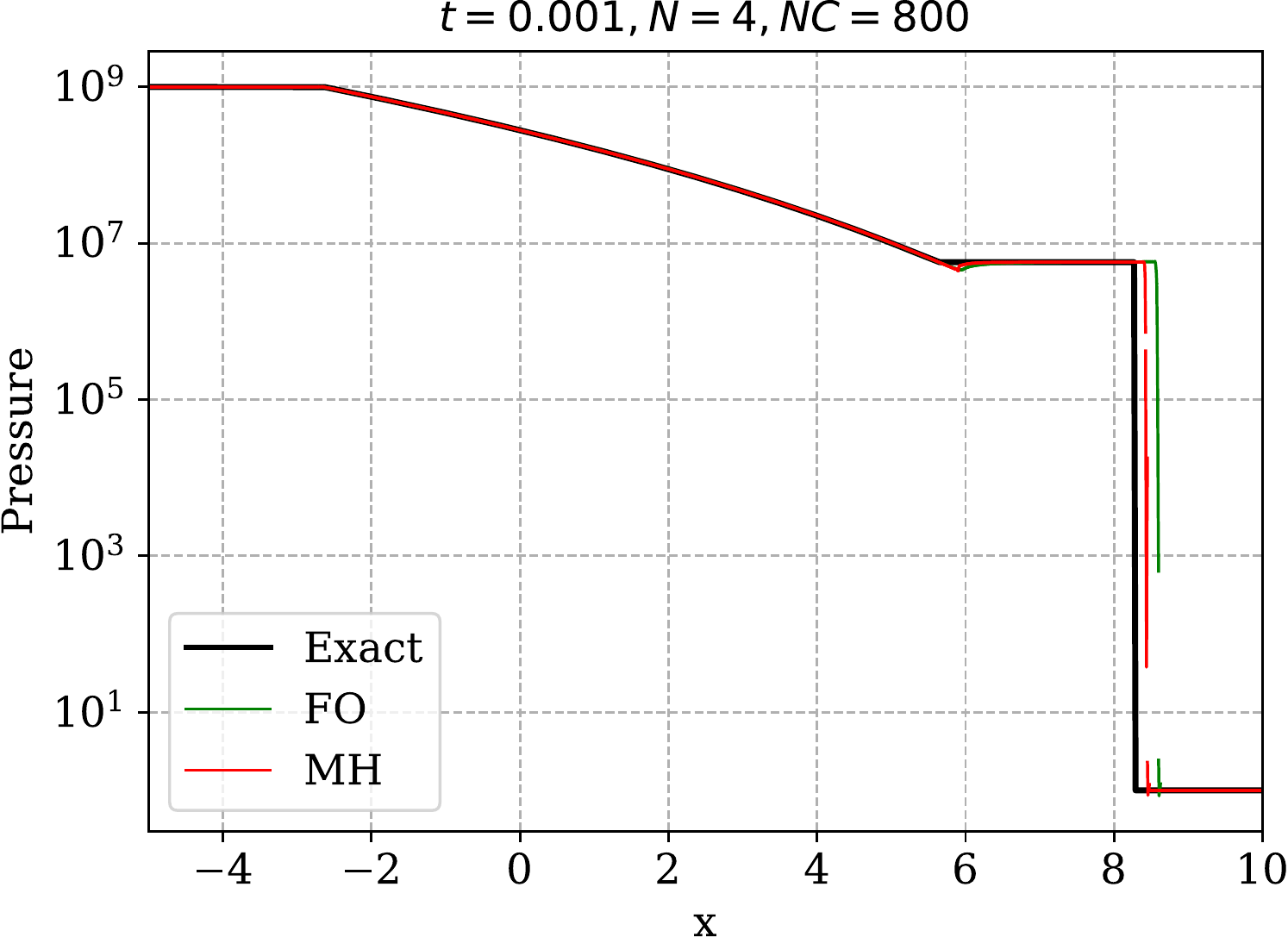} \\
(a) & (b)
\end{tabular}
\caption{Leblanc's test, numerical solution with polynomial degree $N=4$ using first order (FO) and MUSCL-Hancock (MH) blending. (a) Density and (b) pressure profiles of numerical solutions with log-scales are plotted at $t=0.001$ on a mesh of $800$ cells.}
\label{fig:leblanc}
\end{figure}

\subsection{2-D advection equation}
We consider the advection equation in two dimensions
\begin{equation}\label{eq:2dvaradv}
u_t+\nabla \cdot [a(x,y)u] =0
\end{equation}
with a test case from~\cite{LeVeque1996} where the equation~\eqref{eq:2dvaradv} is solved with a divergence free velocity field, $a=(\frac{1}{2}-y,x-\frac{1}{2})$, and an initial condition which consists of a slotted disc, a cone and a smooth hump, given as follows
\begin{align*}
u(x,y,0) &= u_1(x,y)+u_2(x,y)+u_3( x,y),\quad  (x,y)\in [0,1]\times[0,1]\\
u_1(x,y) &= \frac{1}{4}(1+\cos (\pi q( x,y))),\quad q(x,y) = \min ( \sqrt{(x-\bar x)^2+(y-\bar y)^2},r_0 )/r_0,{(\bar x,\bar y)} = (0.25,0.5), r_0=0.15\\
u_2(x,y) &= \begin{cases}
1-\dfrac{1}{r_0} \sqrt{(x-\bar x)^2+(y-\bar y)^2} & \mbox{ if } (x-\bar x)^2+(y-\bar y)^2\le r_0^2\\
0 & \mbox{otherwise}
\end{cases}, \quad {(\bar x,\bar y)} =(0.5,0.25), r_0=0.15\\
u_3(x,y) &= \begin{cases}
1 & \mbox{ if } (x,y) \in \mathrm{C}\\
0 & \mbox{otherwise}
\end{cases}
\end{align*}
where $\mathrm{C}$ is a slotted disc with center at $(0.5,0.75)$ and radius of $0.15.$

The characteristics of the PDE are circles and the solution returns to its initial state after a period of time $t=2 \pi$. Figure~(\ref{fig:composite.signal.2d}) compares contour plots of polynomial solutions obtained using the LWFR method of degree $N=4$ with TVB limiter \correction{using} a fine-tuned parameter $M=100$, and with blending limiter \correction{using first order} and MUSCL-Hancock reconstruction\correction{s}, after one time period. The blending limiter with MUSCL-Hancock reconstruction is shown to produce more accurate solutions \correction{among the} three profiles \correction{especially when} compared to the TVB limiter, as the TVB limiter results in greater smearing of the profile. The sharp features of slotted disc profile show the most notable improvement.


\begin{figure}
\centering
\begin{tabular}{cccc}
\newcommand{\spacingcom}{\kern-2.2em}
\kern-2.2em \includegraphics[width=0.27\textwidth] {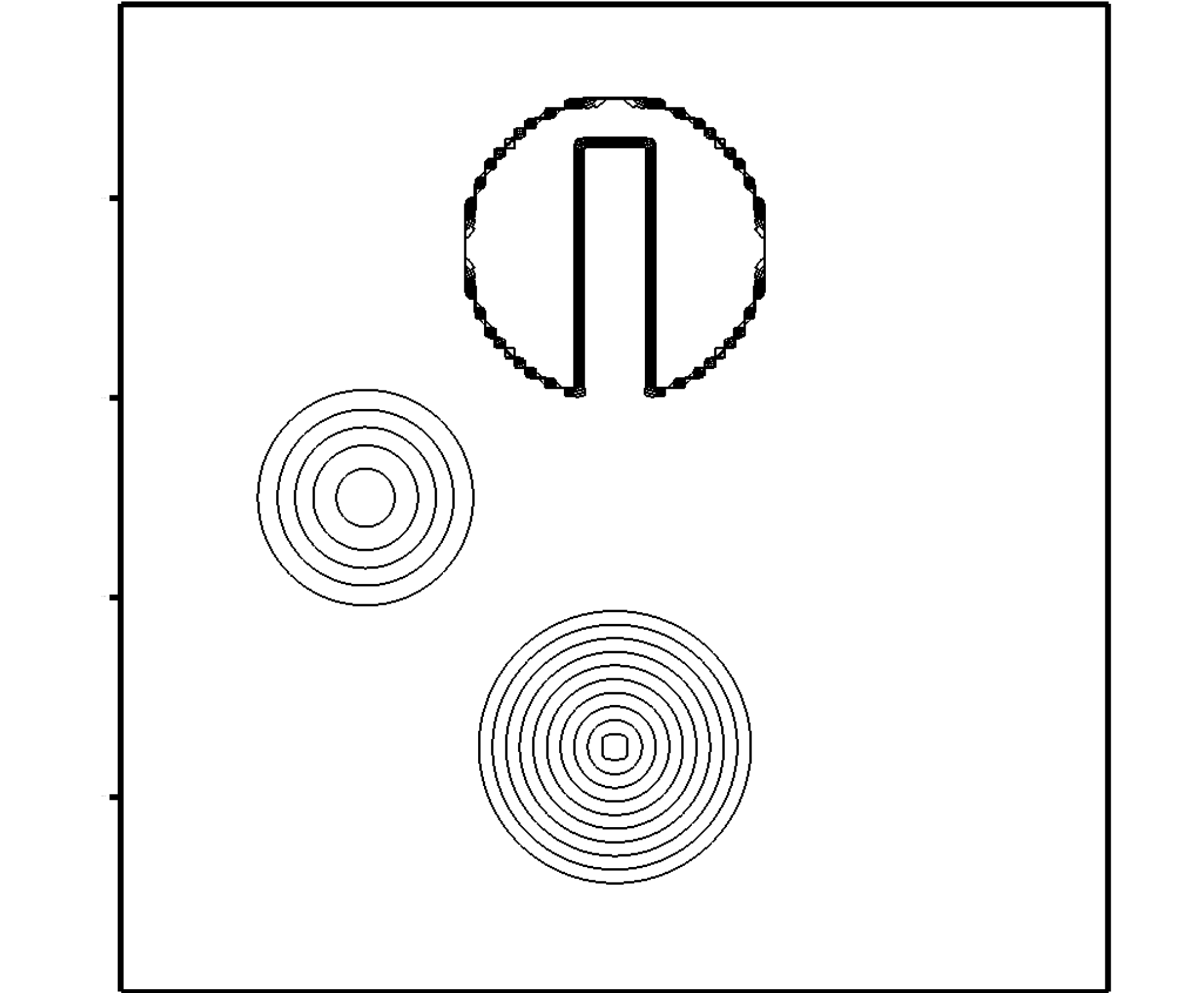} \kern-2.2em &
\includegraphics[width=0.27\textwidth]{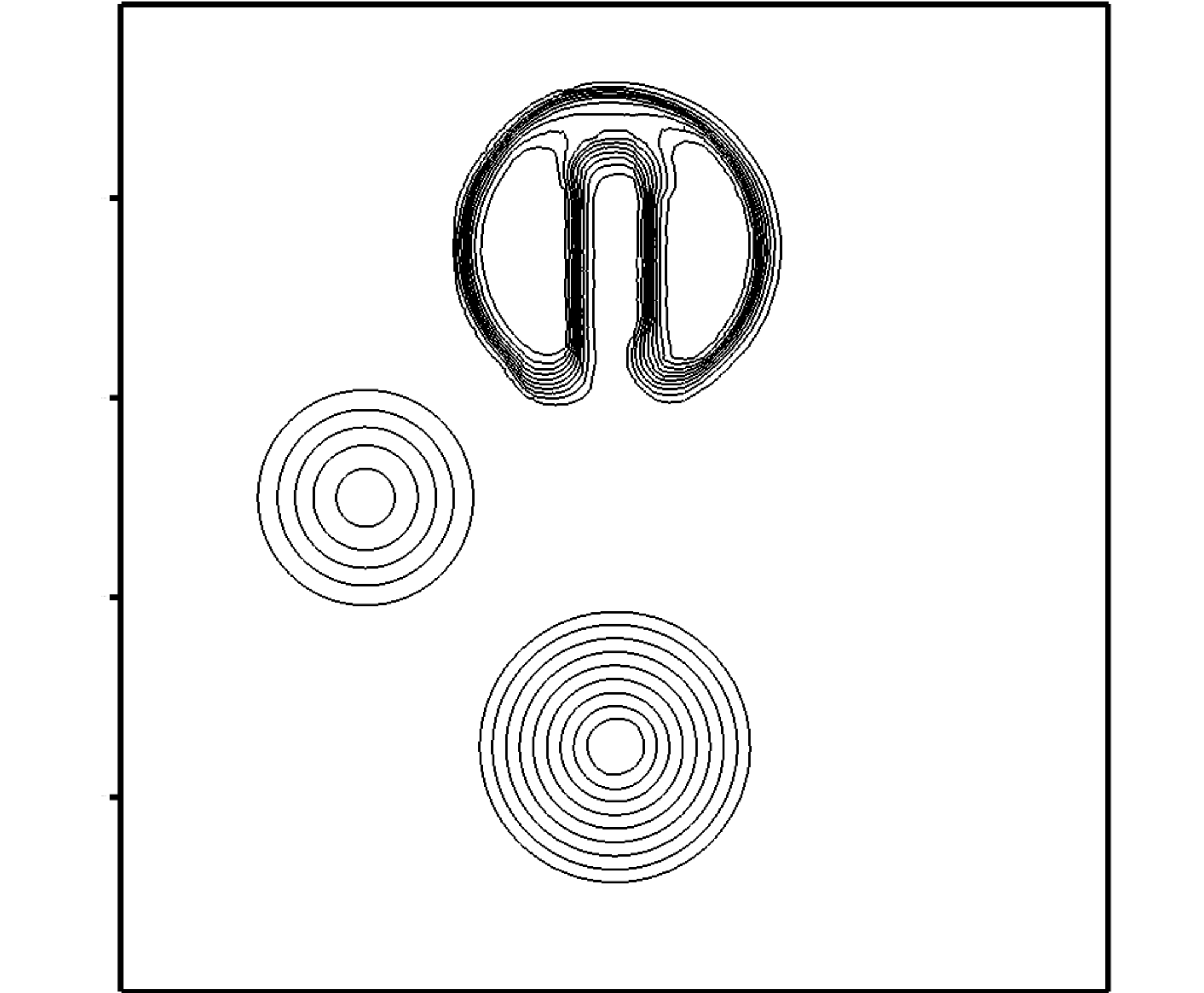} \kern-2.2em &
\includegraphics[width=0.27\textwidth]{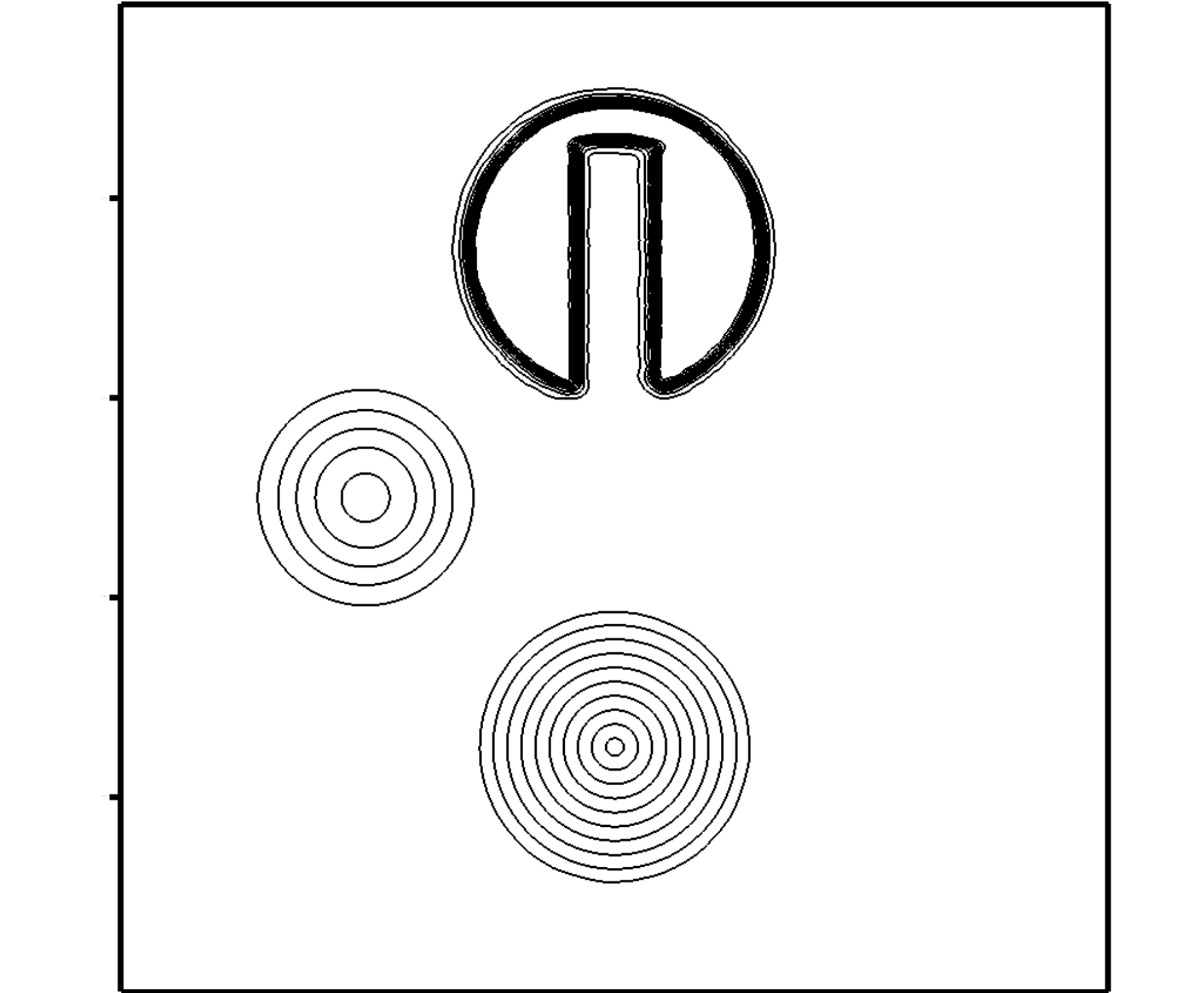} \kern-2.2em &
\includegraphics[width=0.27\textwidth]{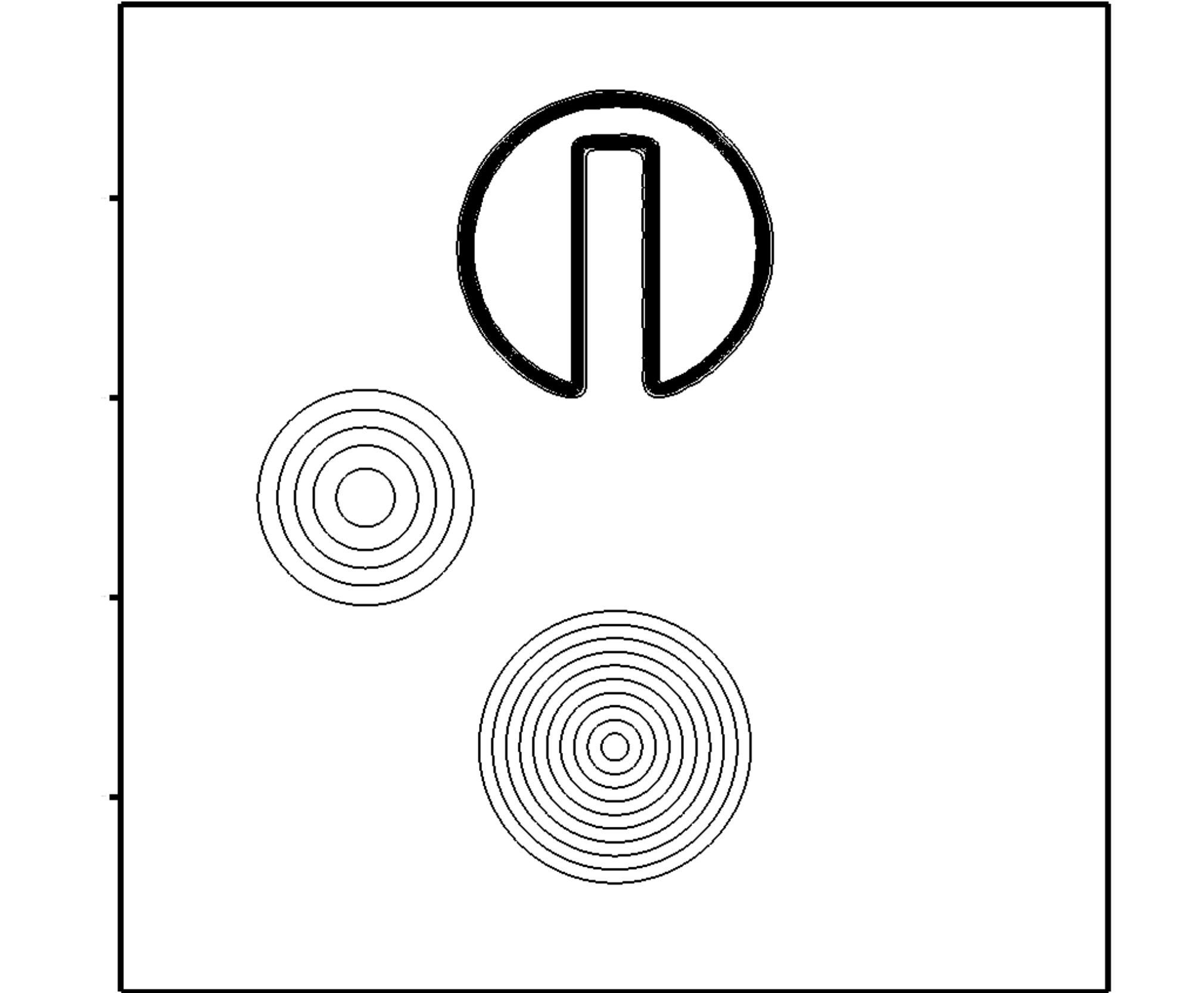} \kern-2.2em \\
\kern-2.2em (a) Exact \kern-2.2em & (b) TVB with $M=100$ \kern-2.2em & (c) FO blending \kern-2.2em & (d) MH blending \kern-2.2em
\end{tabular}
\caption{Rotation of a composite signal with velocity $\mathbf{a} = (\frac{1}{2} - y, x - \frac{1}{2})$,
numerical solution with polynomial degree $N=4$ on a mesh of $100^2$ elements.}
\label{fig:composite.signal.2d}
\end{figure}

\subsection{2-D Euler equations}

We consider the two-dimensional Euler equations of gas dynamics given by
\begin{equation}\label{eq:2deuler}
\pd{}{t} \begin{pmatrix}
\rho \\
\rho u \\
\rho v \\
E
\end{pmatrix} +
\pd{}{x} \begin{pmatrix}
\rho u \\
p + \rho u^2 \\
\rho u v \\
(E+p)u
   \end{pmatrix} +
\pd{}{y} \begin{pmatrix}
\rho v \\
\rho u v \\
p + \rho v^2 \\
(E+p)v
\end{pmatrix}= 0
\end{equation}
where $\rho, p$ and $E$ denote the density, pressure and total energy of the gas, respectively and $(u, v)$ are Cartesian components of the fluid velocity. For a polytropic gas, an equation of state $E=E(\rho, u, v, p)$ which leads to a closed system is given by
\begin{equation}\label{eq:2dstate}
E = E(\rho, u, v, p) = \frac{p}{\gamma -1}+\frac{1}{2} \rho (u^2 + v^2)
\end{equation}
where $\gamma > 1$ is the adiabatic constant. Unless \correction{otherwise specified}, the adiabatic constant will be taken as $1.4$ in the numerical tests, which is the typical value for air.

The time step size for polynomial degree $N$ is computed as
\begin{align}
\Delta t = C_s \min_e\left ( {\frac{|\overline{u}_e| + \overline{c}_e}{\Delta x_e} + \frac{|\overline{v}_e| + \overline{c}_e}{\Delta y_e}} \right)^{-1} \text{CFL}(N)
\label{eq:time.step.2d}
\end{align}
where $e$ is the element index, $(\overline{u}_e,\overline{v}_e), \overline{c}_e$ are velocity and sound speed of element mean in element $e$, $\text{CFL}(N)$ is the optimal CFL number obtained by Fourier stability analysis (Table 1 of~\cite{babbar2022}) and $C_s \le 1$ is a safety factor. \correction{Most of the numerical results presented in this work use degree $N=4$ for which $\text{CFL}(N) = 0.069$.} As in the 1-D case,~\eqref{eq:time.step.2d} will not guarantee that the time step restriction for admissibility of MUSCL-Hancock scheme on the subcells is satisfied. However, we have found all tests to work with~\eqref{eq:time.step.2d} using $C_s=0.98$ and the results are shown with that safety factor unless otherwise specified.

For verification of numerical results and to demonstrate the accuracy gain of our proposed Lax-Wendroff blending scheme with MUSCL-Hancock reconstruction using Gauss-Legendre points, we will compare our results with the first order blending scheme using Gauss-Legendre-Lobatto (GLL) points of~\cite{henemann2021} available in {\tt Trixi.jl}~\cite{Ranocha2022}. Both solvers use the same time step sizes in all results. We have also performed experiments using LWFR with first order blending scheme and Gauss-Legendre (GL) points, and observed lower accuracy than the MUSCL-Hancock blending scheme, but higher accuracy than the first order blending scheme implementation of {\tt Trixi.jl} using GLL points. These results are expected since GL points and quadrature are more accurate than GLL points, and MUSCL-Hancock is also more accurate than first order finite volume method. However, to save space, we have not presented the results of LWFR with first order blending.

\subsubsection{Isentropic vortex convergence test}
This problem~\cite{Yee1999,Spiegel2016} consists of a vortex that advects at a constant velocity while the entropy is constant in both space and time. The initial condition is given by
\[
\rho = \left[ 1 - \frac{\beta^2 (\gamma - 1)}{8 \gamma \pi^2} \exp (1 -
r^2) \right]^{\frac{1}{\gamma - 1}}, \qquad u = M_\infty \cos \alpha -
\frac{\beta (y - y_c)}{2 \pi} \exp \left( \frac{1 - r^2}{2} \right)
\]
\[
v = M_\infty \sin \alpha + \frac{\beta (x - x_c)}{2 \pi} \exp \left( \frac{1 -
r^2}{2} \right), \qquad r^2 = (x - x_c)^2 + (y - y_c)^2
\]
and the pressure is given by $p = \rho^{\gamma}$. We choose the parameters $\beta = 5$, $M_\infty = 0.5$, $\alpha = 45^o$, $(x_c, y_c) = (0, 0)$ and the domain is taken to be $[- 10, 10] \times [- 10, 10]$ with periodic boundary conditions.  For this configuration, the vortex returns to its initial position after a time interval of $T=20\sqrt{2}/M_\infty$  units. We run the computations up to a  time  $t = T$ when the vortex has crossed the domain once in the diagonal direction. Figure~(\ref{fig:isentropic.convergence}a) compares the $L^2$ error of density sampled at $N+3$ equispaced points against grid resolution when using the blending limiter. It can be seen that the limiter does not activate for adequately high resolution, yielding the same optimal convergence rates as those achieved without the limiter, as shown in Figure~(\ref{fig:isentropic.convergence}b).

\begin{figure}
\begin{center}
\begin{tabular}{cc}
\includegraphics[width=0.45\textwidth]{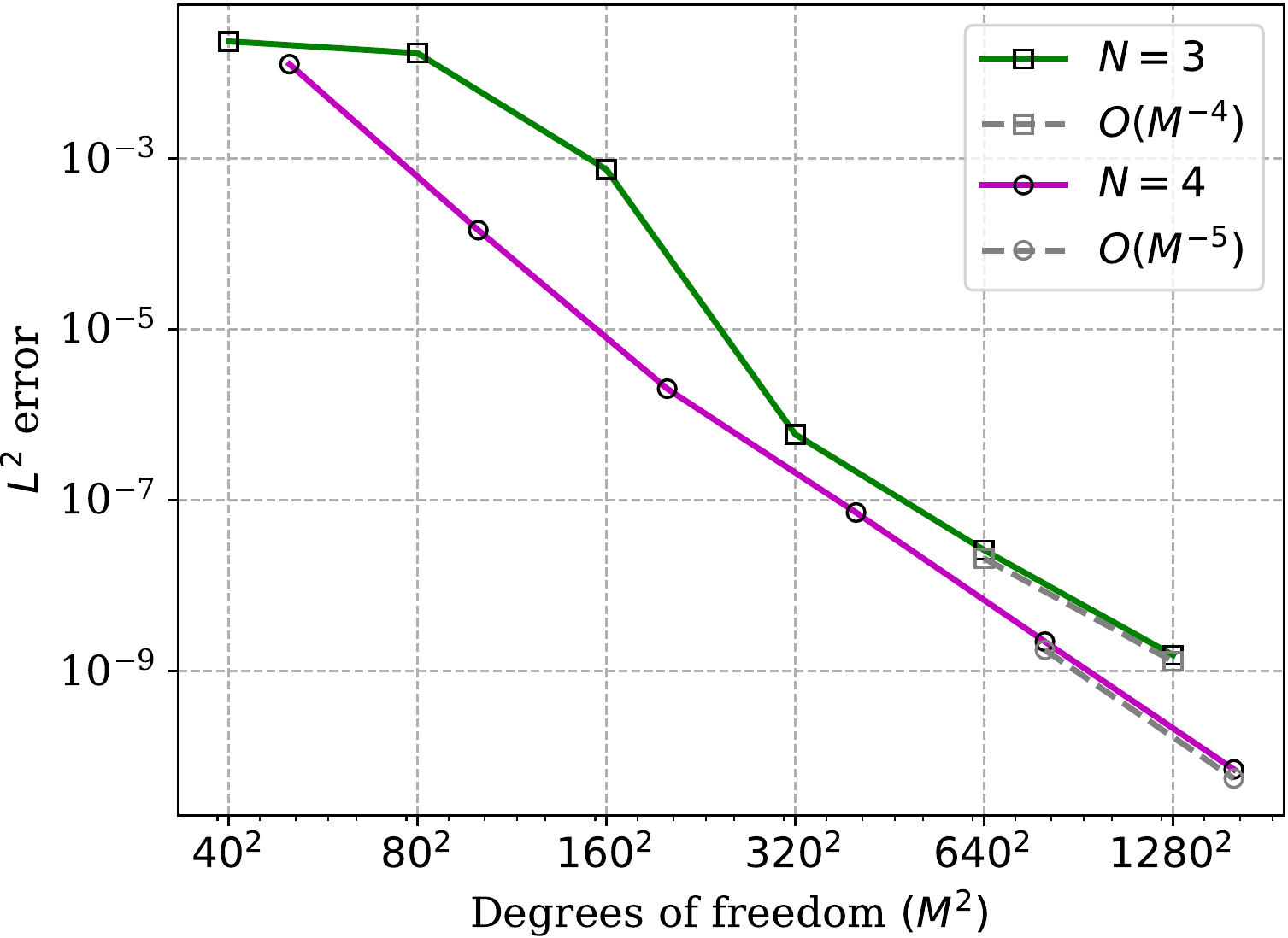} & \includegraphics[width=0.45\textwidth]{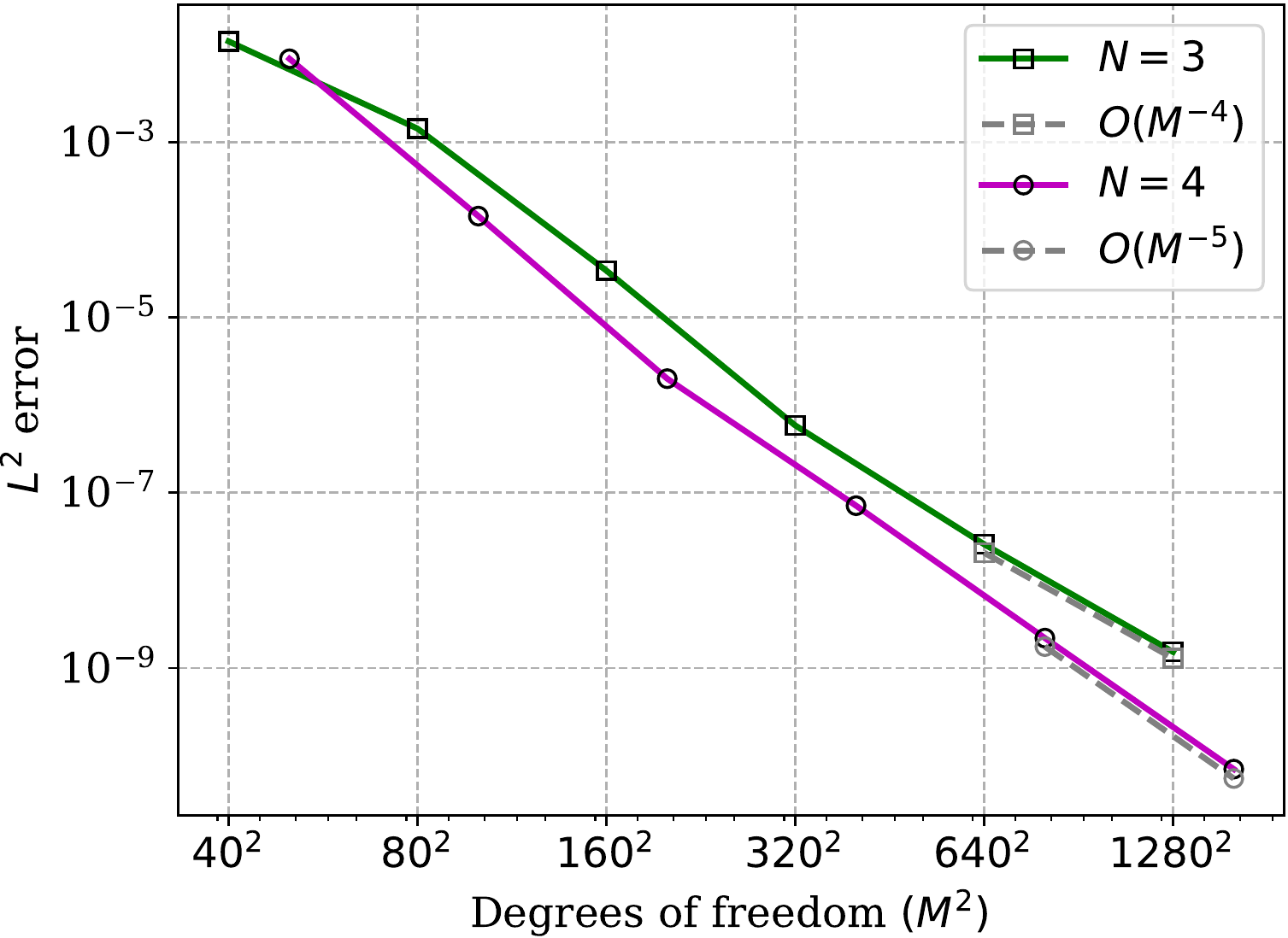} \\
(a) & (b)
\end{tabular}
\end{center}
\caption{Convergence analysis of isentropic vortex test for polynomial degrees $N=3,4$ when (a) the blending limiter is active (b) no limiter is active.}
\label{fig:isentropic.convergence}
\end{figure}


\subsubsection{2-D Riemann problem}
2-D Riemann problems consist of four constant states and have been studied theoretically and numerically for gas dynamics in~\cite{Glimm1985}. We consider this problem in the square domain $[0,1]^2$ where each of the four quadrants has one constant initial state and every jump in initial condition leads to an elementary planar wave, i.e., a shock, rarefaction or contact discontinuity. There are only 19 such genuinely different configurations possible~\cite{Zhang1990,Lax1998}. As studied in~\cite{Zhang1990}, a bounded region of subsonic flows is formed by interaction of different planar waves leading to appearance of many complex structures depending on the elementary planar flow. We consider configuration 12 of~\cite{Lax1998} consisting of 2 positive slip lines and two forward shocks, with initial condition
\begin{align*}
(\rho, u, v, p) = \begin{cases}
(0.5313, 0, 0, 0.4)  \qquad & \text{if } x \ge 0.5, y \ge 0.5 \\
(1, 0.7276, 0, 1)  & \text{if } x < 0.5, y \ge 0.5 \\
(0.8, 0, 0, 1) & \text{if } x < 0.5, y < 0.5 \\
(1, 0, 0.7276, 1) & \text{if } x \ge 0.5, y < 0.5
\end{cases}
\end{align*}
The simulations are performed with transmissive boundary conditions on an enlarged domain \correction{up to time $t=0.25$}. The density profiles obtained from the MUSCL-Hancock blending scheme and {\tt Trixi.jl} are shown in Figure~(\ref{fig:rp2d}). We see that both schemes give similar resolution in most regions. The MUSCL-Hancock blending scheme gives better resolution of the small scale structures arising across the slip lines.

\correction{A plot of the blending coefficients computed by the smoothness indicator is shown in Figure~(\ref{fig:rp2d.alpha}) at an early time $t=0.025$~(\ref{fig:rp2d.alpha}a) and the final time $t=0.25$~(\ref{fig:rp2d.alpha}b). The blending coefficient takes values close to $\alpha = 1$ in the vicinity of shocks while smaller values are seen near the stationary contact discontinuities. Figure~\eqref{fig:rp2d.alpha.stats} shows the percentage of cells in which the indicator function $\alpha>0$ as a function of time. From these figures we see that limiting is only performed in a small subset of the elements in the grid and the indicator is able to track the sharp features and ignore the smooth regions.}

\begin{figure}
\centering
\begin{tabular}{cc}
\includegraphics[width=0.45\textwidth]{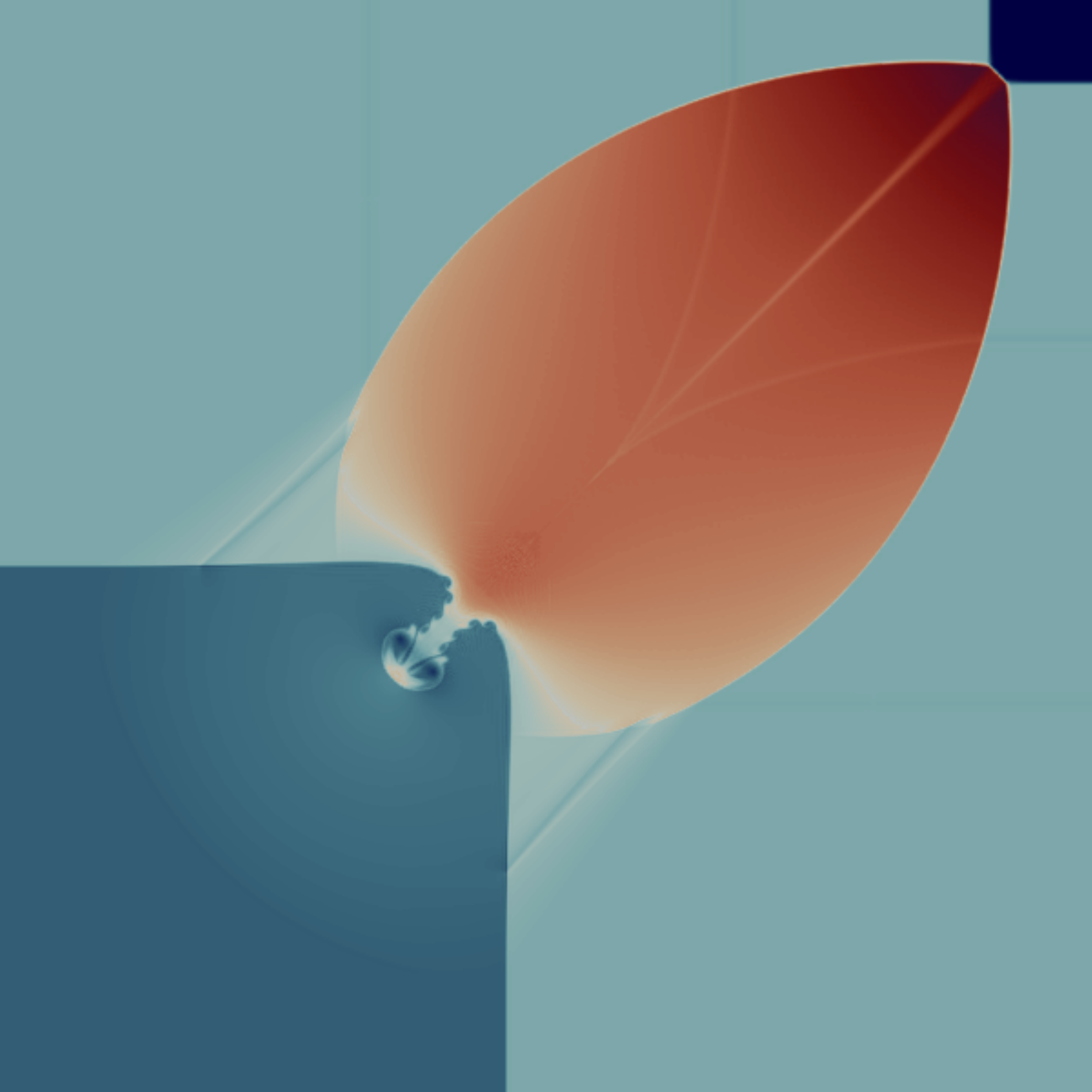} &
\includegraphics[width=0.45\textwidth]{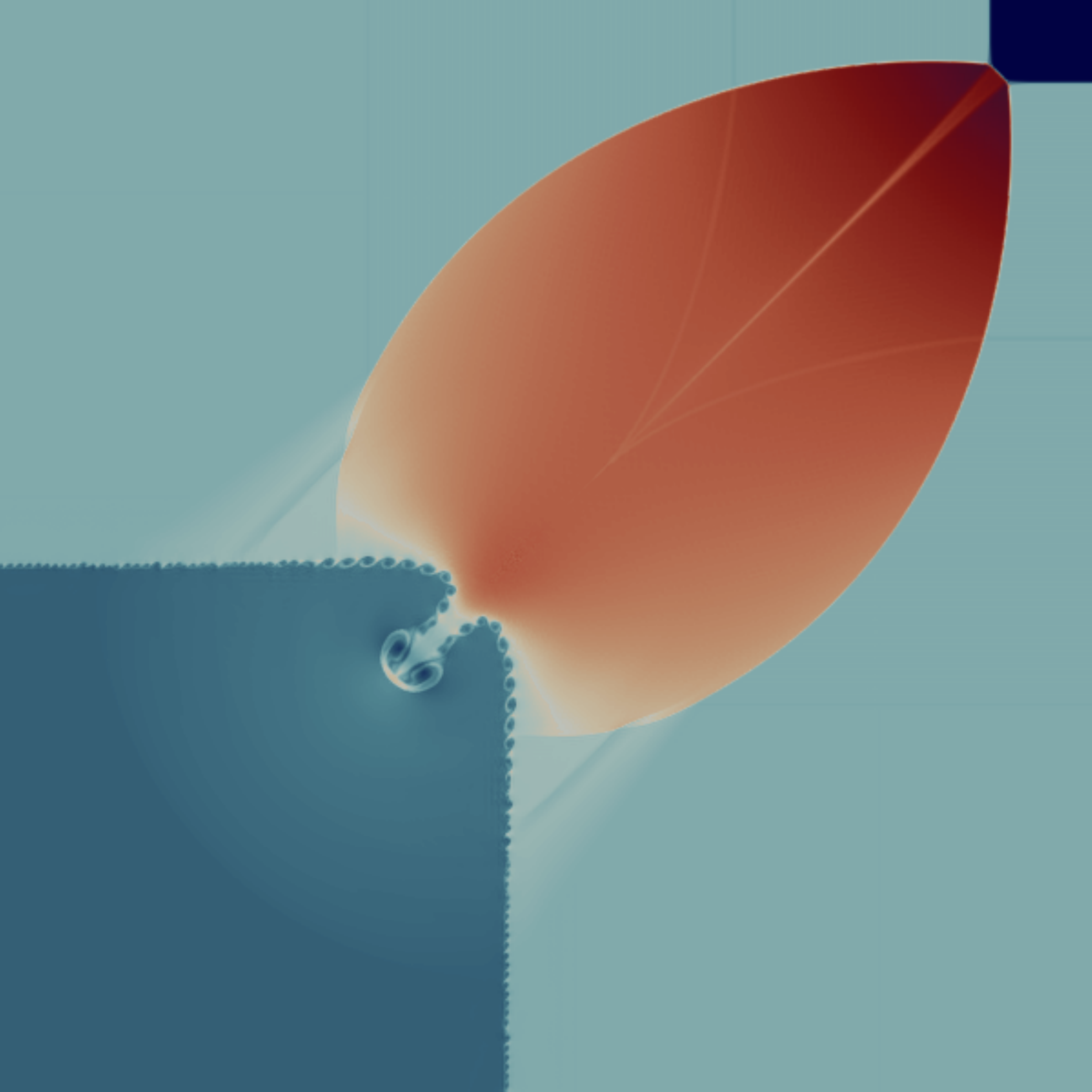} \\
(a) {\tt Trixi.jl} & (b) LW-MH
\end{tabular}
\caption{2-D Riemann problem, density plots of numerical solution at $t=0.25$ for degree $N=4$ on a $256 \times 256$ mesh.}
\label{fig:rp2d}
\end{figure}

\begin{figure}
\centering
\includegraphics[width=0.9\textwidth]{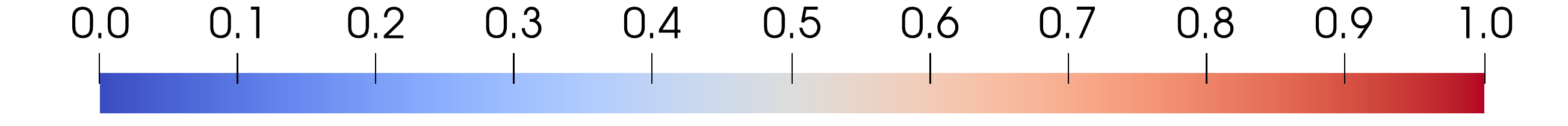}
\begin{tabular}{cc}
\includegraphics[width=0.45\textwidth]{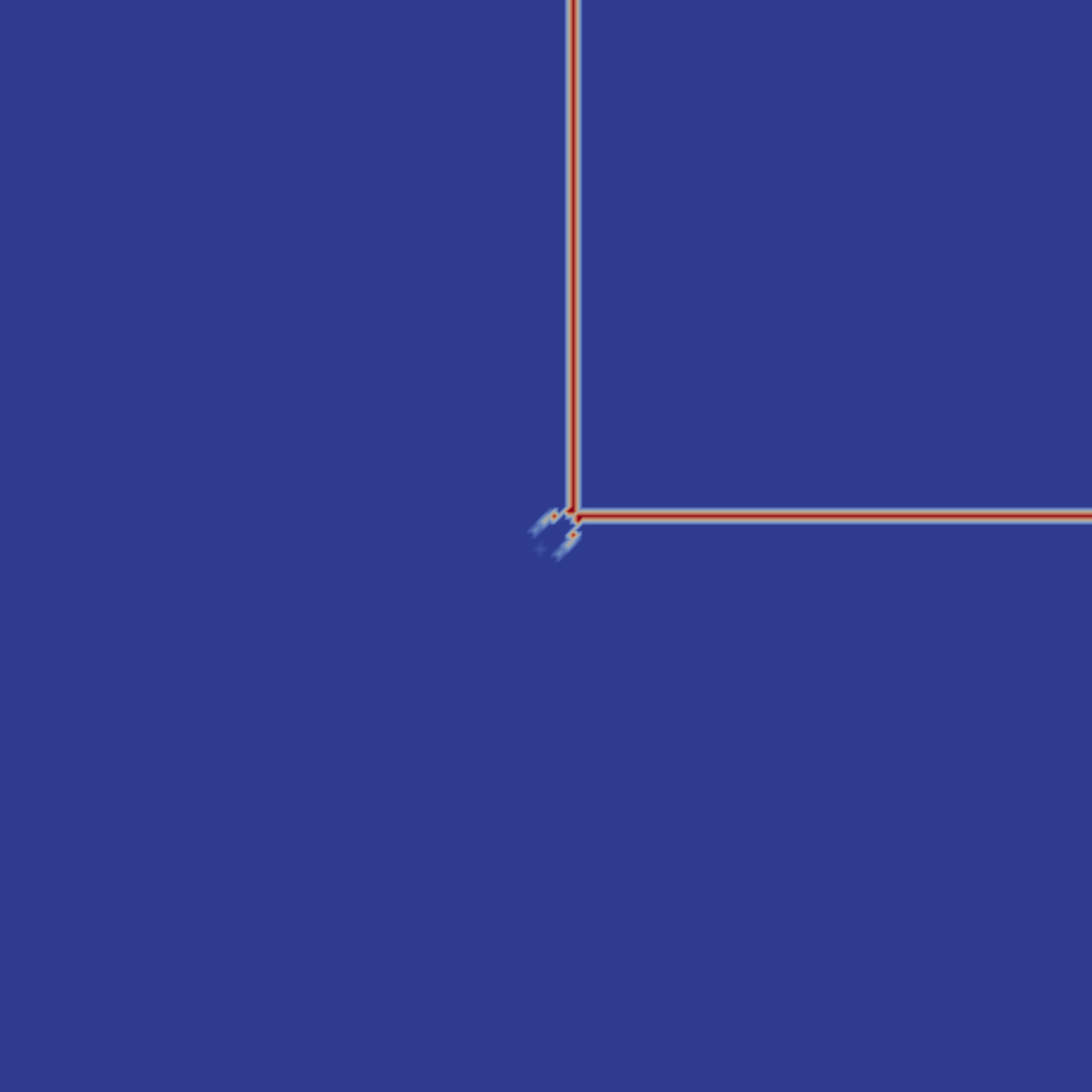} &
\includegraphics[width=0.45\textwidth]{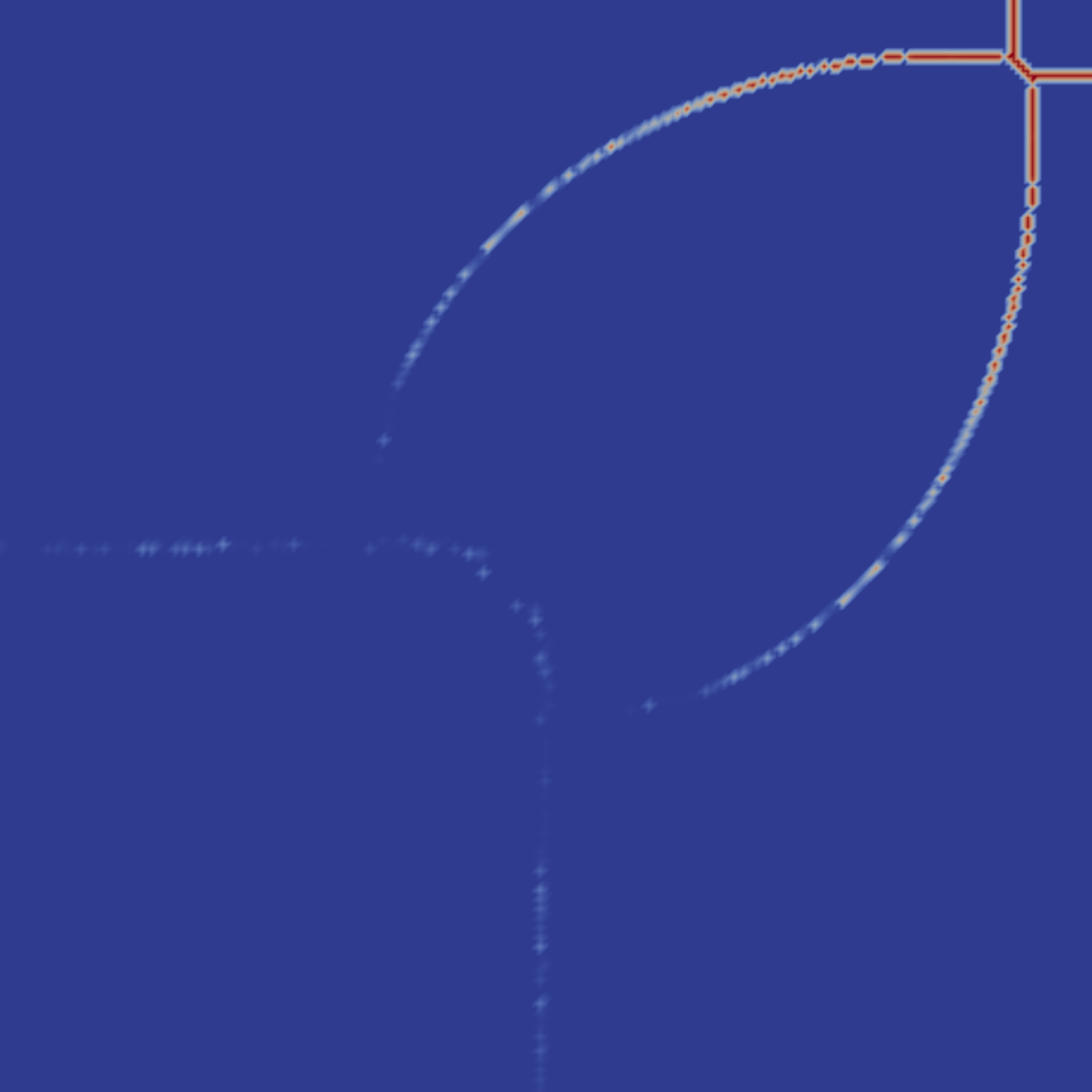} \\
(a) $t=0.025$  & (b) $t=0.25$
\end{tabular}
\caption{\correction{2-D Riemann problem, blending coefficient $\alpha$ for degree $N=4$ on a $256 \times 256$ mesh.}}
\label{fig:rp2d.alpha}
\end{figure}

\begin{figure}
\centering
\includegraphics[width=0.6\textwidth]{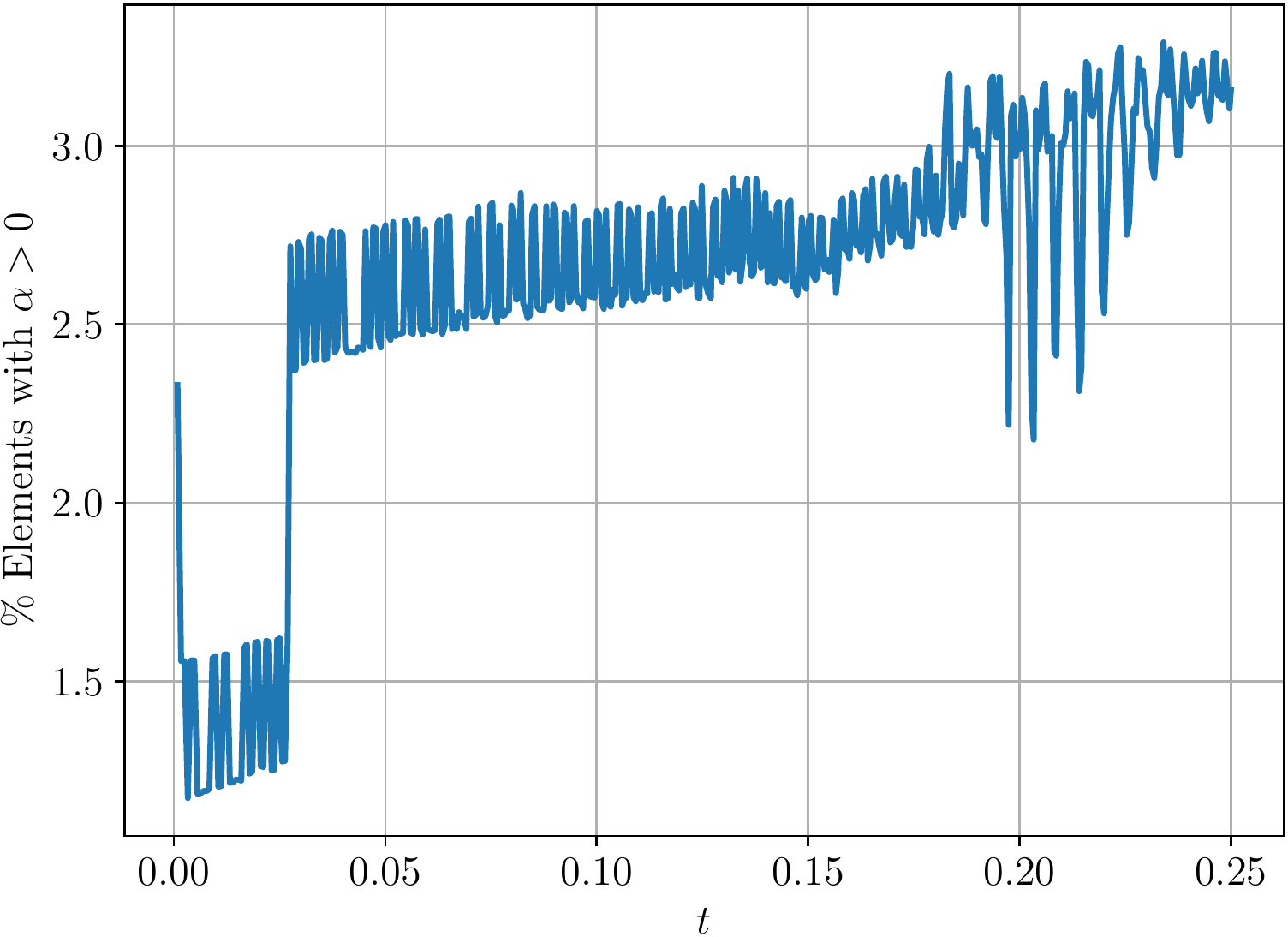}
\caption{\correction{2-D Riemann problem, percentage of elements where the smoothness coefficient $\alpha>0$ vs time $t$, for approximate solution with polynomial degree $N=4$ on a $256 \times 256$ mesh.}}
\label{fig:rp2d.alpha.stats}
\end{figure}

\subsubsection{Double Mach reflection} \label{sec:dmr}

This test case was originally proposed by Woodward and Colella~\cite{Woodward1984} and consists of a shock impinging on a wedge/ramp which is inclined by 30 degrees. The solution consists of a self similar shock structure with two triple points. By a change of coordinates, the situation is simulated in the rectangular domain $\Omega = [0,4] \times [0,1],$ where the wedge/ramp is positioned at $x=1/6, y=0.$ Defining $\mathbf{u}_b = \mathbf{u}_b(x,y,t)$ with primitive variables given by
\begin{equation*}
(\rho,u,v,p)=\begin{cases}
(8, 8.25 \cos\left( \frac{\pi}{6} \right), -8.25 \sin\left( \frac{\pi}{6} \right), 116.5), & \mbox{ if } x < \frac{1}{6} + \frac{y + 20 t}{\sqrt{3}}\\
(1.4, 0, 0, 1), & \mbox{ if } x > \frac{1}{6} + \frac{y + 20 t}{\sqrt{3}}
\end{cases}
\end{equation*}
we define the initial condition to be $\mathbf{u}_0(x,y) = \mathbf{u}_b(x,y,0)$. With $\mathbf{u}_b$, we impose inflow boundary conditions at the left side $\{0\} \times [0,1]$, outflow boundary conditions both at $[0,1/6] \times \{0\}$ and $\{4\} \times [0,1]$, reflecting boundary conditions at $[1/6, 4] \times \{0\}$ and inflow boundary conditions at the upper side $[0,4] \times \{1\}$.

The simulation is run on a mesh of $600 \times 150$ elements using degree $N=4$ polynomials upto time $t=0.2$. In Figure~(\ref{fig:dmr}), we compare the results of {\tt Trixi.jl} with the MUSCL-Hancock blended scheme zoomed near the primary triple point. As expected, the small scale structures are captured better by the MUSCL-Hancock blended scheme.


\begin{figure}
\centering
\begin{tabular}{cc}
\includegraphics[width=0.4\textwidth]{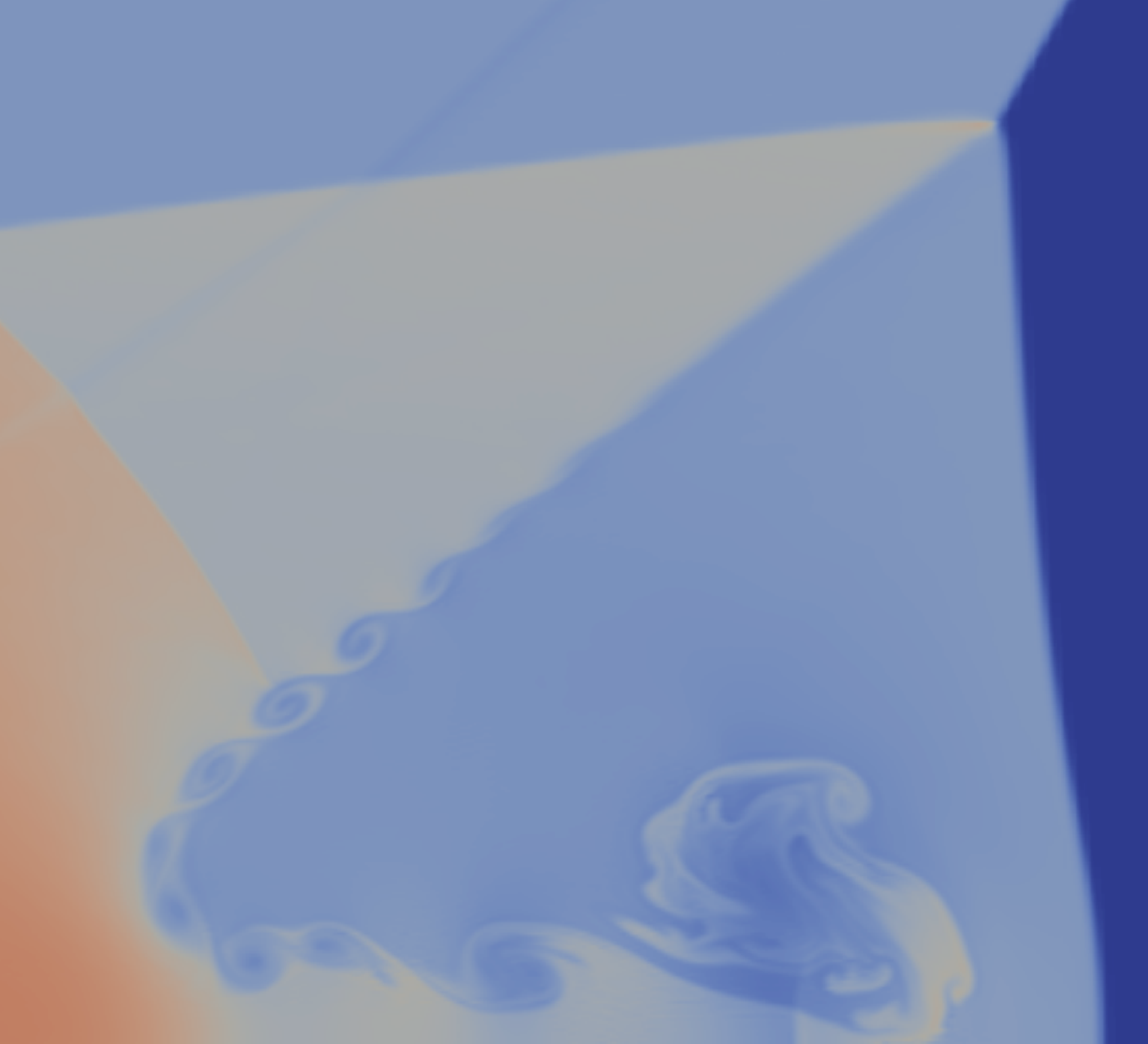} &
\includegraphics[width=0.4\textwidth]{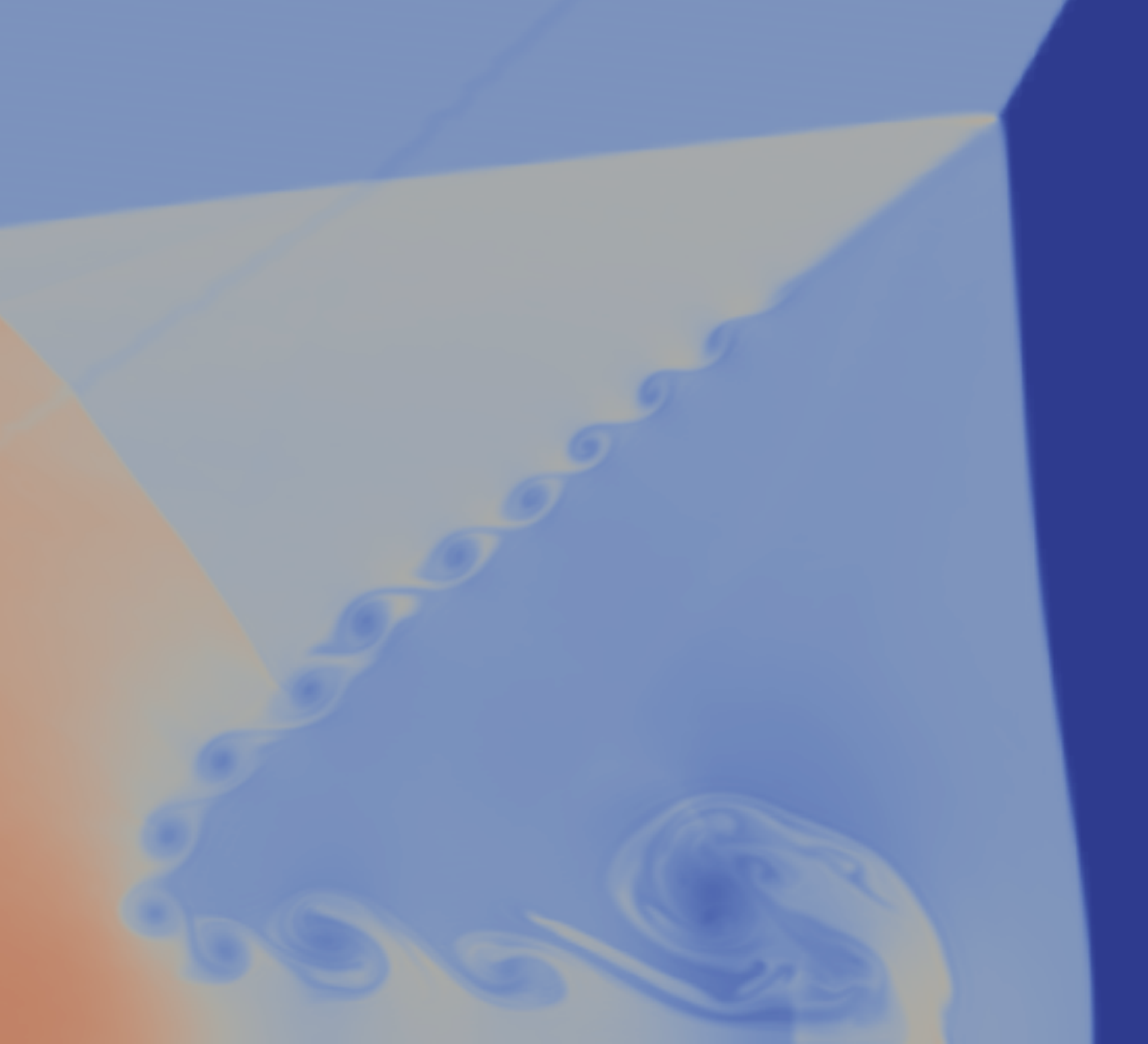} \\
(a) {\tt Trixi.jl} & (b) LW-MH
\end{tabular}
\caption{Double Mach reflection problem, density plots of numerical solution at $t=0.2$ using polynomial degree $N=4$ on a $600 \times 150$ mesh zoomed near the primary triple point. }
\label{fig:dmr}
\end{figure}

\subsubsection{Kelvin-Helmholtz instability}
Fluid instabilities are essential for mixing processes and turbulence production, and play a significant role in many astrophysical phenomena. They are crucial for accurately modeling stripping of gas from satellite galaxies, as well as calculating the expected levels of turbulence and entropy in the intracluster gas of galaxy clusters~\cite{Volker2010}. The Kelvin-Helmholtz instability is a common fluid instability that occurs across contact discontinuities in the presence of a tangential shear flow. This instability leads to the formation of vortices that grow in amplitude and can eventually lead to the onset of turbulence. We adopt the initial conditions for this instability from~\cite{Volker2010} \correction{over the domain $[0,1]^2$},
\begin{eqnarray*}
\rho(x,y) &=& \begin{cases}
2, \quad &\text{if} \quad 0.25<y<0.75\\
1, \quad &\text{otherwise}
\end{cases} \\
u(x,y) &=& \begin{cases}
0.5, \quad &\text{if} \quad 0.25<y<0.75  \\
-0.5, \quad &\text{otherwise},
\end{cases} \\
v(x,y) &=& w_0 \sin(4\pi x)\left \{\exp\left[ -\frac{(y-0.25)^2}{2\sigma^2} \right ] + \exp\left[ -\frac{(y-0.75)^2}{2\sigma^2} \right ]\right \} \\
p(x,y) &=& 2.5
\end{eqnarray*}
with $w_0 = 0.1$, $\sigma = 0.05/\sqrt{2}$ and the adiabatic index $\gamma = 7/5$ corresponding to diatomic gases. The initial conditions consist of a single strongly excited mode in the $y$ velocity concentrated near the interfaces. The wavelength is chosen to be equal to half the domain size so that the single mode dominates the linear growth of instability. This instability leads to shearing and small scale, self-similar vortex structures. We run this test using solution polynomial degree $N=4$ on a mesh of \correction{$512^2$} elements and periodic boundary conditions. We compare the density profiles of {\tt Trixi.jl} and our MUSCL-Hancock blending scheme in Figure~(\ref{fig:kh}). The presence of more vortex structures with the MUSCL-Hancock scheme suggests that the scheme has lesser dissipation errors and is capable of capturing small scale features.

\begin{figure}
\centering
\begin{tabular}{cc}
\includegraphics[width=0.44\textwidth]{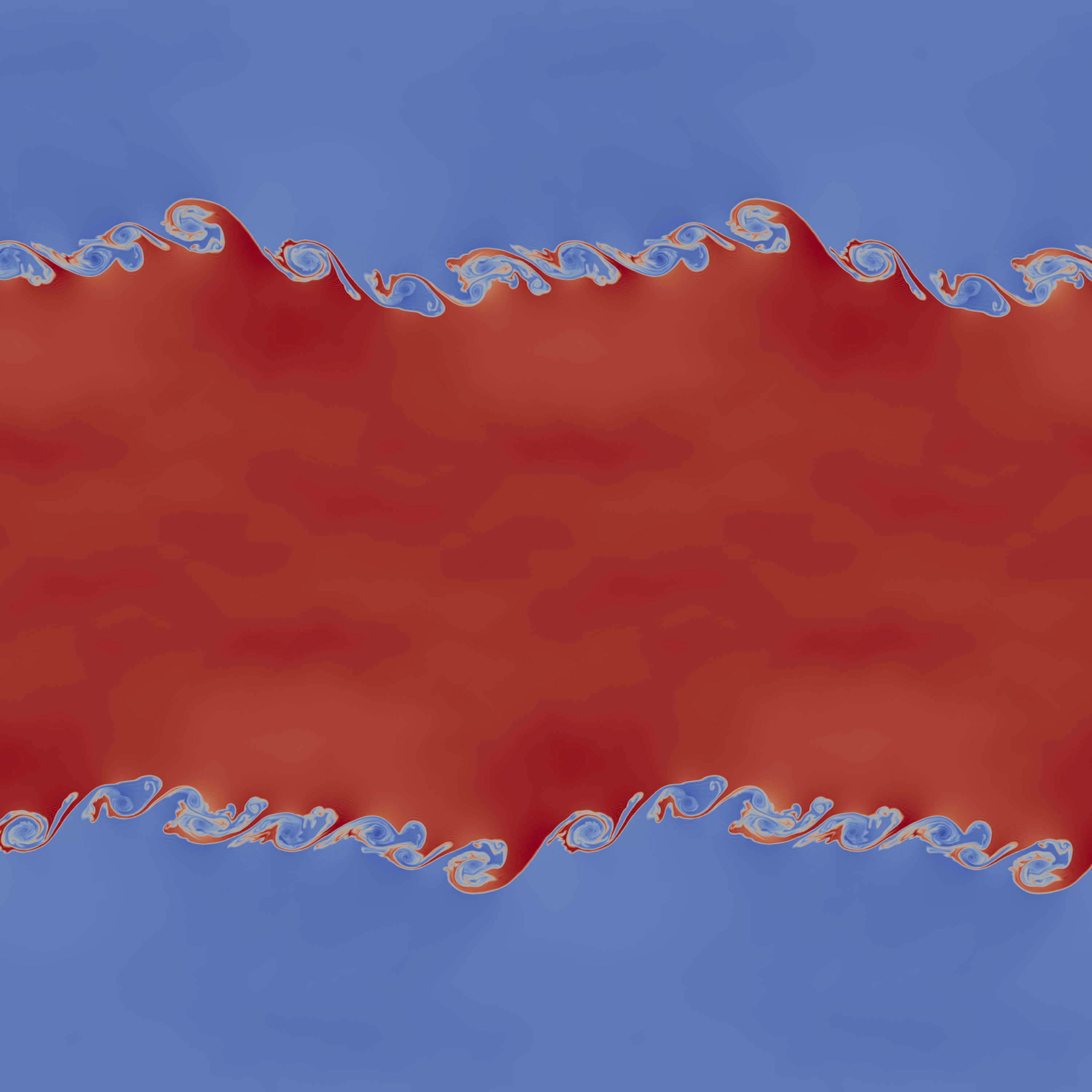} &
\includegraphics[width=0.44\textwidth]{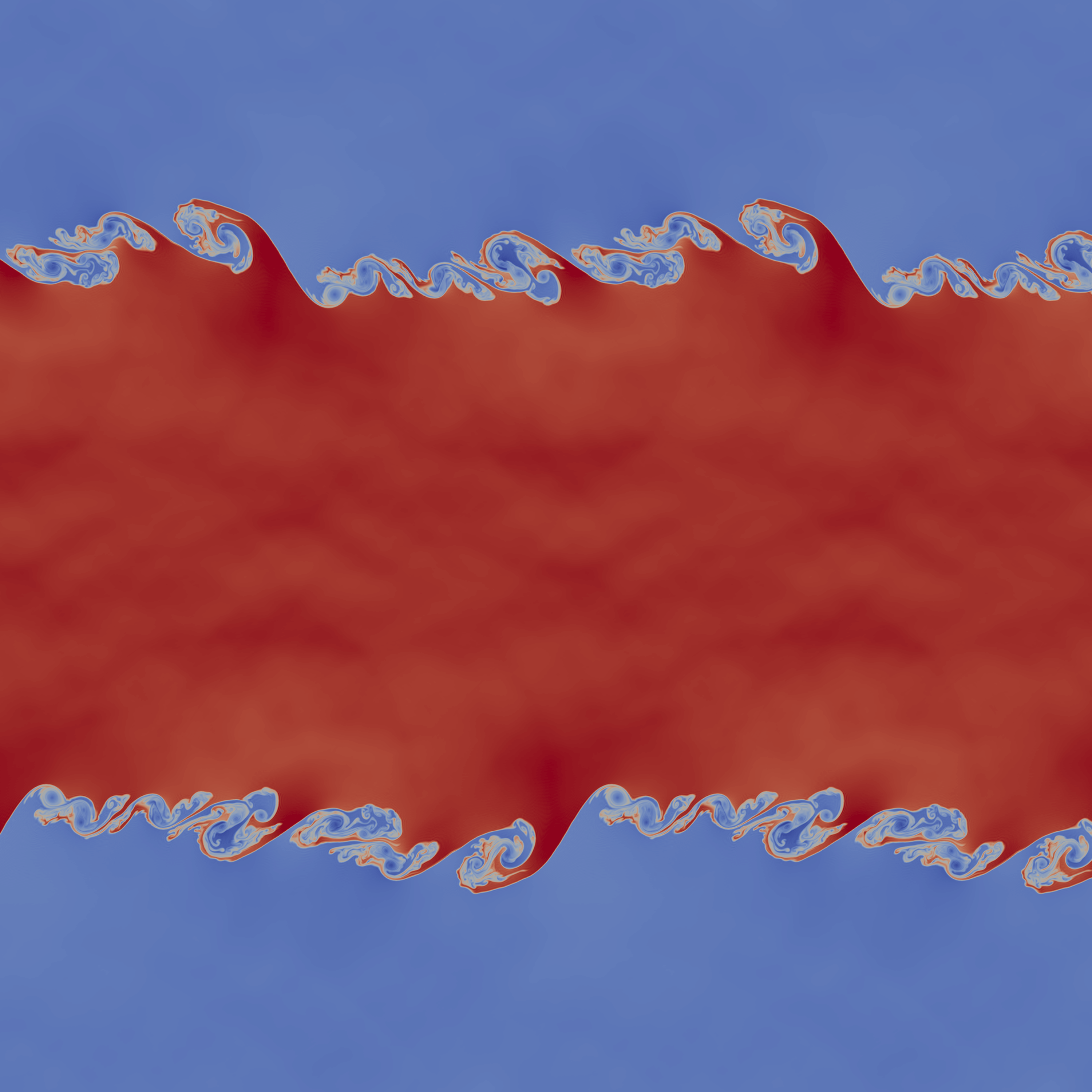} \\
(a) {\tt Trixi.jl} & (b) LW-MH \\
\includegraphics[width=0.44\textwidth]{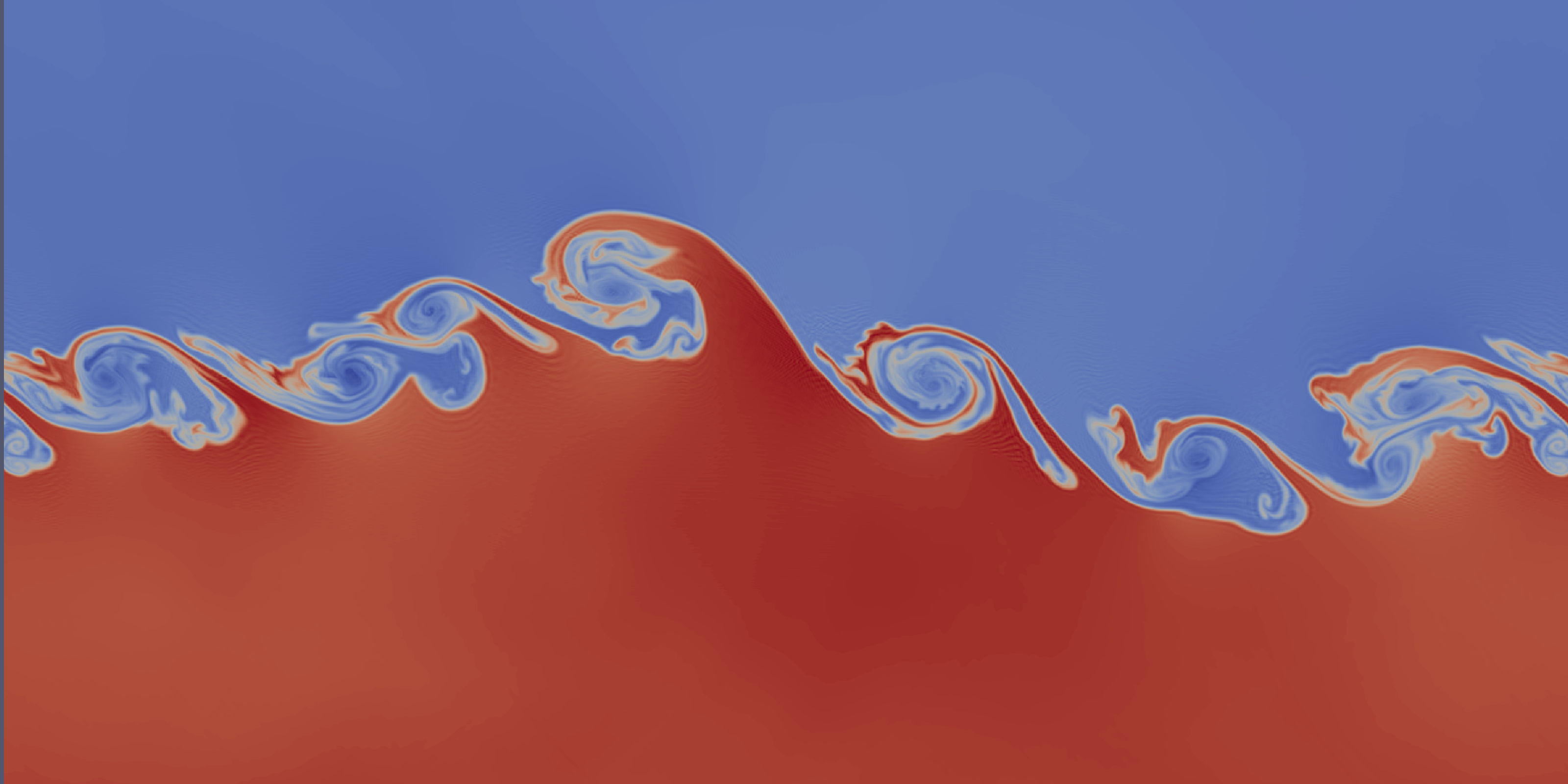} &
\includegraphics[width=0.44\textwidth]{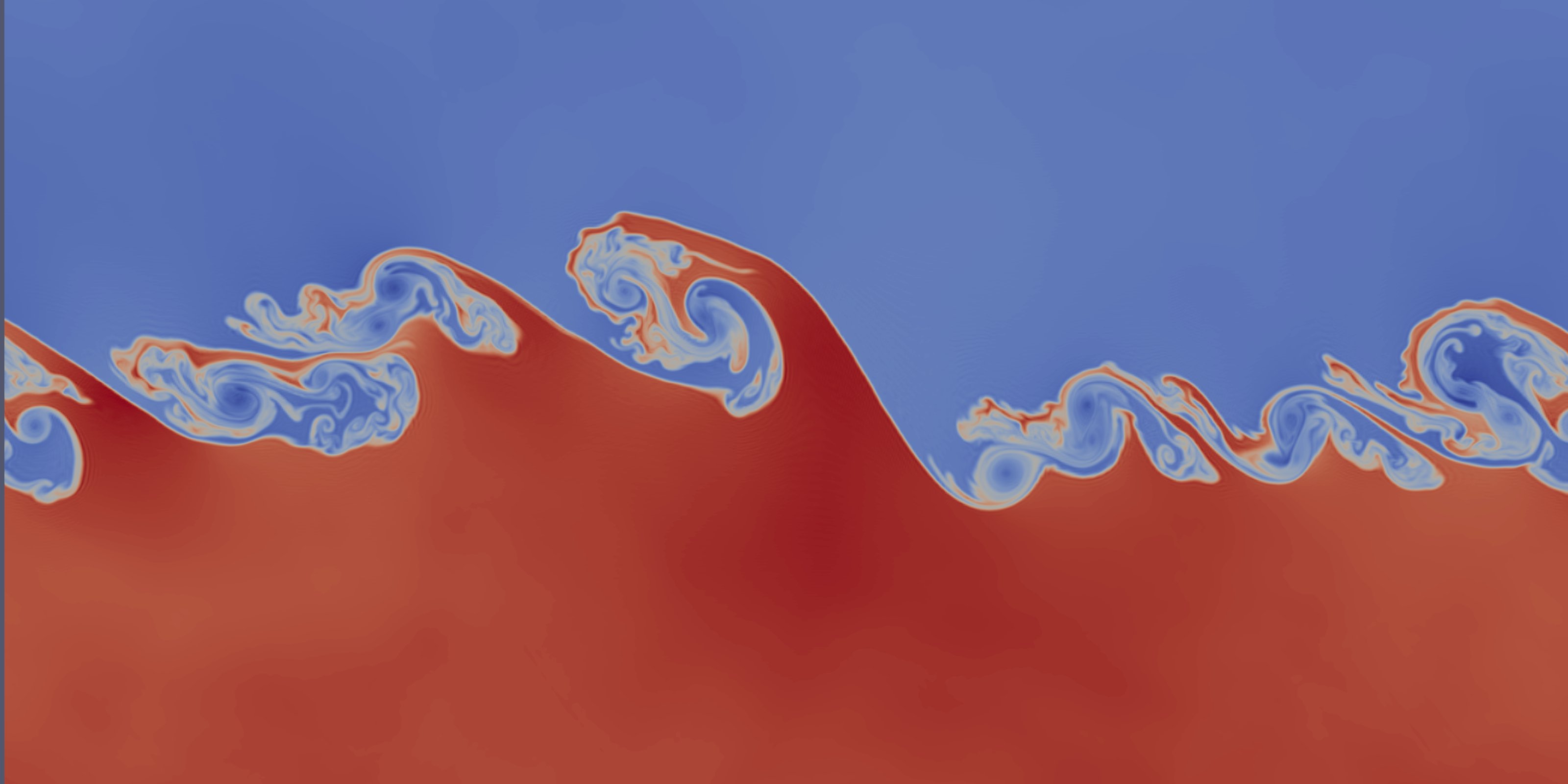} \\
(c) {\tt Trixi.jl} zoomed around top left & (d) LW-MH zoomed around top left
\end{tabular}
\caption{\correction{Kelvin-Helmholtz instability, density plots of numerical solution at $t=0.4$ using polynomial degree $N=4$ with Rusanov flux on a $512^2$ element mesh.}}
\label{fig:kh}
\end{figure}

\subsubsection{Astrophysical jet}

In this test, a hypersonic jet is injected into a uniform medium with a Mach number of 2000 relative to the sound speed in the medium. Following~\cite{ha2005, zhang2010c}, the domain is taken to be $[0,1]\times[-0.5,0.5]$, the ambient gas in the interior has state $\boldsymbol u_{a}$ defined in primitive variables as
\[
(\rho,u, v, p)_a = (0.5, 0, 0, 0.4127)
\]
and inflow state $\boldsymbol u_j$ is defined in primitive variables as
\[
(\rho, u, v, p)_j = (5, 800, 0, 0.4127)
\]
On the left boundary, we impose the boundary conditions
\[
\boldsymbol u_b =
\begin{cases}
\boldsymbol u_a, & \quad \text{if} \quad y \in [-0.05, 0.05]  \\
\boldsymbol u_j, & \quad \text{otherwise}
\end{cases}
\]
and outflow conditions on the right, top and bottom. The HLLC numerical flux was used in the left most cells to distinguish between characteristics entering and exiting the domain. To get better resolution of vortices, we used a smaller time step with $C_s=0.5$ in~\eqref{eq:time.step.2d} and included ghost elements in time step computation to handle the cold start. The high velocity makes the kinetic energy much higher than internal energy. Thus, it is very likely for numerical solvers to give negative pressures. At the same time, a Kelvin-Helmholtz instability arises before the bow shock. Thus, it is a good test both for admissibility preservation and capturing small scale structures. The simulation gives negative pressures if used without the proposed admissibility preservation techniques. While the large scale structures are captured similarly by both the schemes as seen in Figure~\correction{(\ref{fig:astrophysical.jet})}, the LWFR with MH blending scheme shows more small scales near the front of the jet.

\begin{figure}
\centering
\begin{tabular}{cc}
\includegraphics[width=0.45\textwidth]{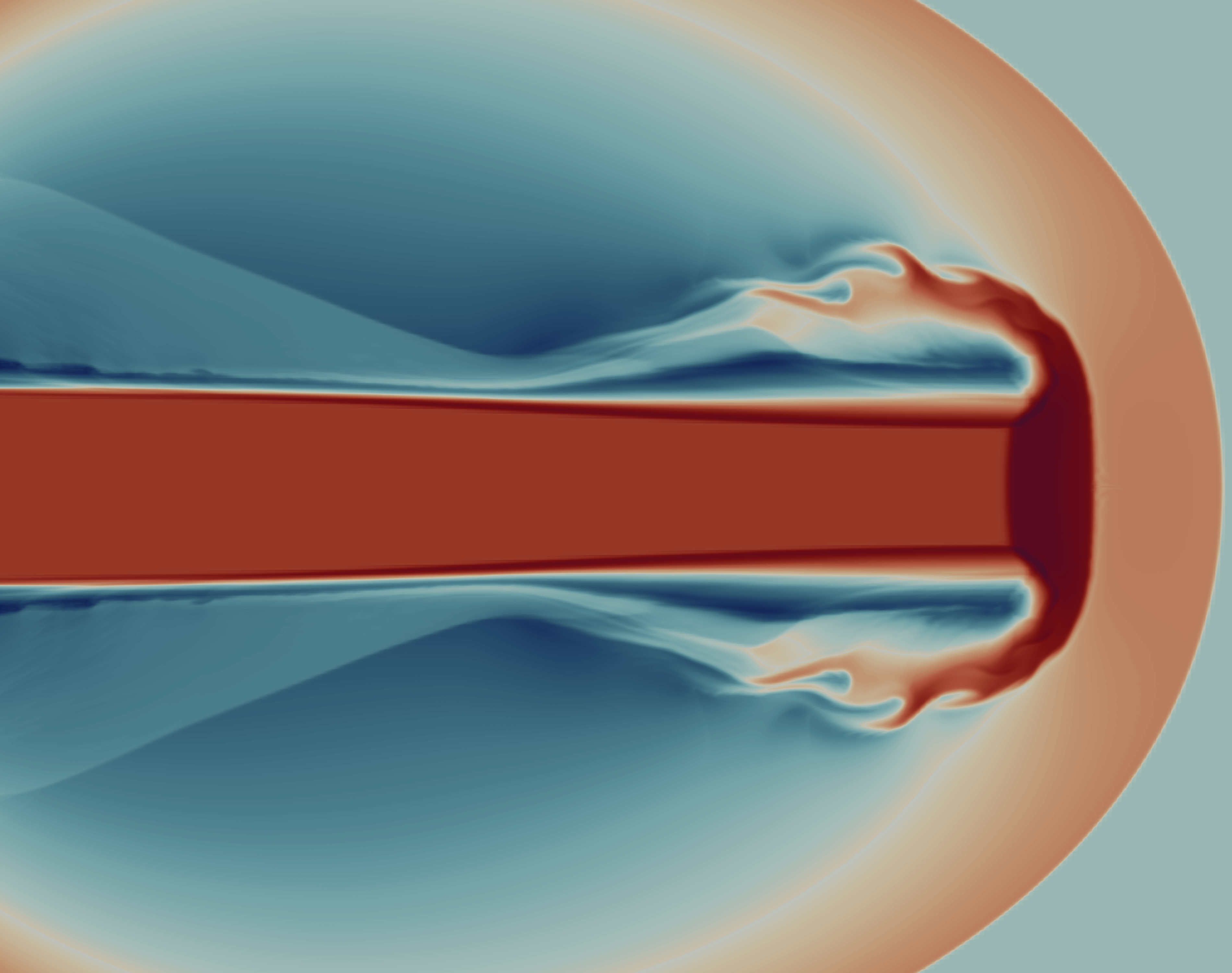} & \qquad \includegraphics[width=0.45\textwidth]{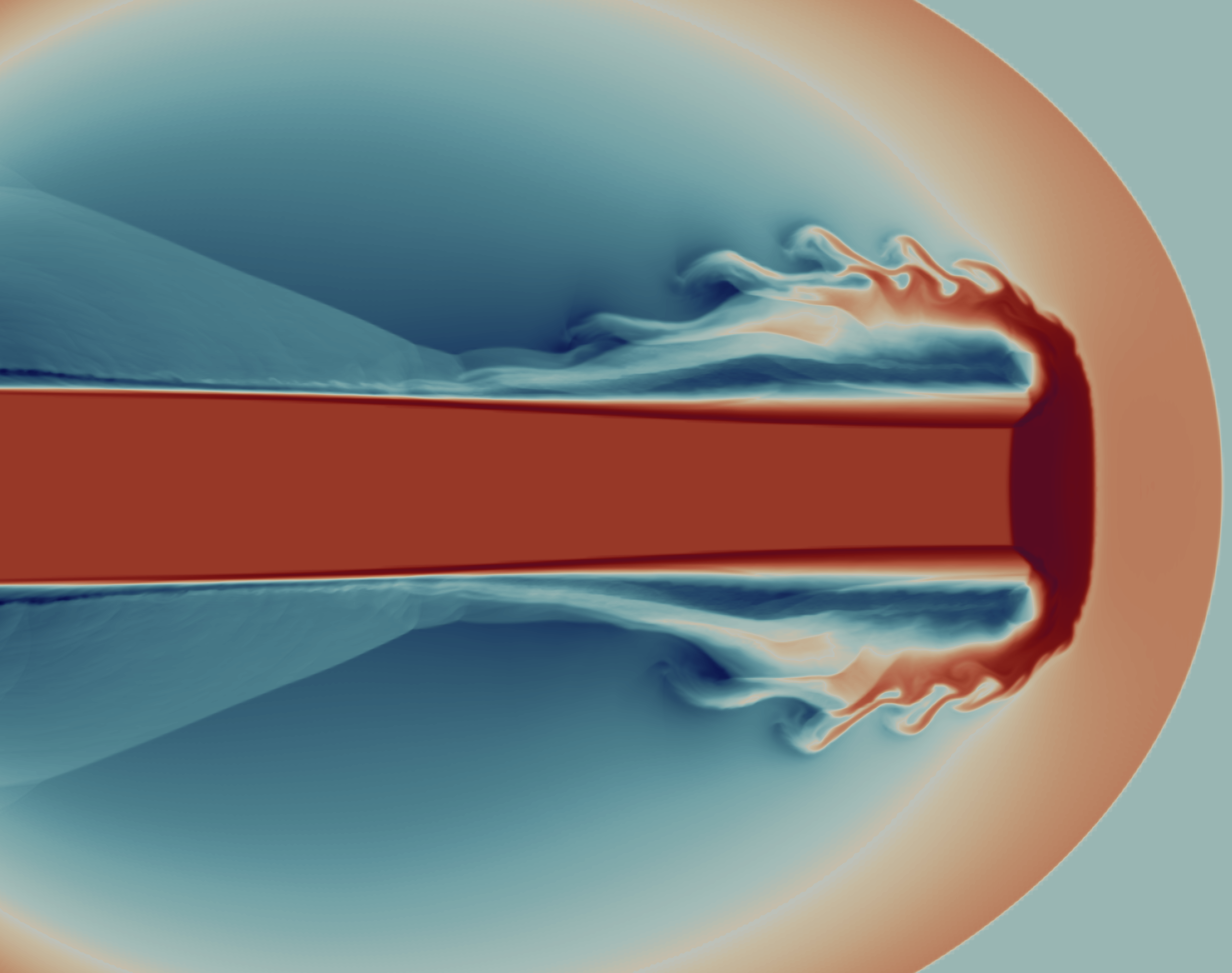} \\
(a) {\tt Trixi.jl} & (b) LW-MH
\end{tabular}
\caption{Mach 2000 astrophysical jet, density plot of numerical solution in log scales on $400 \times 400$ mesh at time $t=0.001$.}
\label{fig:astrophysical.jet}
\end{figure}

\subsubsection{Sedov's blast case with periodic boundary conditions}

Similar to Sedov's blast test in Sections~\ref{sec:sedov.blast.1d}
this test from~\cite{ramirez2021} on domain $[-1.5,1.5]^2$ has energy concentrated at the origin. More precisely, for the initial condition, we assume that the gas is at rest ($u = v = 0$) with Gaussian distribution of density and pressure
\begin{equation}
\rho(x,y) = \rho_0 + \frac{1}{4\pi\sigma_\rho^2} \exp \left( -\frac{r^2}{2\sigma_\rho^2} \right), \qquad p(x,y) = p_0  + \frac{\gamma - 1}{4 \pi \sigma_p^2} \exp\left( -\frac{r^2}{2\sigma_p^2} \right), \qquad r^2 = x^2 + y^2
\end{equation}
where $\sigma_\rho = 0.25$ and $\sigma_p = 0.15$. The ambient density and ambient pressure are set to $\rho_0 = 1$, $p_0 = 10^{-5}$. There are two differences in this Sedov's test compared to the previous - the energy concentrated at the origin is lesser, and domain is assumed to be periodic. When shocks collide at the periodic boundaries, the resulting interaction leads to the formation of small scale structures. In Figure~(\ref{fig:blast.periodic}), we compare the density profiles of the numerical solutions of polynomial degree $N=4$ on a mesh of $64^2$ elements using {\tt Trixi.jl} and the proposed MUSCL-Hancock blending scheme in log scales. Looking at the reference solution on a finer $128^2$ element mesh (Figure~\correction{(\ref{fig:blast.periodic.reference})}), we see that the MUSCL-Hancock scheme resolves the small scale structures better.

\begin{figure}
\centering
\begin{tabular}{cc}
\includegraphics[width=0.45\textwidth]{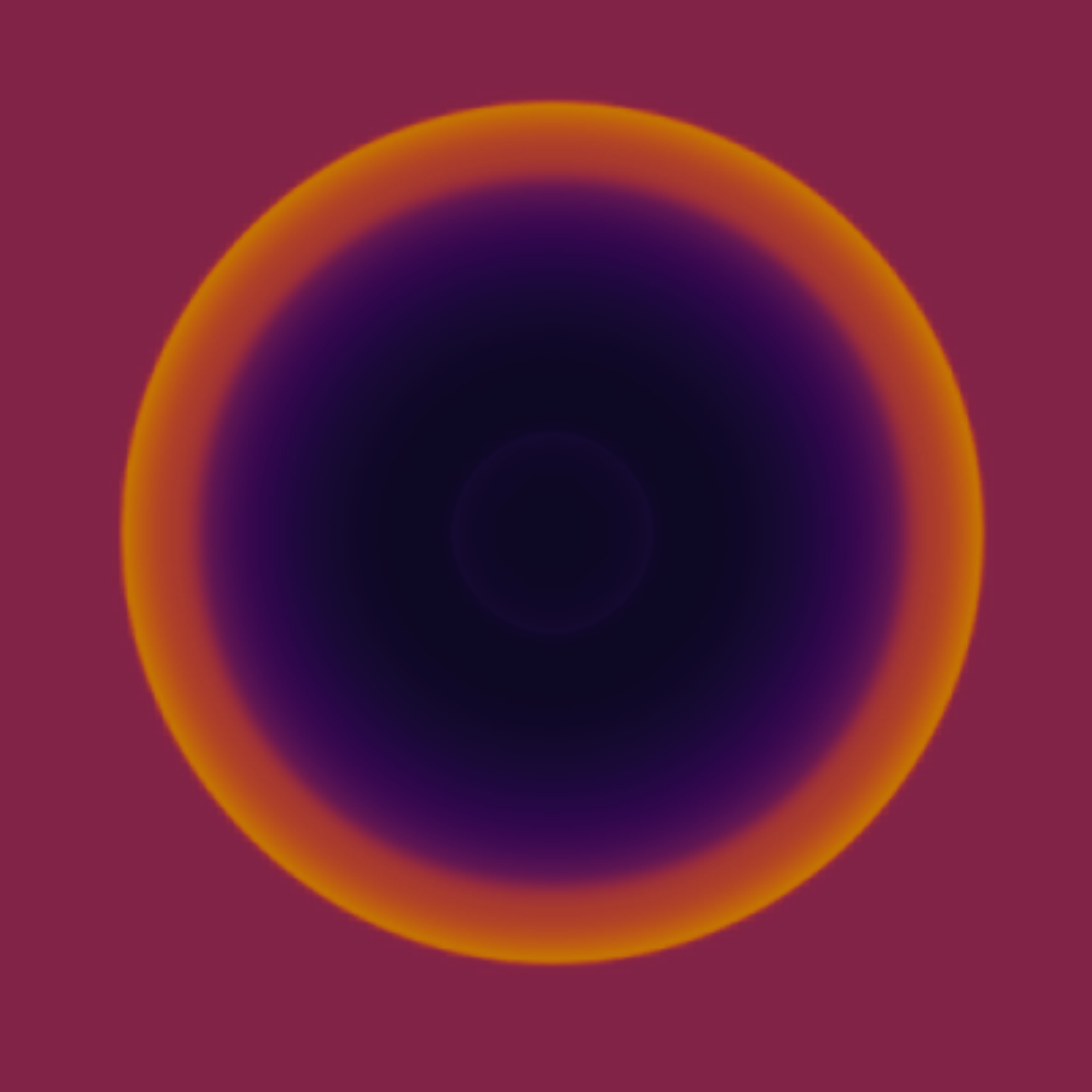} & \qquad \includegraphics[width=0.45\textwidth]{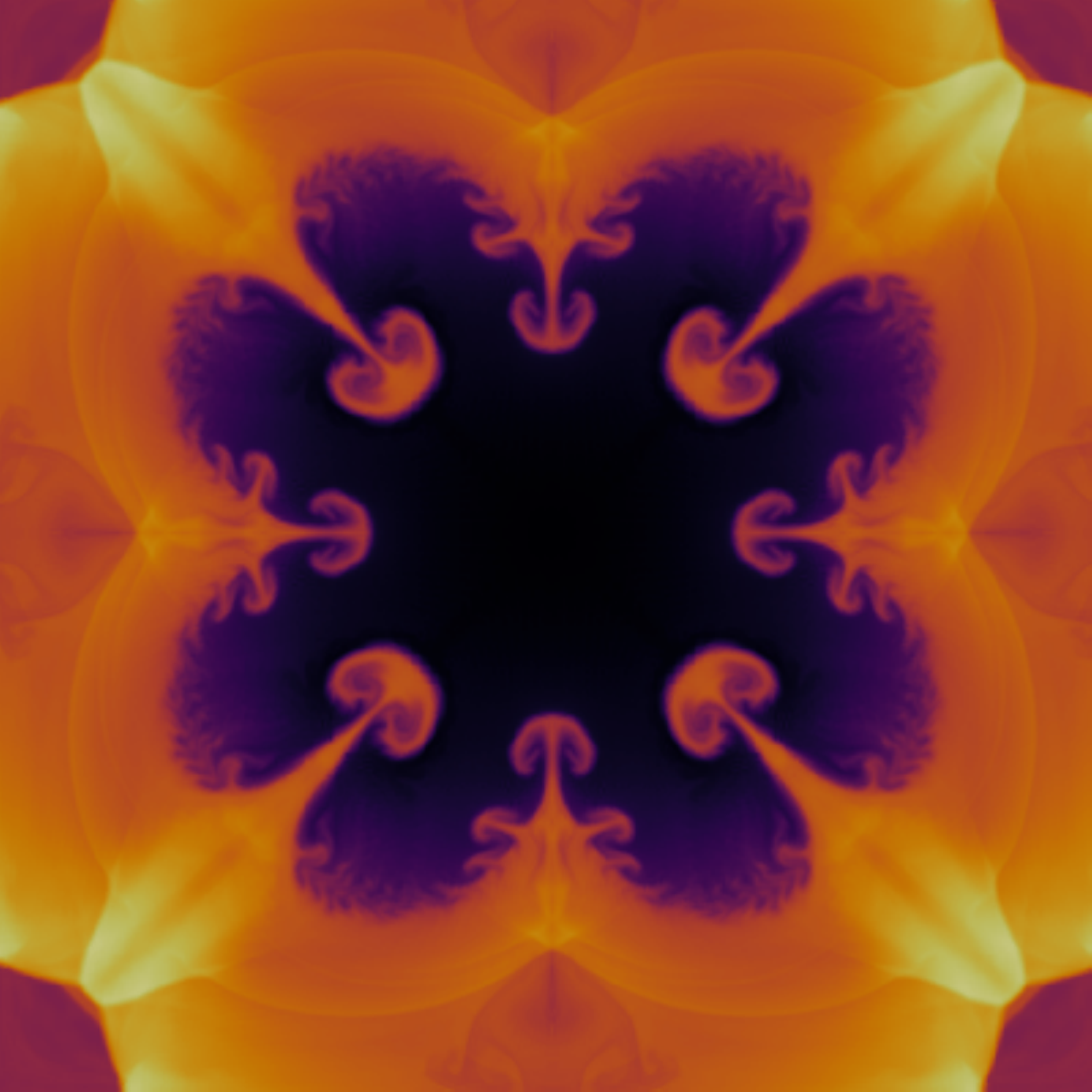} \\
(a) $t=2$ & (b) $t=20$
\end{tabular}
\caption{Sedov's blast test with periodic domain, density plot of numerical solution on $128 \times 128$ mesh in log scales with degree $N=4$ at (a) $t=2$ and (b) $t=20$ with polynomial degree $N=4$ computed using {\tt Trixi.jl}.}
\label{fig:blast.periodic.reference}
\end{figure}

\begin{figure}
\centering
\begin{tabular}{cc}
\includegraphics[width=0.45\textwidth]{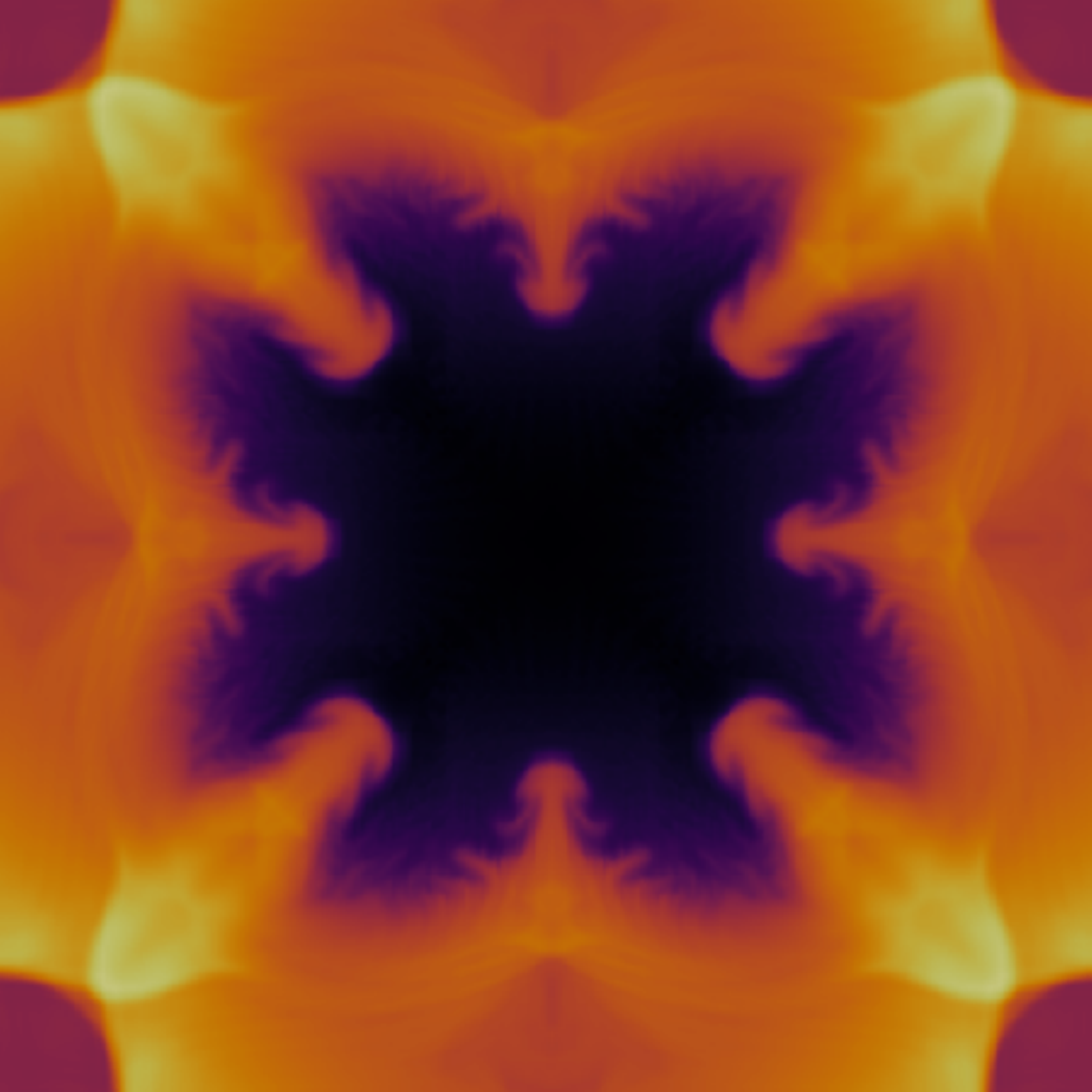} & \qquad \includegraphics[width=0.45\textwidth]{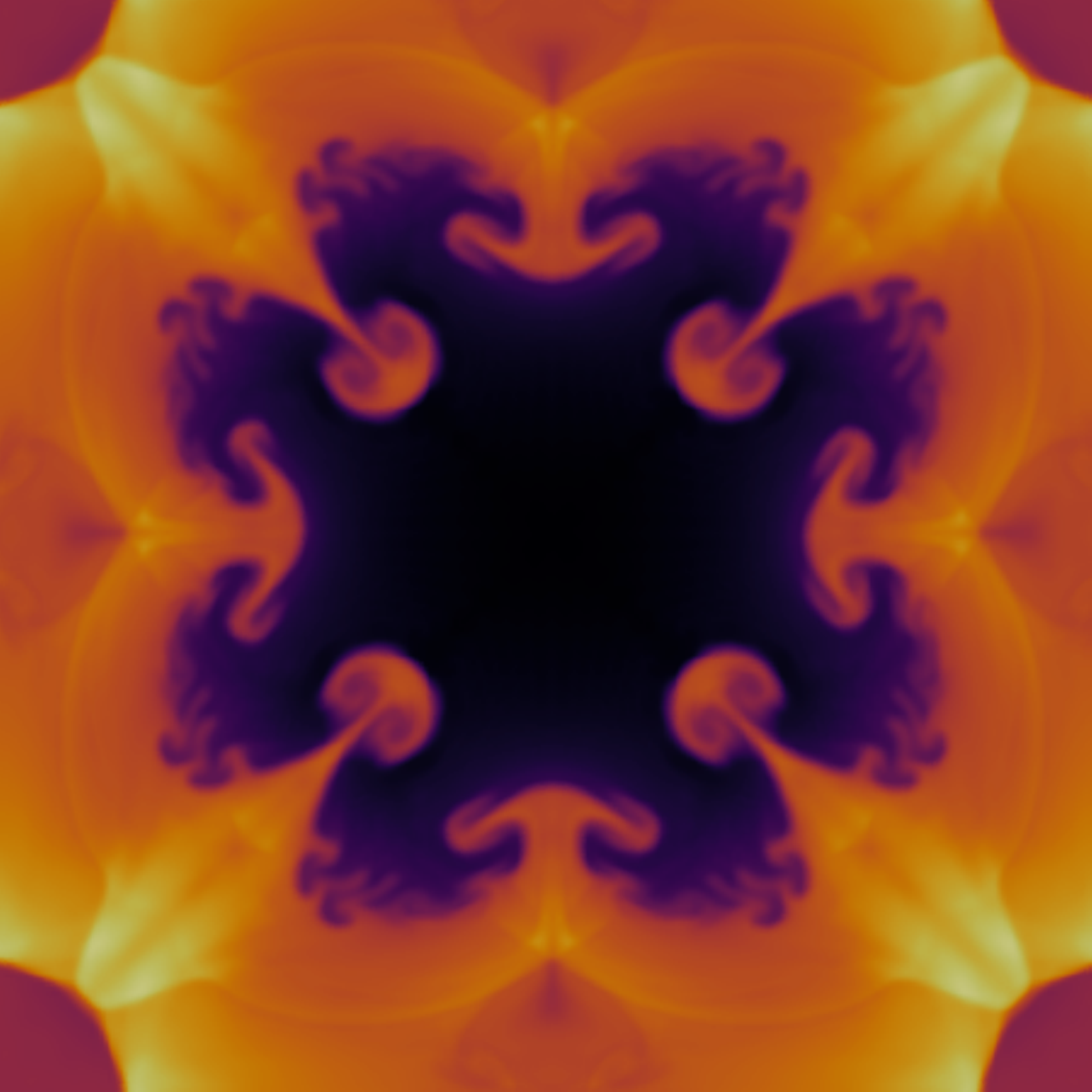} \\
(a) {\tt Trixi.jl} & (b) LW-MH
\end{tabular}
\caption{Sedov's blast test with periodic domain, density plot of numerical solution on $64 \times 64$ mesh in log scales at $t=20$ with degree $N=4$.}
\label{fig:blast.periodic}
\end{figure}

\subsubsection{Detonation shock diffraction}

This test~\cite{Takayama1991} involves a planar detonation wave that interacts with a wedge-shaped corner and diffracts around it, resulting in a complicated wave pattern comprising of transmitted and reflected shocks, as well as rarefaction waves. The computational domain is  $\Omega = [0,2]^2 \backslash ([0,0.5] \times [0,1])$ and following~\cite{henemann2021}, the simulation is performed by taking the initial condition to be a pure right-moving shock with Mach number of 100 initially located at $x=0.5$ and travelling through a channel of resting gas. The post shock states are computed by normal relations~\cite{naca1951}, so that the initial data is
\begin{alignat*}{2}
\rho(x,y) &= \begin{cases}
5.9970, \qquad & \text{if } x \le 0.5\\
1, & \text{if } x > 0.5
\end{cases}, \qquad
&& u(x,y) = \begin{cases}
98.5914, \qquad & \text{if } x \le 0.5 \\
0, & \text{if } x > 0.5
\end{cases} \\
v(x,y) &= 0, \qquad
&& p(x,y) = \begin{cases}
11666.5, \qquad & \text{if } x \le 0.5 \\
1, & \text{if } x > 0.5
\end{cases}
\end{alignat*}

The left boundary is set as inflow and right boundary is set as outflow, all other boundaries are solid walls. The numerical results at $t=0.01$ with polynomial degree $N=4$ on a  Cartesian grid consisting of uniformly sized squares with $\Delta x = \Delta y = 1/200$ are shown in Figure~(\ref{fig:shock.diffraction}). The results look similar to~\cite{henemann2021}; the strong shock makes this a tough test for the admissibility preservation and negative pressure values are obtained if the proposed admissibility correction is not applied.

\begin{figure}
\centering
\begin{tabular}{cc}
\includegraphics[width=0.4\textwidth]{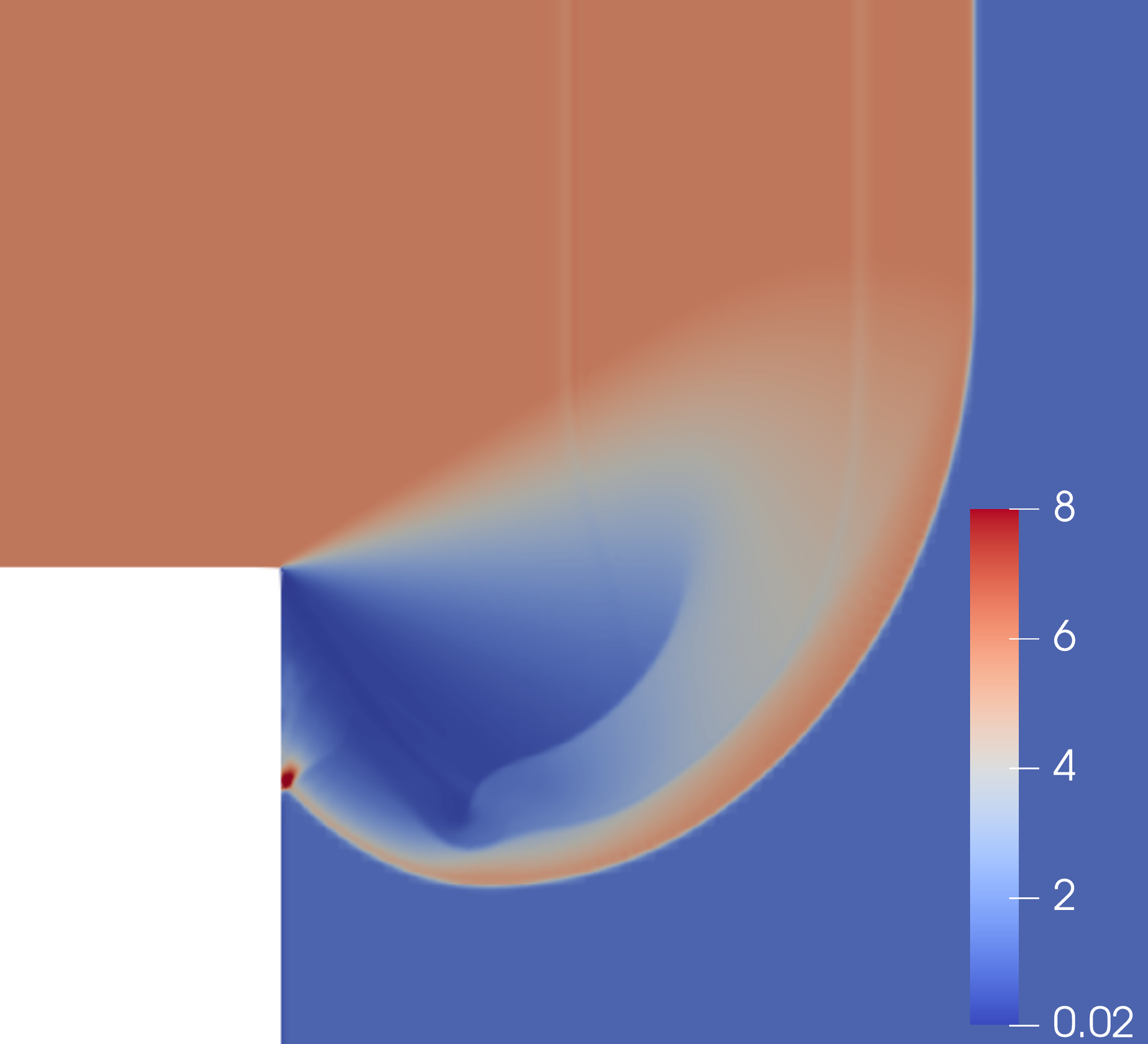} & \qquad \includegraphics[width=0.4\textwidth]{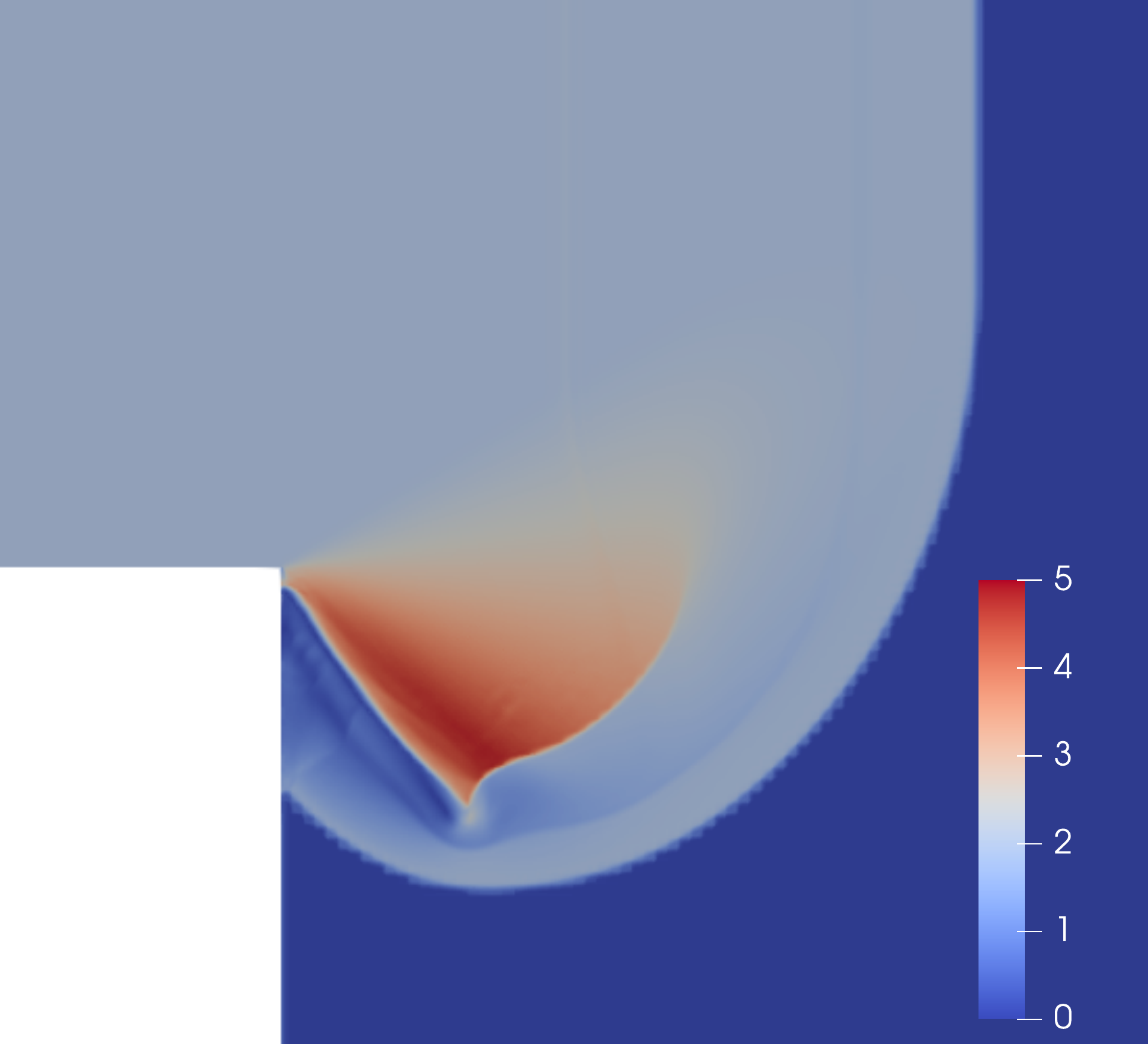} \\
(a) Density & (b) Mach number
\end{tabular}
\caption{Shock diffraction test, numerical solution at time $t=0.01$ with degree $N=4$ \correction{and grid spacing $\Delta x = \Delta y = 1/200$.}}
\label{fig:shock.diffraction}
\end{figure}

\subsubsection{Forward facing step}
Forward facing step is a classical test case from~\cite{emery1968, Woodward1984} where a uniform supersonic flow passes through a channel with a forward facing step generating several phenomena like a strong bow shock, shock reflections and a Kelvin-Helmholtz instability. It is a good test for demonstrating a shock capturing scheme's capability of capturing small scale vortex structures while suppressing spurious oscillations arising from shocks. The step is simulated by taking the domain to be $\Omega = ([0,3] \times [0,1]) \backslash ([0.6,3] \times [0,0.2])$ and the initial conditions are taken to be
\[
(\rho, u, v, p) = (1.4, 3, 0, 1)
\]
The initial conditions are taken to be constant over the whole domain $\Omega$. The left boundary condition is taken as an inflow and the right one is an outflow, the rest are solid walls. The corner $(0.6,0.2)$ of the step is the center of a rarefaction fan and is thus a singular point leading to formation of a spurious boundary layer. The modern treatment of this issue is to use a more refined mesh near the corner point. Since we only have a Cartesian mesh, we obtain the same outcome by forming 1-D meshes in $[0,1], [0,3]$ with the same grid spacing $\Delta x_\text{max}$ away from the singularity and the smaller grid spacing $\Delta x_\text{min} = \frac 14 \Delta x_\text{max} $ in $[0.15, 0.25], [0.45, 0.75]$. Then, the 2-D mesh is formed by taking a tensor product of the two 1-D meshes with cells from $[0.6,3] \times \correction{[0, 0.2]}$ removed. We show the density profile of numerical solutions in \correction{Figure~(\ref{fig:forward.step}a, b, c) for solution polynomial degrees $N=2,3,4$ with $\Delta x_\text{max} = 1/160$}. The scheme captures both the shock and the small scale vortices\correction{, with better resolution of shear structures from the triple shock point near the top wall as the overall resolution is increased}. The corner point singularity causes an artificial boundary layer and Mach stem to occur but these numerical artifacts decrease as we \correction{increase mesh resolution by increasing the polynomial degree}. \correction{Figure~\eqref{fig:ffs.alpha.stats} shows the time evolution of the percentage of cells in the grid where the blending coefficient $\alpha>0$ and Figure~(\ref{fig:forward.step}d) plots the blending coefficient for degree $N=4$ solution at the final time; these figures show that the blending limiter is activated in a small fraction of the cells and only in the vicinity of shocks.}

\begin{figure}
\centering
\begin{tabular}{c}
\includegraphics[width=0.5\textwidth]{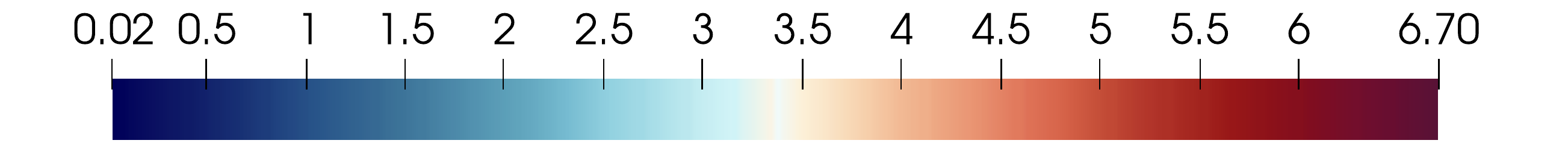} \\
\includegraphics[width=0.6\textwidth]{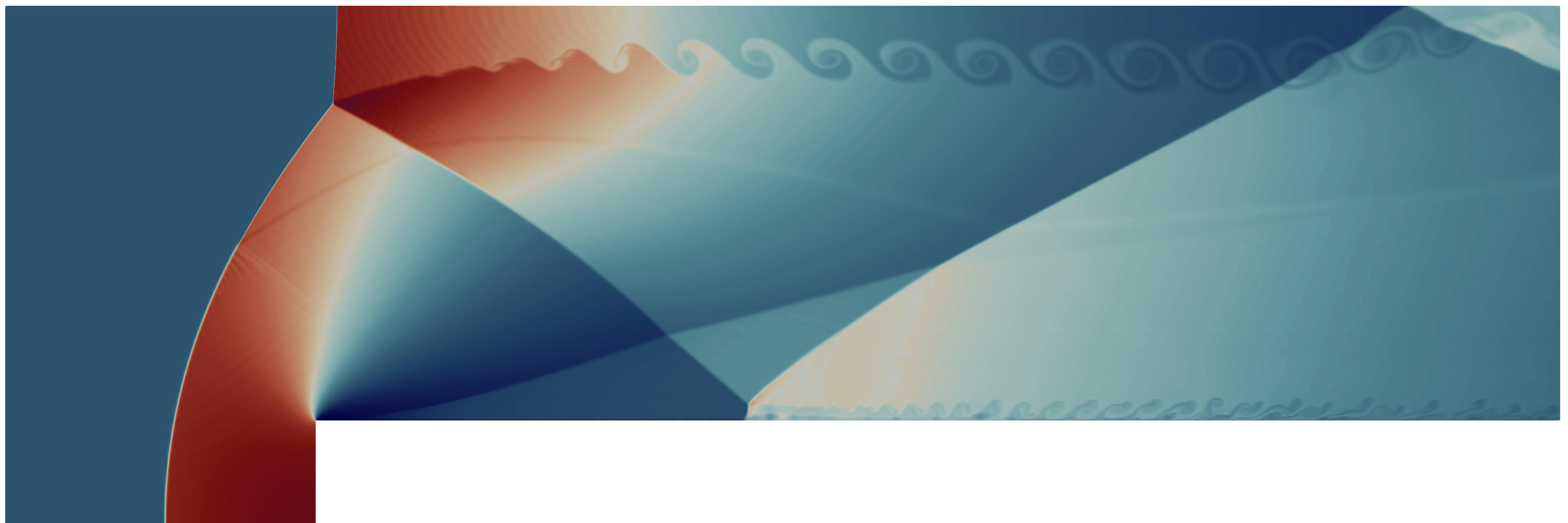} \\
(a) \correction{$N = 2$} \\
\includegraphics[width=0.6\textwidth]{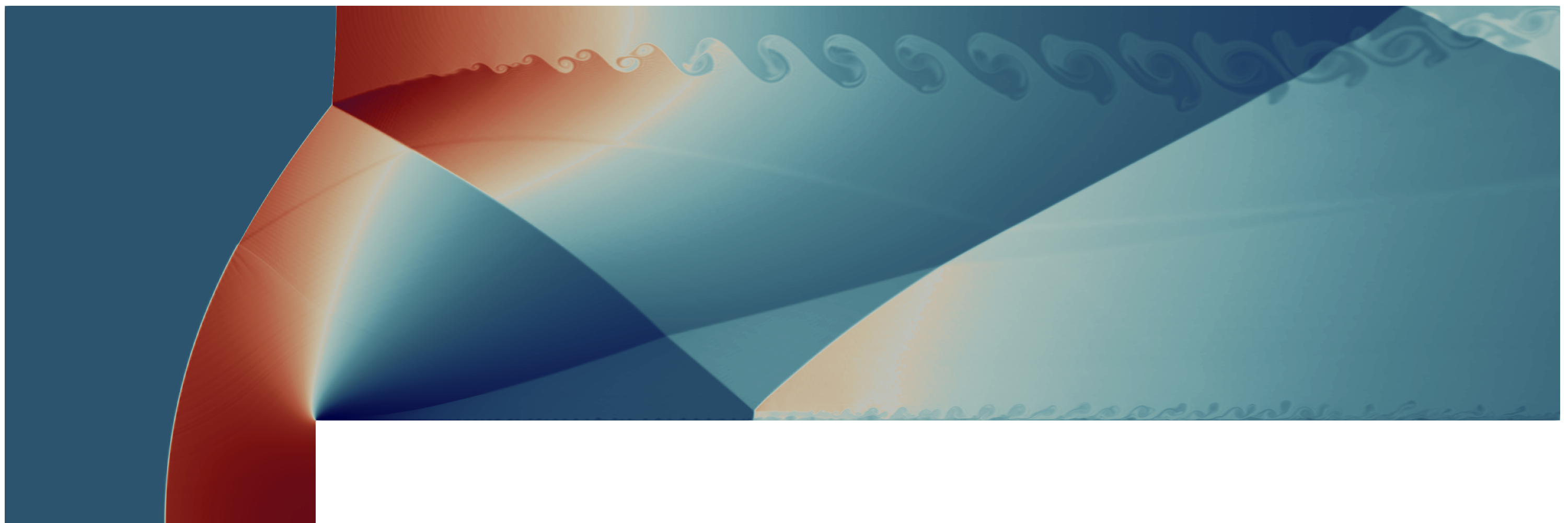} \\
(b) \correction{$N = 3$} \\
\includegraphics[width=0.6\textwidth]{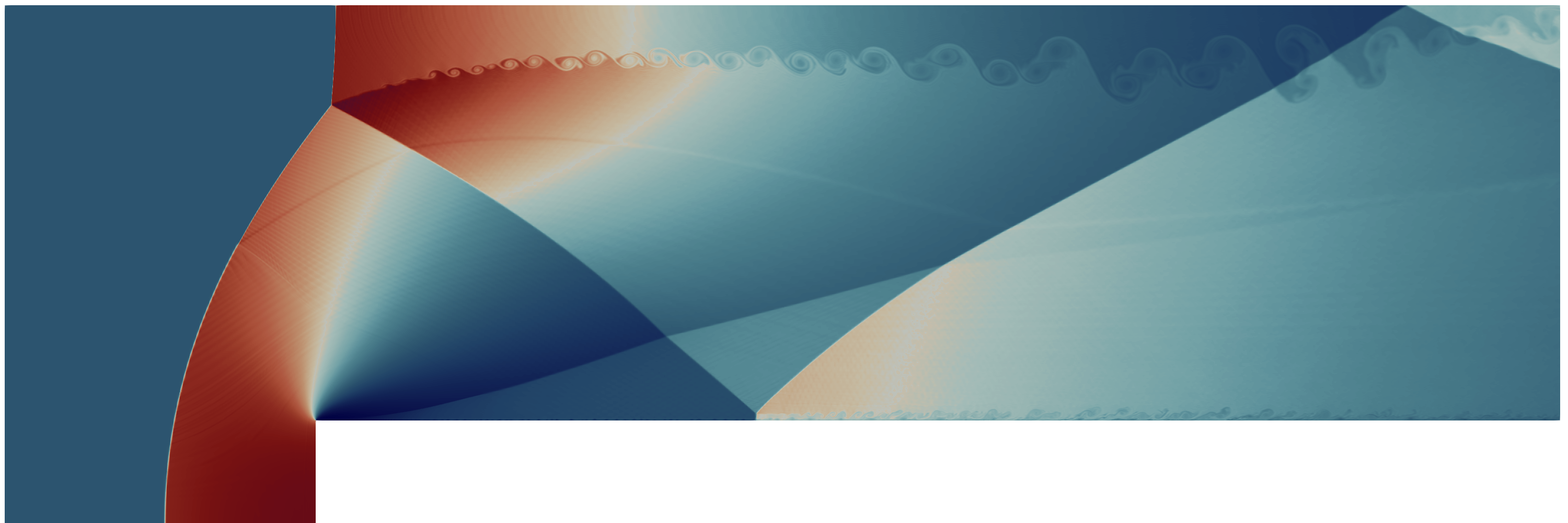} \\
(c) \correction{$N = 4$} \\
\includegraphics[width=0.5\textwidth]{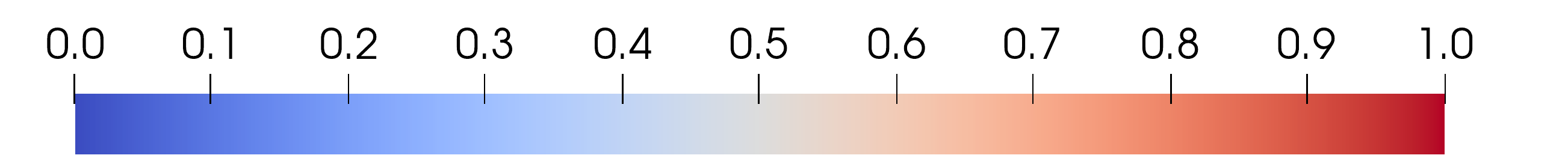} \\
\includegraphics[width=0.6\textwidth]{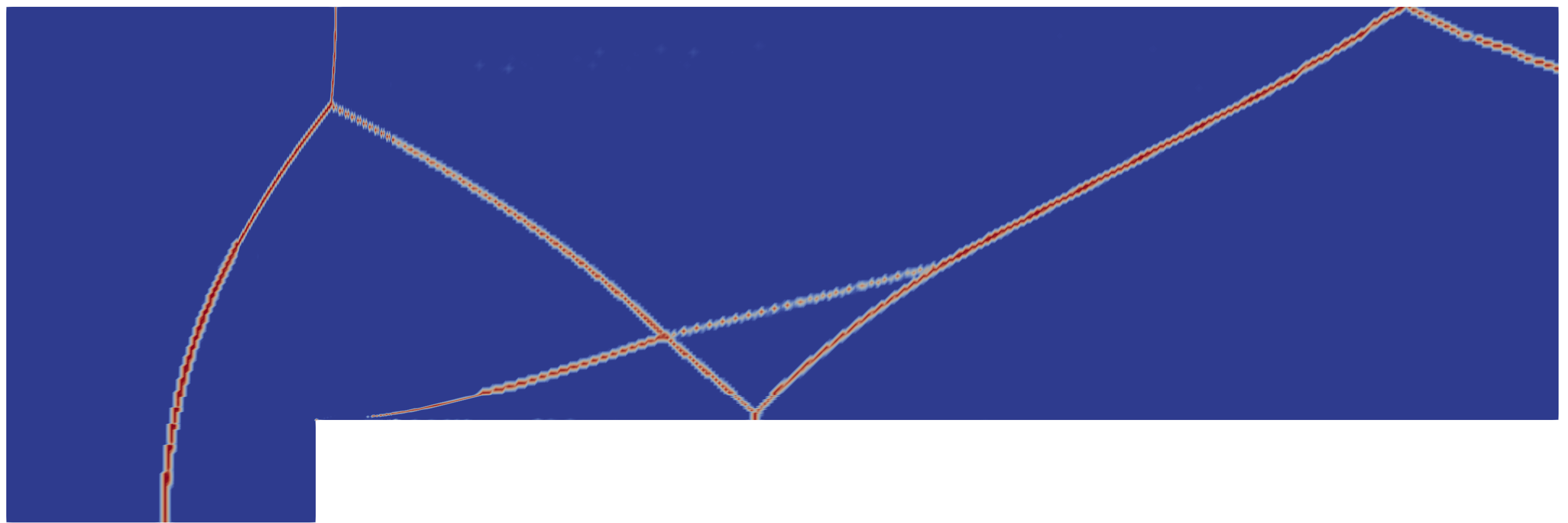} \\
(d) \correction{Blending coefficient $\alpha$ for $N=4$}
\end{tabular}
\caption{Forward facing step, density plots of numerical solution at time $t=3.0$ \correction{with solution polynomial degrees $N=2,3,4$ (a, b, c) and blending coefficient plot for degree $N=4$ (d)}. The meshes are formed by taking grid spacing $\Delta x_\text{max} = \Delta y_\text{max}\correction{ = 1/160}$ away from the corner and smaller grid spacing $\Delta x_\text{min} = \Delta y_\text{min} = \frac{1}{4}\Delta x_\text{max}$ near the corner.}
\label{fig:forward.step}
\end{figure}
\begin{figure}
\centering
\includegraphics[width=0.6\textwidth]{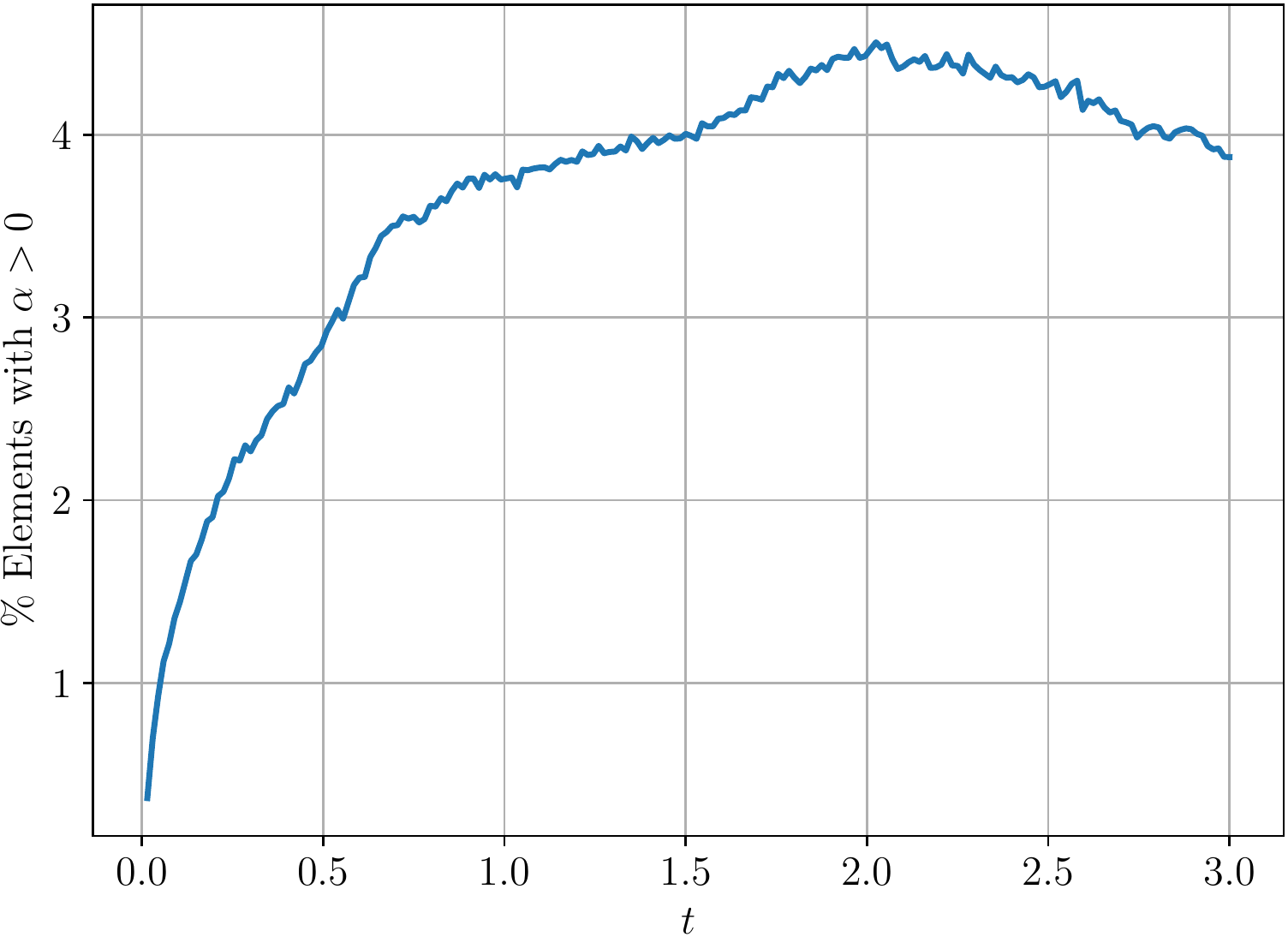}
\caption{\correction{Forward facing step test case, percentage of elements where the smoothness coefficient $\alpha$ is non-zero versus time $t$ for approximate solution with polynomial degree $N=4$ on a mesh with $\Delta x_\text{max} = 1/160$.}}
\label{fig:ffs.alpha.stats}
\end{figure}
\section{Summary and conclusions}\label{sec:sum}
An admissibility preserving subcell-based blending limiter for the \correction{high order} Lax-Wendroff Flux Reconstruction (LWFR) scheme has been constructed by extending the LWFR scheme proposed in~\cite{babbar2022} using the blending limiter of~\cite{henemann2021}. The scheme uses a smoothness indicator to blend two single-stage solvers on the FR grid, one based on the \correction{high order} LWFR method and the other based on a finite volume update on the subcells. At the FR element interfaces, a \textit{blended numerical flux} is constructed using the Lax-Wendroff time averaged flux and lower order numerical flux. The same blended numerical flux is used by both schemes at the element interfaces to maintain conservation. The crucial observation used for obtaining admissibility preservation was that admissibility preservation in means is a consequence of admissibility of the lower order updates. A simple and efficient procedure to obtain admissibility preservation in means was thus proposed, where lower-order updates are made admissible by appropriately constructing the blending numerical flux within the face loop. This approach eliminates the need for additional element or interface loops, minimizing storage requirements. The user only needs to provide the admissibility constraints $\{p_k, k=1,\ldots,K\}$ which are concave functions of the conservative variables and whose positivity implies that the solution is in the admissibility set $\Uad$, making the correction problem-independent. Once admissibility preservation in means is obtained, we use the scaling limiter of~\cite{zhang2010c} to enforce admissibility of the polynomial values. To enhance accuracy, we modified the blending scheme of~\cite{henemann2021} to use Gauss-Legendre solution points and used the second-order MUSCL-Hancock scheme to compute the lower-order residual. We extended the slope restriction criterion of~\cite{Berthon2006} for admissibility of the MUSCL-Hancock scheme to non-cell-centered grids that arise in the blending scheme to maintain the conservation property. We also proposed a problem-independent procedure to enforce the slope restriction. The scheme is robust and the higher resolution of MUSCL-Hancock is more superior in capturing small scale structures, as was demonstrated by numerical experiments on compressible Euler equations.
\section*{Acknowledgments}
The work of Arpit Babbar and Praveen Chandrashekar is supported by the Department of Atomic Energy,  Government of India, under project no.~12-R\&D-TFR-5.01-0520. The work of Sudarshan Kumar Kenettinkara is supported by the  Science and Engineering Research Board, Government of India, under MATRICS project no.~MTR/2017/000649.

\section*{Data availability}
The code and data used to produce the results in this paper will be made publicly available at~\cite{paperrepo,tenkai} after publication of the paper.
\appendix
\section{Admissibility of MUSCL-Hancock scheme for general grids}\label{sec:muscl.admissibility.proof}

For the conservation law~\eqref{eq:con.law}, define $\sigma \left( \bw_1, \bw_2 \right)$ as
\[
\sigma \left( \bw_1, \bw_2 \right) = \max \{ \rho(\ff'( \bw_\lambda)) : \bw_\lambda = \lambda \bw_1 + (1-\lambda) \bw_2, \quad 0 \le \lambda \le 1\}
\]
where $\rho(A)$ denotes the spectral radius of matrix $A$. For the 2-D hyperbolic conservation law
\begin{equation}\label{eq:2d.hyp.con.law}
\bw_t +  \ff_x +  \fg_y = 0
\end{equation}
where $(f,g)$ are Cartesian components of the flux vector; the wave speed estimates in $x,y$ directions are defined as follows
\begin{align*}
\sigma_x \left( \bw_1, \bw_2 \right) = \max \{ \rho(\ff'( \bw_\lambda)) : \bw_\lambda = \lambda \bw_1 + (1-\lambda) \bw_2, \quad 0 \le \lambda \le 1\} \\
\sigma_y \left( \bw_1, \bw_2 \right) = \max \{ \rho(\fg'( \bw_\lambda)) : \bw_\lambda = \lambda \bw_1 + (1-\lambda) \bw_2, \quad 0 \le \lambda \le 1\}
\end{align*}
We assume that the \correction{admissibility set} $\Uad$ 
of the conservation law is a convex subset of $\re^d$ which can be written as~\eqref{eq:uad.form}. The following assumption is made concerning the admissibility of first order finite volume scheme.

\paragraph{Admissibility of first order finite volume scheme.} Under the time step restriction
\begin{equation} \label{eq:numflux.admissibility.cfl}
\correction{\max_{j} \frac{\Delta t}{\Delta x_j} \sigma(\bw_j, \bw_{j+1})} \le 1
\end{equation}
the first order finite volume method
\[
\bw_j^{n+1} = \bw_j^n -\frac{\Delta t}{\Delta x_j}\left ( \ff(\bw_j^n,\bw_{j+1}^n) - \ff(\bw_{j-1}^n, \bw_j^n) \right )
\]
is admissibility preserving, i.e., $\bw_j^n \in \Uad$ for all $j$ implies that $\bw_j^{n+1} \in \Uad$ for all $j$.

\subsection{Review of MUSCL-Hancock scheme}
Here we review the MUSCL-Hancock scheme for general uniform grids that need not be cell-centered (Figure~(\ref{fig:general.grid})) in the sense that
\begin{equation} \label{eq:non.cell.centred.defn}
x_\jph - x_j \ne x_j - x_\jmh,
\end{equation}
for some $j$ where $x_j$ is the solution point in finite volume element $(x_\jmh, x_\jph)$. The grid used in the blending limiter (Figure~\ref{fig:subcells}) is a special case of~\eqref{eq:non.cell.centred.defn}.
\begin{figure}
\begin{center}
\includegraphics[width=0.7\textwidth]{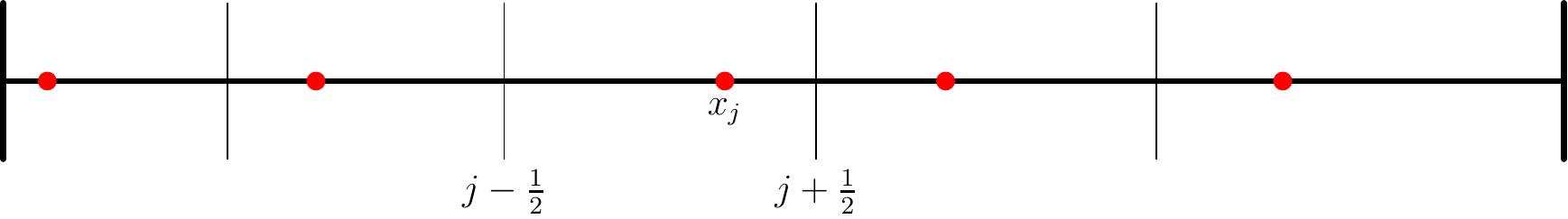}
\caption{Non-uniform, non-cell-centered finite volume grid}
\label{fig:general.grid}
\end{center}
\end{figure}

For the $j^{\text{th}}$ finite volume element $(x_{j - \frac{1}{2}}, x_{j + \frac{1}{2}})$, the constant state is denoted $\bw_j^n$ and the linear approximation will be denoted $r^n_j (x)$. For conservative reconstruction, the linear reconstruction is given by
\[
r^n (x) = \bw_j^n + (x - x_j) \slope_j,
\qquad
x \in \left( x_{j - \frac{1}{2}}, x_{j + \frac{1}{2}} \right)
\]
The values on left and right faces will be computed as
\begin{equation}
\label{eq:reconstruction.general}
\bw^{n, -}_j = \bw_j^n + (x_{\jmh} - x_j)
\slope_j, \qquad \bw_j^{n, +} = \bw_j + (x_{j +
\half} - x_j) \slope_j
\end{equation}
We use Taylor's expansion to evolve the solution to $t_n + \half \Delta t$
\begin{equation}
\label{eq:evolution.general}
\begin{split}
\bw_j^{\nph,-} &= \bw_j^{n,-}-\frac{\Delta t}{2 \Delta x_j}(f(\bw_j^{n, +})-f(\bw_j^{n, -})) \\
\bw_j^{\nph,+} &= \bw_j^{n,+}-\frac{\Delta t}{2 \Delta x_j}(f(\bw_j^{n, +})-f(\bw_j^{n, -}))
\end{split}
\end{equation}
where $\Delta x_j = x_{\jph} - x_{\jmh}$. The final update is performed by using an approximate Riemann solver on the evolved quantities
\begin{equation}
\bw_j^{n + 1} = \bw_j^n - \frac{\Delta t}{\Delta x_j}  \left( \ff_{j + \frac{1}{2}}^{n + \frac{1}{2}} - \ff^{n + \frac{1}{2}}_{j - \frac{1}{2}} \right) \label{eq:muscl.final.general}
\end{equation}
where
\[
\ff_{j + \frac{1}{2}}^{n + \frac{1}{2}} = \ff \left( \bw_j^{n + \frac{1}{2}, +}, \bw_{j + 1}^{n + \frac{1}{2}, -} \right)
\]
is some numerical flux function. The key idea of the proof is to write the evolution $\bw_{j}^{\nph, \pm}$ from~\eqref{eq:evolution.general} as a convex combination of exact solution of some Riemann problem and the final update $\bw_j^{n+1}$ from~\eqref{eq:muscl.final.general} as a convex combination of first order finite volume updates on appropriately chosen subcells.
\subsection{Primary generalization for proof}
For the uniform, cell-centered case, Berthon~\cite{Berthon2006} defined $\bw_j^{*,\pm}$ to satisfy
\[
\frac{1}{2} \bw_j^{n, -} + \bw_j^{\ast, \pm} + \frac{1}{2} \bw_j^{n, +} = 2 \bw_j^{n, \pm}
\]
We generalize it for non-cell centred grids~\eqref{eq:non.cell.centred.defn}
\[
\mum \bw_j^{n, -} + \bw_j^{\ast, \pm} + \mup \bw_j^{n, +} = 2 \bw_j^{n, \pm}
\]
where
\begin{equation}\label{eq:mu.pm}
\mum = \frac{x_{\jph} - x_j}{x_{\jph} - x_{\jmh}}, \qquad
\mup = \frac{x_j - x_{\jmh}}{x_{\jph} - x_{\jmh}}
\end{equation}
This choice was made to keep the natural extension of $\bw_j^{*,\pm}$ in the conservative reconstruction case:
\[
\bw_j^{*,\pm} = \bw_j^n + 2(x_{j\pm \frac{1}{2}} - x_j)\slope_j
\]
noting that $\bw_j^{n, \pm}$ are given by~\eqref{eq:reconstruction.general}.
\subsection{Proving admissibility}\label{sec:mh.adm}
The following lemma about conservation laws will be crucial in the proof.
\begin{lemma}
Consider the 1-D Riemann problem
\begin{align*}
\bw_t + \ff(\bw)_x & = 0\\
\bw (x, 0) & = \left\{\begin{array}{lll}
\bw_l, & \quad & x < 0\\
\bw_r, &  & x > 0
\end{array}\right.
\end{align*}
in $[-h,h] \times [0,\Delta t]$ where
\begin{equation} \label{eq:con.law.dt}
\frac{\Delta t}{h} \sigma (\bw_l, \bw_r) \leq 1
\end{equation}
Then, for all $t \le \Delta t$,
\begin{align}
\label{lemma:avg.riemann.problem}
\int_{- h}^h \bw(x,t) \ud x = h (\bw_l + \bw_r) - t (f (\bw_r) - f (\bw_l))
\end{align}
\end{lemma}
\begin{proof}
Integrate the conservation law over $(- h, 0) \times (0, t)$
\begin{align*}
0 & = \int_{- h}^0 \bw(x, t) \ud x - h \bw_l + \int_0^t (\ff (\bw (0^-,t)) - \ff (\bw (-h, t))) \ud t\\
& = \int_{- h}^0 \bw(x,t) \ud x - h \bw_l + t (\ff (\tilde{\bw} (0^-)) - \ff(\bw_l))
\end{align*}
where, by self-similarity of solution of Riemann problem, $\tilde{\bw}$ is defined so that $\bw(x,t) = \tilde{\bw}(x/t)$ and $\ff(\bw(-h,t)) = \ff(\bw_l)$ is obtained as characteristics from $[0,h]$ do not reach $x=-h$ due to the time restriction~\eqref{eq:con.law.dt}. Rewriting gives
\begin{align*}
\int_{- h}^0 \bw(x,t) \ud x = h \bw_l - t (\nobracket \ff(\tilde{\bw}
(0^-)) - \ff(\bw_l))
\end{align*}
Similarly,
\[
\int_0^h \bw(x,t) \ud x = h \bw_r - t (\ff(\bw_r) - \ff(\tilde{\bw}(0^+)))
\]
If $\tilde{\bw}$ is discontinuous at $x=0$, by Rankine-Hugoniot conditions, we will have a stationary jump at $x/t=0$ and obtain $\ff(\tilde{\bw}(0^+)) = \ff(\tilde{\bw}(0^-))$. The same trivially holds if $\tilde{\bw}$ is continuous at $x/t=0$. Thus, we can sum the previous two identities to get \eqref{lemma:avg.riemann.problem}.
\end{proof}

We will now give a criterion under which we can prove $\bw_j^{\nph,\pm} \in \Uad$, i.e., the evolution step~\eqref{eq:evolution.general} preserves $\Uad$.
\begin{lemma}
\label{lemma:m.h.step.1}Define $\mu_\pm$ by \eqref{eq:mu.pm}  and pick $\bw_j^{\ast, \pm}$ to satisfy
\begin{equation} \label{eq:ustar.defn}
\frac{\mum}{2} \bw_j^{n, -} + \frac{1}{2} \bw_j^{\ast, \pm} +
\frac{\mup}{2} \bw_j^{n, +} = \bw_j^{n, \pm}
\end{equation}
Assume $\bw_j^{n, \pm}, \bw_j^{\ast, \pm} \in \Uad$  and the CFL restrictions
\begin{align}
\max_{j} \frac{\Delta t}{\mum \Delta x_j} \sigma
\left( \bw_j^{n, -}, \bw_j^{\ast, \pm} \right) \leq
1, \qquad
\max_{j} \frac{\Delta t}{\mup \Delta x_j} \sigma
\left( \bw_j^{\ast, \pm}, \bw_j^{n, +} \right) \leq
1
\label{eq:new.cfl1}
\end{align}
are satisfied. Then,  $\bw_j^{n + \half, \pm}$ given by the first step~\eqref{eq:evolution.general} of the MUSCL-Hancock scheme is in $\Uad$.
\end{lemma}
\begin{proof}
We will prove that $\bw_j^{n + \frac{1}{2}, +} \in \Uad$, and the proof for $\bw_j^{n + \frac{1}{2}, -}$ shall follow similarly. The key idea is to write 
$\bw_j^{n + \frac{1}{2}, \correction{+}}$ as the exact solution of some Riemann problems. Define $\bw^h (x, t) : (x_\jmh, x_\jph) \times (0, \Delta t/2) \rightarrow \Uad$ to be the weak solution of the Cauchy problem with initial data
\[
\bw^h (x, 0) = \begin{cases}
\bw_j^{n, -}, \quad & \tmop{if} x \in (x_{\jmh}, x_{j - 1 /
4})\\
\bw_j^{\ast, +}, & \tmop{if} x \in (x_{j - 1 / 4}, x_{j + 1 /
4})\\
\bw_j^{n, +}, & \tmop{if} x \in (x_{j + 1 / 4}, x_{\jph})
\end{cases}
\]
where
\[
x_{j - \frac{1}{4}} = \frac{1}{2} (x_{\jmh} + \tilde{x}_j), \qquad x_{j +
\frac{1}{4}} = \frac{1}{2} (\tilde{x}_j + x_{\jph}), \qquad \tilde{x}_j = x_{j - \frac{1}{2}} + \mum \Delta x_j
\]

Under our time step restrictions~\eqref{eq:new.cfl1}, the solution $\bw^h$ at time $\frac{\Delta t}{2}$
is made up of non-interacting Riemann problems centered at
$x_{j \pm \frac{1}{4}}$, see Figure~(\ref{fig:non.interacting.rp1}). We take the projection of
$\bw^h (x, \Delta t / 2)$ on piecewise-constant functions
\[ \tilde{\bw}_j^{n + \frac{1}{2}, +} := \frac{1}{\Delta x_j} \int_{x_\jmh}
^{x_\jph} \bw^h \left( x, \frac{\Delta t}{2} \right) \ud x \]
Since we assumed that the conservation law preserves $\Uad$, we get $\tilde{\bw}_j^{n + \frac{1}{2}, +} \in
\Uad$. If we prove $\tilde{\bw}_j^{n + \frac{1}{2}, +} = \bw_j^{n +
\half, +}$, we will have our claim. Applying Lemma \ref{lemma:avg.riemann.problem} to the two non-interacting Riemann problems, we get
\begin{figure}
\begin{center}
\includegraphics[width=0.9\textwidth]{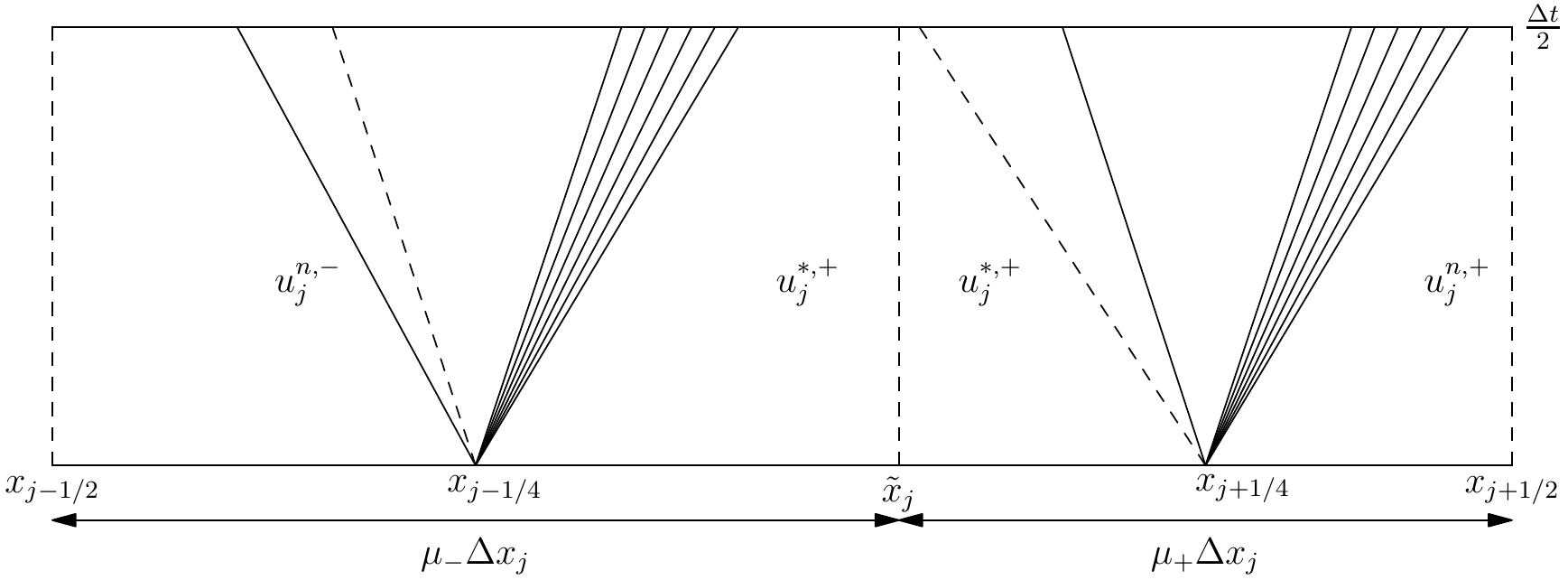}
\caption{Two non-interacting Riemann problems}
\label{fig:non.interacting.rp1}
\end{center}
\end{figure}
\begin{align*}
\widetilde{\bw}_j^{n + \frac{1}{2}, +} & =\frac{1}{\Delta x_j}
\left (\int_{x_{\jmh}}^{\tilde{x}_j} \bw^h \left( x, \frac{\Delta t}{2}
\right) \ud x
+
\int_{\tilde{x}_j}^{x_{\jph}} \bw^h \left( x, \frac{\Delta t}{2}
\right)  \ud x \right)\\
& = \frac{1}{\Delta x_j} \left[ \frac{\tilde{x}_j - x_{\jmh}}{2} \bw_j^{n, -} +
\frac{\Delta x_j}{2} \bw_j^{\ast, +} + \frac{x_{\jph} - \tilde{x}_j}{2}
\bw_j^{n, +}- \frac{\Delta t}{2} \left( f \left( \bw_j^{n, +} \right) - f \left(
\bw_j^{n, -} \right) \right) \right ]\\
& = \frac{1}{2} \left( \mum \bw_j^{n, -} + \bw_j^{\ast, +}
+ \mup \bw_j^{n, +} \right) - \frac{\Delta t / 2}{\Delta x_j} \left( f \left(
\bw_j^{n, +} \right) - f \left( \bw_j^{n, -} \right) \right)\\
& = \bw_j^{n, +} - \frac{\Delta t / 2}{\Delta x_j} \left( f \left( \bw_j^{n, +}
\right) - f \left( \bw_j^{n, -} \right) \right), \quad \text{using~\eqref{eq:ustar.defn}}\\
& = \bw_j^{n + \frac{1}{2}, +}, \quad \text{by~\eqref{eq:evolution.general}}
\end{align*}
This proves our claim.
\end{proof}

Now, we introduce a new variable $\bw_j^{n + \frac{1}{2}, \ast}$ defined as follows:
\begin{equation}\label{eq:uj.nph.s}
 \mum \bw_j^{n + \frac{1}{2}, -} + \bw_j^{n + \frac{1}{2},
\ast} + \mup \bw_j^{n + \frac{1}{2}, +} = 2 \bw_j^n
\end{equation}
\begin{figure}
\begin{center}
\includegraphics[width=0.9\textwidth]{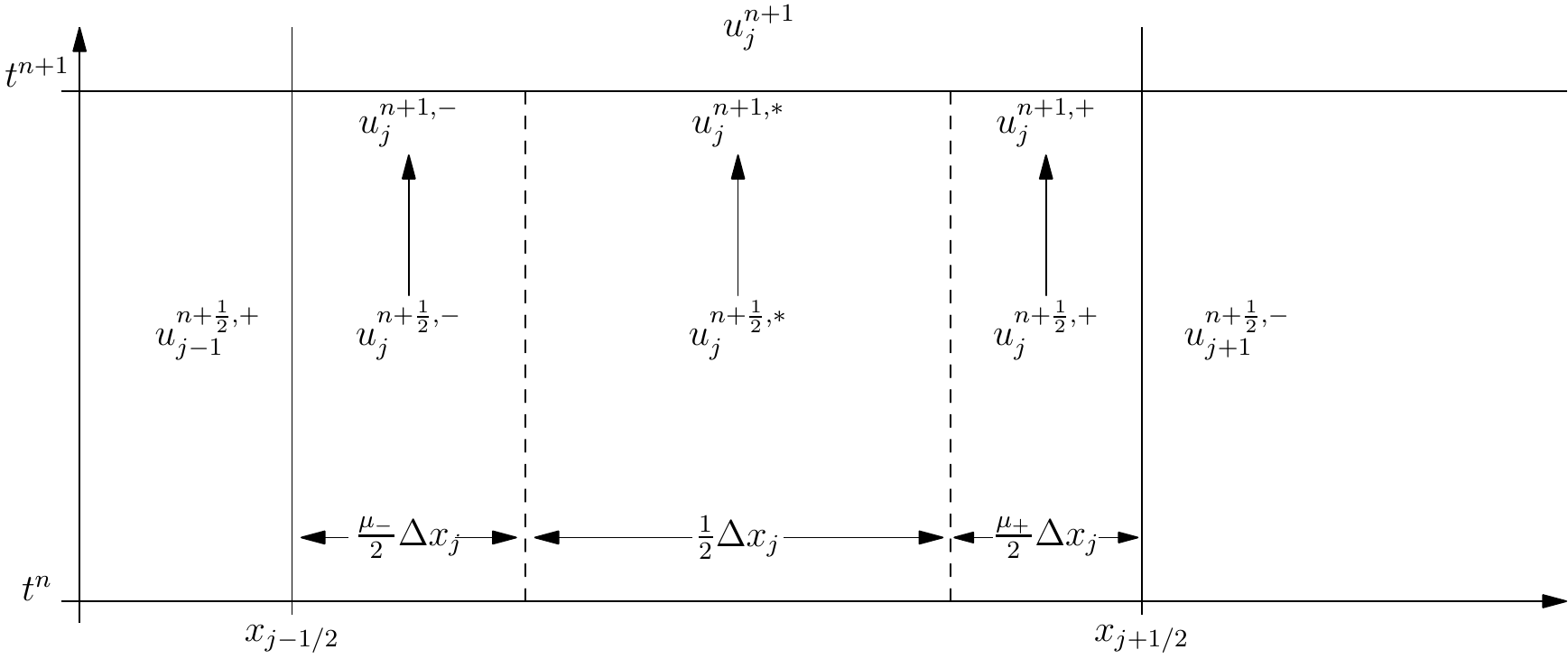}
\caption{Finite volume evolution}
\label{fig:fv.evolution}
\end{center}
\end{figure}
As illustrated in Figure~(\ref{fig:fv.evolution}), we evolve each state according to the associated first order scheme to define the following
\begin{equation}
\label{eq:all.fvm.updates}
\begin{split}
\bw_j^{n + 1, -} & = \bw_j^{n + \frac{1}{2}, -} -
\cfrac{\Delta t}{\mum \Delta x_j / 2} \left( f \left( \bw_j^{n + \frac{1}{2}, -}, \bw_j^{n + \frac{1}{2}, \ast} \right) - f \left( \bw_{j
- 1}^{n + \frac{1}{2}, +}, \bw_j^{n + \frac{1}{2}, -} \right) \right)\\
\bw_j^{n + 1, \ast} & = \bw_j^{n + \frac{1}{2}, \ast} - \cfrac{\Delta t}{\Delta x_j /
2} \left( f \left( \bw_j^{n + \frac{1}{2}, \ast}, \bw_j^{n + \frac{1}{2},
+} \right) - f \left( \bw_j^{n + \frac{1}{2}, -}, \bw_j^{n + \frac{1}{2},
\ast} \right) \right)\\
\bw_j^{n + 1, +} & = \bw_j^{n + \frac{1}{2}, +} -
\cfrac{\Delta t}{\mup \Delta x_j / 2} \left( f \left( \bw_j^{n +
\frac{1}{2}, +}, \bw_{j + 1}^{n + \frac{1}{2}, -} \right) - f \left(
\bw_j^{n + \frac{1}{2}, \ast}, \bw_j^{n + \frac{1}{2}, +} \right) \right)
\end{split}
\end{equation}
Recall that~\eqref{eq:muscl.final.general} is
\[
\bw_j^{n + 1} = \bw_j^n - \frac{\Delta t}{\Delta x_j} \left(  \ff \left( \bw_j^{n +
\frac{1}{2}, +}, \bw_{j + 1}^{n + \frac{1}{2}, -} \right) -  \ff
\left( \bw_{j - 1}^{n + \frac{1}{2}, +}, \bw_j^{n + \frac{1}{2}, -} \right)
\right)
\]
Using~\eqref{eq:uj.nph.s} and~\eqref{eq:all.fvm.updates}, we get
\[
\frac{\mum}{2} \bw_j^{n + 1, -} + \frac{1}{2} \bw_j^{n + 1, \ast} + \frac{\mup}{2} \bw_j^{n + 1, +} = \bw_j^{n + 1}
\]
Thus, assuming $\bw_j^{n + \frac{1}{2}\pm}, \bw_j^{n + \frac{1}{2}, \ast} \in \Uad$ for all $j$, and since $\half \mum + \half \mup = 1$, we get $\bw_j^{n+1} \in \Uad$ under the following time step restrictions arising from the assumed time step requirement~\eqref{eq:numflux.admissibility.cfl} for admissibility of the first order finite volume method
\begin{equation}
\begin{gathered}
\begin{aligned}
&\max_j \cfrac{\Delta t}{\mum \Delta x_j / 2} \sigma \left( \bw_j^{n + \frac{1}{2}, -}, \bw_j^{n + \frac{1}{2}, \ast} \right) \leq 1,\\
& \max_j \cfrac{\Delta t}{\mum \Delta x_j / 2} \sigma \left( \bw_{j - 1}^{n + \frac{1}{2}, +}, \bw_j^{n + \frac{1}{2}, -} \right)\leq 1,\\
& \max_j \cfrac{\Delta t}{\Delta x_j / 2} \sigma \left(\bw_j^{n + \frac{1}{2}, \ast}, \bw_j^{n + \frac{1}{2}, +}\right) \leq 1,
\end{aligned} \qquad
\begin{aligned}
& \max_j  \cfrac{\Delta t}{\Delta x_j / 2} \sigma \left(\bw_j^{n + \frac{1}{2}, -}, \bw_j^{n + \frac{1}{2}, \ast}\right) \leq 1 \\
& \max_j \cfrac{\Delta t}{\mup \Delta x_j / 2} \sigma \left( \bw_j^{n + \frac{1}{2}, +}, \bw_{j + 1}^{n + \frac{1}{2}, -} \right) \leq 1\\
& \max_j \cfrac{\Delta t}{\mup \Delta x_j / 2} \sigma \left( \bw_j^{n + \frac{1}{2}, \ast}, \bw_j^{n + \frac{1}{2}, +} \right) \leq 1
\end{aligned}
\label{eq:new.cfl2}
\end{gathered}
\end{equation}
This can be summarised in the following Lemma.
\begin{lemma}
\label{lemma:muscl.step2.general}Assume that the states $\left\{ \bw_j^{n + \frac{1}{2}, \pm} \right\}_j$, 
$\left\{ \bw_j^{n + \frac{1}{2}, \ast} \right\}_j$ belong to $\Uad$, where
$\bw_j^{n + \frac{1}{2}, \ast}$ is defined as in~\eqref{eq:uj.nph.s}.
Then, the updated solution $\bw_j^{n+1}$ of MUSCL-Hancock scheme (\ref{eq:reconstruction.general}-\ref{eq:muscl.final.general}) is in $\Uad$ under the
CFL conditions {\eqref{eq:new.cfl2}}.
\end{lemma}

Since Lemma \ref{lemma:m.h.step.1} states that $\bw_j^{n + \frac{1}{2}, \pm} \in
\Uad$ if $\bw_j^{\ast, \pm} \in \Uad$, the only new condition pertains to $\bw_j^{n + \frac{1}{2}, \ast}$. Our goal now is to understand this condition, and ultimately prove that it follows from the requirement that $\bw_j^{\ast,\pm} \in \Uad$ in case of conservative reconstruction.

Recall that $\bw_j^{n + \frac{1}{2}, \ast}$ was defined by~\eqref{eq:uj.nph.s}; expanding the definition of $\bw_j^{n + \frac{1}{2}, \pm}$ given by~\eqref{eq:evolution.general} yields
\begin{equation}
\label{eq:uj.nph.s.explicit}
\bw_j^{n + \frac{1}{2}, \ast} = 2 \bw_j^n - \left(
   \mum \bw_j^{n, -} + \mup \bw_j^{n, +}
   \right) - \frac{\Delta t}{2 \Delta x_j} (f
   (\bw_j^{n, -}) - f (\bw_j^{n, +}))
\end{equation}
This identity~\eqref{eq:uj.nph.s.explicit} will be seen as an evolution update similar to~\eqref{eq:evolution.general} with $\bw_j^{n, +}$ and $\bw_j^{n, -}$ being swapped and $\bw_j^n$ replaced with $2 \bw_j^n - \left(\mum \bw_j^{n, -} + \mup \bw_j^{n, +} \right)$. The admissibility of $\bw_j^{n + \frac{1}{2}, \ast}$ will be studied by adapting the proof of admissibility for~\eqref{eq:evolution.general}, accounting for the differences in the case of~\eqref{eq:uj.nph.s.explicit}. Define $\bw_j^{*,*}$ so that
\begin{equation} \label{eq:uss.defn}
\frac{\mum}{2} \bw_j^{n,-} + \frac{1}{2}\bw_j^{*,*} + \frac{\mup}{2} \bw_j^{n,+} = 2\bw_j^n-(\mum\bw_j^{n,-}+\mup\bw_j^{n,+})
\end{equation}
i.e.,
\begin{align}
\bw_j^{*,*} = 4 \bw_j^n - 3(\mum\bw_j^{n,-}+\mup\bw_j^{n,+}) \label{eq:wss.simplified}
\end{align}
The following Lemma extends the proof of Lemma \ref{lemma:m.h.step.1} to obtain conditions for $\bw_j^{n + \frac{1}{2}, \ast} \in \Uad$.
\begin{lemma}
\label{lemma:muscl.step3.wss}Assume that $\bw_j^n \in \Uad$ for all $j$. Consider the reconstructions $\bw_j^{n,\pm}$ and the $\bw_j^{\ast, \ast}$ defined in~\eqref{eq:uss.defn}.
Assume $\bw_j^{n, \pm}, \bw_j^{\ast, \ast} \in \Uad$ and the time step restrictions
\begin{equation}
\max_{j} \frac{\Delta t}{\mum \Delta x_j} \sigma \left( \bw_j^{\ast, \ast}, \bw_j^{n, -} \right) \leq 1, \qquad
\max_{j} \frac{\Delta t}{\mup \Delta x_j} \sigma \left( \bw_j^{n, +}, \bw_j^{\ast, \ast} \right) \leq 1
\label{eq:new.cfl3}
\end{equation}
Then $\bw_j^{n + \frac{1}{2}, \ast} \in \Uad$.
\end{lemma}

\begin{proof}
We will use the identity which follows from~(\ref{eq:uj.nph.s.explicit},\ref{eq:uss.defn})
\begin{equation} \label{eq:ujph.s.identity}
\bw_j^{n + \frac{1}{2}, \ast} = \frac{\mum \bw_j^{n,-} + \bw_j^{*,*} + \mup \bw_j^{n,+}}{2} - \frac{\Delta t}{2 \Delta x_j}
(f (\bw_j^{n, -}) - f (\bw_j^{n, +}))
\end{equation}
to fall back to previous case of Lemma \ref{lemma:m.h.step.1}.

Define $\bw^h (x, t) : (x_\jmh, x_\jph) \times (0, \Delta t/2) \rightarrow \Uad$ to be the weak solution of the Cauchy problem with initial data
\begin{align*}
\bw^h (x, 0) = \begin{cases}
\bw_j^{n, +}, \quad & \tmop{if} x \in (x_\jmh, x_{j - 1 /
4})\\
\bw_j^{\ast, \ast}, & \tmop{if} x \in (x_{j - \frac 14}, x_{j + 1 /
4})\\
\bw_j^{n, -}, & \tmop{if} x \in (x_{j + \frac 14}, x_\jph)
\end{cases}
\end{align*}
where
\[
x_{j - \frac{1}{4}} = \frac{1}{2} (x_\jmh + x_j), \qquad x_{j + \frac{1}{4}} = \frac{1}{2} (x_j + x_\jph)
\]

Note that we have already accounted for the swapped $\bw_j^{n,-}$ and $\bw_j^{n,+}$ while defining this initial condition, see Figure~(\ref{fig:non.interacting.rp2}).

Under the assumed CFL conditions~\eqref{eq:new.cfl3}, the solution $\bw^h$ at time $\frac{\Delta t}{2}$ is made up of non-interacting Riemann problems centered at $x_{j \pm \frac{1}{4}}$. Take the projection of $\bw^h (x, t / 2)$ on piecewise-constant functions
\[ \widetilde{\bw}_j^{n + \frac{1}{2}, \ast} := \frac{1}{\Delta x_j} \int_{x_{j - \frac 12}}^{x_{\jph}} \bw^h \left( x, \frac{\Delta t}{2} \right) \ud x \in \Uad \]
\begin{figure}
\begin{center}
\includegraphics[width=0.9\textwidth]{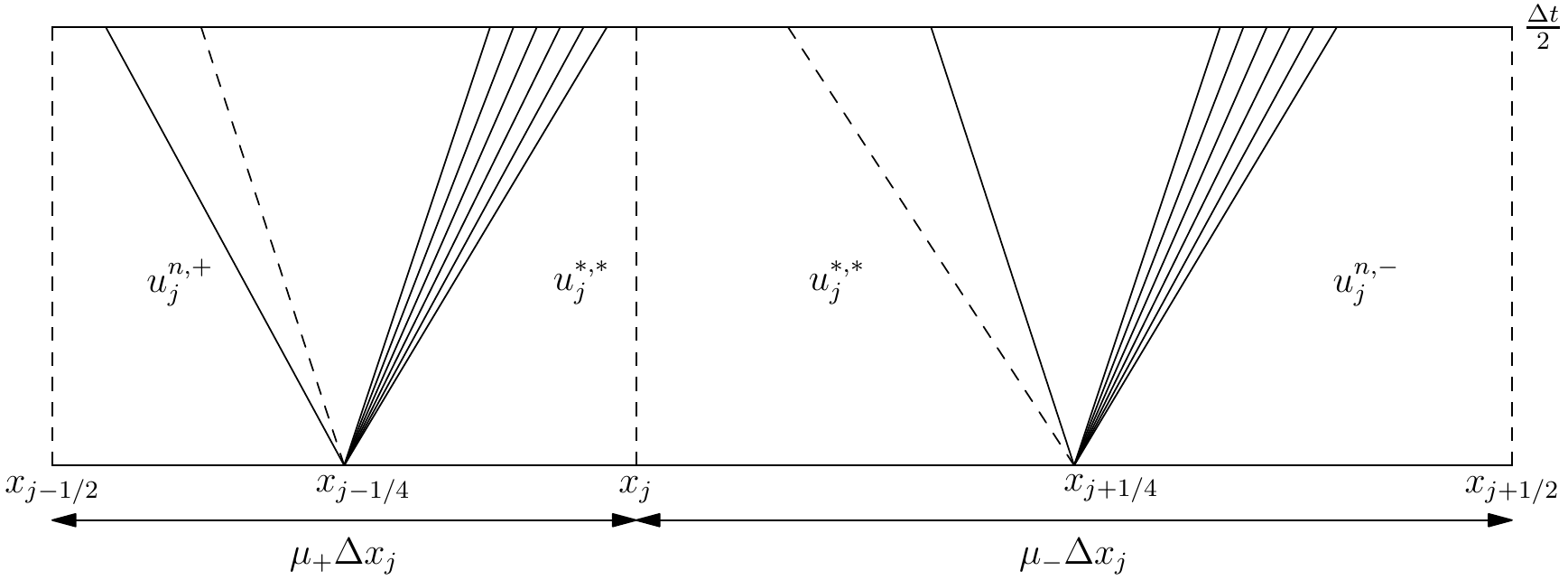}
\caption{Two non-interacting Riemann problems}
\label{fig:non.interacting.rp2}
\end{center}
\end{figure}
As in Lemma~\ref{lemma:m.h.step.1}, we will show $\bw_j^{n + \frac{1}{2}, \ast} \in \Uad$ by showing $\bw_j^{n + \frac{1}{2}, \ast} = \tilde{\bw}_j^{n + \frac{1}{2}, \ast}$. Applying Lemma~\ref{lemma:avg.riemann.problem} to the two non-interacting Riemann problems, we get
\begin{align*}
\widetilde{\bw}_j^{n + \frac{1}{2}, \ast} & = \frac{1}{\Delta x_j}
\left (\int_{x_{\jmh}}^{x_j} \bw^h \left( x, \frac{\Delta t}{2} \right) \ud x
+
\int_{x_j}^{x_{\jph}} \bw^h \left( x, \frac{\Delta t}{2}
\right)  \ud x \right)\\
& = \frac{1}{\Delta x_j} \left( \frac{x_j - x_{\jmh}}{2} \bw_j^{n, +} +
 \frac{\Delta x_j}{2} \bw_j^{\ast, \ast} + \frac{x_{\jph} - x_j}{2}
 \bw_j^{n, -} - \frac{\Delta t}{2} \left( f \left( \bw_j^{n, -} \right) - f \left(
 \bw_j^{n, +} \right) \right)\right)\\
& = \frac{1}{2} \left( \mup \bw_j^{n, +} + \bw_j^{\ast,
\ast} + \mum \bw_j^{n, -} \right) - \frac{\Delta t / 2}{\Delta x_j} \left( f \left(
\bw_j^{n, -} \right) - f \left( \bw_j^{n, +} \right) \right)\\
& =\bw_j^{n + \frac{1}{2}, \ast},\quad \text{by~\eqref{eq:ujph.s.identity}}
\end{align*}
This proves our claim.
\end{proof}

For conservative reconstruction,
\[
\mum \bw_j^{n, -} + \mup \bw_j^{n, +} = \bw_j^n
\]
and thus by {\eqref{eq:wss.simplified}}, $\bw_j^{\ast, \ast} = \bw_j^n$. The previous lemma can thus be specialized as follows.
\begin{lemma}
\label{lemma:muscl.step3.conservative}Assume that $\bw_j^n \in \Uad$ and $\bw_j^{n, \pm} \in \Uad$ for all $j$ with conservative reconstruction. Also assume the CFL restrictions
\begin{equation}
\max_{j} \frac{\Delta t}{\mum \Delta x_j} \sigma \left( \bw_j^n, \bw_j^{n, -} \right) \leq 1, \qquad
\max_{j} \frac{\Delta t}{\mup \Delta x_j} \sigma \left( \bw_j^{n, +}, \bw_j^n \right) \leq 1
\label{eq:new.cfl3.conservative}
\end{equation}
where $\mu^{\pm}$ are defined in~\eqref{eq:mu.pm}. Then, $\bw_j^{n + \frac{1}{2}, \ast}$ defined in~\eqref{eq:uj.nph.s} is in $\Uad$.
\end{lemma}
Combining Lemmas~\ref{lemma:m.h.step.1},~\ref{lemma:muscl.step2.general},~\ref{lemma:muscl.step3.conservative}, we obtain the final criterion for admissibility preservation of MUSCL-Hancock with conservative reconstruction in the following Theorem~\ref{thm:final.condn.conservative}.
\begin{theorem}
\label{thm:final.condn.conservative}Let $\bw_j^n \in \Uad$ for all $j$ and $\bw_j^{n, \pm}$ be the conservative reconstructions defined as
\[
\bw_j^{n, +} = \correction{\bw_j^n} + (x_{\jph} - x_j) \slope_j, \qquad \bw_j^{n, -} = \correction{\bw_j^n} + (x_{\jmh} - x_j) \slope_j \]
so that $\bw_j^{\ast, \pm}$ defined in~\eqref{eq:ustar.defn} is also given by
\begin{equation} \label{eq:us.conservative}
\bw_j^{\ast, \pm} = \bw_j^n + 2 ( x_{j \pm \frac{1}{2}} - x_j ) \slope_j
\end{equation}
Assume the slope $\slope_j$ \correction{is} chosen such that $\bw_j^{\ast, \pm} \in \Uad$ and the CFL restrictions~(\ref{eq:new.cfl1}, \ref{eq:new.cfl2}, \ref{eq:new.cfl3.conservative}) hold. Then, the updated solution $\bw_j^{n + 1}$, defined by MUSCL-Hancock scheme~\eqref{eq:muscl.final.general} is in $\Uad$.
\end{theorem}

\begin{proof}
Once we obtain $\bw_j^{n, \pm} \in \Uad$, the claim follows from Lemmas~\ref{lemma:m.h.step.1}-\ref{lemma:muscl.step3.conservative}. To prove that $\bw_j^{n, \pm}$ is indeed in $\Uad$, we make the straight forward observation that
\[
\bw_j^{n, \pm} = \frac{1}{2} \bw_j^{\ast, \pm} + \frac{1}{2} \bw_j^n
\]
Since $\bw_j^{\ast, \pm}$ and $\bw_j^n$ are in $\Uad$, the proof is completed by the convex property of $\Uad$.
\end{proof}

\begin{remark}\label{rmk:mh.restriction.for.fr}
The strictest time step restriction for admissibility of the MUSCL-Hancock scheme is imposed by~\eqref{eq:new.cfl2}. Thus, we can find the CFL coefficient for grid used by subcell-based blending scheme~\eqref{eq:low.order.update} by minimizing the denominator in~\eqref{eq:new.cfl2} which is given by
\[
\frac{1}{2} \min_{j=0,\dots,N}\left(\xi_j - \sum_{k=0}^{j-1} w_k\right )w_j = \frac{1}{2} \xi_0w_0
\]
where $\xi_0,w_0$ are the first Gauss-Legendre quadrature point~\eqref{eq:xi0.defn} and weight in $[0,1]$. This coefficient is less than half of the optimal CFL coefficient that arises from Fourier stability analysis of the LWFR scheme with D2 dissipation, see Table 1 of~\cite{babbar2022}.
\end{remark}

\subsection{Non-conservative reconstruction}
To maintain the simple admissibility criterion (Theorem~\ref{thm:final.condn.conservative}), we have restricted ourselves to conservative reconstruction in this work. In this section, we explain the complexities that will arise in enforcing admissibility if we perform reconstruction with non-conservative variables $\bv$ defined by the change of variables formula
\[
\bv = \kappa \left( \bw \right)
\]
The linear approximation is given by
\[
r^n (x) = \bv^n_j + (x - x_j)\slope_j, \quad x \in [x_{\jmh}, x_{\jph}]
\]
and thus the trace values are
\[
\bv_j^{n, \pm} = \bv_j^n + (x_{\jpmh} - x_j)\slope_j
\]
Since the arguments of proof of admissibility depend on constraints on the conservative variables, we have to take the inverse map on our reconstructions. For example, conservative variables at the face are obtained as
\begin{equation}\label{eq:non.con.face.defn1}
\bw_j^{n, \pm} = \kappa^{- 1} (\bv_j^{n, \pm})
\end{equation}
Due to the non-linearity of the map $\kappa$, unlike the conservative case, we have
\[
\mum \bw_j^{n, -} + \mup \bw_j^{n, +} \neq \bw_j^n
\]
which is why several reductions of admissibility constraints will fail. The admissibility criteria for non-conservative reconstruction is stated in Theorem~\ref{thm:non.conservative.mh}.
\begin{theorem}\label{thm:non.conservative.mh}
Assume that $\bw_j^n \in \Uad$ for all $j$. Consider $\bw_j^{n,\pm}$ defined in \eqref{eq:non.con.face.defn1}, $\bw_j^{*,\pm}$ defined in~\eqref{eq:ustar.defn} and
 $\bw_j^{*,*}$ defined so that
\[
\frac{\mum}{2}\bw_j^{n,-}
+ \frac 12 \bw_j^{*,*}
+ \frac{\mup}{2} \bw_j^{n,+} = 2\bw_j^n-(\mum\bw_j^{n,-} + \mup\bw_j^{n,+})
\]
Assume that the slope $\slope_j$ is chosen so that $\bw_j^{n,\pm}, \bw_j^{*,\pm}, \bw_j^{*,*} \in \Uad$ and that the CFL restrictions~(\ref{eq:new.cfl1}, \ref{eq:new.cfl2}, \ref{eq:new.cfl3}) are satisfied. Then the updated solution $\bw_j^{n+1}$ of MUSCL-Hancock scheme~\eqref{eq:muscl.final.general} is in $\Uad$.
\end{theorem}

\subsection{MUSCL-Hancock scheme in 2-D}\label{sec:2d.mh}
Consider the 2-D hyperbolic conservation law~\eqref{eq:2d.hyp.con.law} with fluxes $\ff, \fg$. For simplicity, assume that the reconstruction is performed on conservative variables. Thus, the linear reconstructions are given by
\[
r^n_{ij} (x, y)
= \bw_{ij}^n + (x - x_i) \slope_i^x
+ (y - y_j) \slope_j^y,
\]
and the approximations at the face $\unpx, \unmx, \unpy, \unmy$ are
\begin{equation}\label{eq:2DMH1}
\begin{split}
\unpmx_{ij} & = r^n_{ij} (x_{\ipmh}, y_j) =
{\bw}_{ij}^n + (x_{\ipmh} - x_i)\slope_i^x\\
\unpmy_{ij} & = r^n_{ij} (x_i, y_{\jpmh}) =
{\bw}_{ij}^n + (y_{\jpmh} - y_j) \slope_j^y
\end{split}
\end{equation}
and the derivative approximations are given by
\[
\partial_x   \ff_{i  j} := \frac{1}{\Delta x_i} \left(  \ff \left( \unpx_{\correction{ij}} \right) -  \ff \left(\unmx_{ij} \right) \right), \qquad
\partial_y   \fg_{i  j} := \frac{1}{\Delta y_j} \left(  \fg \left( \unpy_{\correction{ij}} \right) -  \fg \left(\unmy_{ij} \right) \right)
\]
\[
\partial_t  \bw_{i  j}^n :=
- \partial_x   \ff_{i, j}
- \partial_y   \fg_{i, j}
\]
The evolutions to time level $\nph$ are given by
\begin{equation}\label{eq:2DMH2}
\unphpmx_{i  j} = \unpmx_{i  j} +
\frac{\Delta t}{2} \partial_t  \bw_{i  j}^n,\qquad \unphpmy_{i  j} = \unpmy_{i  j} +
\frac{\Delta t}{2} \partial_t  \bw_{i  j}^n
\end{equation}

and then the final update is performed as
\begin{equation}\label{eq:2DMH3}
\bw_{i  j}^{n + 1} =
\bw_{i  j}^n
- \frac{\Delta t}{\Delta x_i}  ( \ff_{\iph, j}^\nph - \ff_{\imh, j}^{\nph})
- \frac{\Delta t}{\Delta y_j}  ( \fg_{i, \jph}^\nph -  \fg_{i,\jmh}^{\nph})
\end{equation}
where the numerical fluxes are computed as
\[
\ff_{\iph, j}^\nph
= \ff \left( \unphpx_{i  j}, \unphmx_{i+1,  j} \right), \qquad \fg_{i, \jph}^\nph
= \fg \left( \unphpy_{i  j}, \unphmy_{i,  j+1} \right)
\]

\subsubsection{First evolution step}
As in 1-D, define $\uspmx_{i  j}, \uspmy_{ij}$ so that
\begin{equation}\label{eq:ustar.2d}
\begin{split}
\mupx  \unpx_{ij}
+ \uspmx
+ \mumx  \unmx_{ij}
& = 2 \unpmx_{ij} \\
\mupy  \unpy_{ij}
+ \uspmy
+ \mumy  \unmy_{ij}
& = 2 \unpmy_{ij}
\end{split}
\end{equation}
where
\begin{equation} \label{eq:muxy.defn}
\begin{split}
\mupx &= \frac{x_i - x_{\imh}}{x_{\iph} - x_{\imh}},
\qquad \mumx = \frac{x_{\iph} - x_i}{x_{\iph} - x_{i - 1 /2}}\\
\mupy &= \frac{y_j - y_{\jmh}}{y_{\jph} - y_{\jmh}}, \qquad
\mumy = \frac{y_{\jph} - y_j}{y_{\jph} - y_{\jmh}}
\end{split}
\end{equation}
Since we assume conservative reconstruction
\[
\mupx  \unpx_{i  j} + \mumx  \unmx_{i
j} = \mupy  \unpy_{i  j} + \mumy  \unmy_{i  j} = \bw_{i  j}^n
\]
Thus, we have
\[
\uspmx_{i  j}
= \bw_{i  j}
+ 2 (x_\ipmh - x_i)  \slope_i^x, \qquad
\uspmy_{i  j}
= \bw_{i  j} + 2 (y_{\jpmh} - x_j)  \slope_j^y
\]
We will particularly discuss admissibility of the updates
\begin{align}\label{eq:2Dupdates}
\unphpx_{i  j}
= \unpx_{i  j}
- \frac{\Delta t / 2}{\Delta x_i}  \left(  \ff \left( \unpx_{i  j} \right) -  \ff \left( \unmx_{ij} \right) \right)
- \frac{\Delta t / 2}{\Delta y_j} \left(  \fg \left( \unpy_{i  j} \right) - \fg \left( \unmy_{i  j} \right) \right)
\end{align}
Admissibility of the other three updates $\unphmx_{i  j}, \unphpmy_{i  j}$ will follow similarly. For some $k_x, k_y$ chosen such that $k_x + k_y = 1,$ we write \eqref{eq:2Dupdates} as
\begin{align*}
\unphpx_{i  j} = k_x \thetapx + k_y \thetapy
\end{align*}
where
\begin{equation}
\thetapx :=\unpx_{i  j} - \frac{\Delta t / 2}{k_x \Delta
x_i}  \left(  \ff \left( \unpx_{i  j} \right) -
 \ff \left( \unmx_{i  j} \right) \right)
\label{eq:update.x.combination}
\end{equation}
and
\begin{equation}
\thetapy
:=
\unpx_{i  j}
- \frac{\Delta t / 2}{k_y \Delta y_j}  \left(  \fg \left( \unpy_{i  j} \right) - \fg \left( \unmy_{i  j} \right) \right)
\label{eq:update.x.combination2}
\end{equation}
We will choose the slopes $\slope^x_i, \slope^y_j$ and time step $\Delta t$ so that $\thetapx,  \thetapy \in \Uad$. Then, we can take convex combinations of the two terms to obtain admissibility of $\unphpx_{i  j}$. The choice of $k_x, k_y$ will not influence the slope restriction, but only the time step restriction required to obtain admissibility. In this work, we only use the Fourier CFL restriction imposed by the Lax-Wendroff scheme~\eqref{eq:time.step.2d} and \correction{observe admissibility preservation in all our numerical experiments} and thus do not study the choice of $k_x,k_y$. However, in a similar context,~\cite{Zhang2010b} proposed the choice of
\begin{equation}\label{eq:kx.defn}
k_x = \frac{a_x / \Delta x_i}{a_x / \Delta x_i + a_y / \Delta y_j}, \qquad k_y = \frac{a_y / \Delta y_j }{a_x / \Delta x_i + a_y / \Delta y_j}
\end{equation}
where
\begin{equation}
a_x = \sigma_x(\unmx_{i j}, \unpx_{i j}), \qquad a_y = \sigma_y(\unmy_{i j}, \unpy_{i,j})
\end{equation}
In~\cite{Cui2023}, it was shown that the time step restriction imposed by the above decomposition is suboptimal and optimal decompositions were proposed. After choosing $k_x, k_y$, following the 1-D procedures from Section~\ref{sec:mh.adm}, the slopes $\slope^x_i, \slope^y_j$ will be limited to enforce admissibility of $\thetapx, \thetapy$ (\ref{eq:update.x.combination}, \ref{eq:update.x.combination2}). The admissibility preservation of $\thetapx$~\eqref{eq:update.x.combination} follows directly from the arguments used in Lemma \ref{lemma:m.h.step.1}, enforcing slope restriction so that $\unpmx_{i j}$ and $\uspx_{i  j}$ are admissible, and appropriate time step restrictions.
For admissibility of $\thetapy$~\eqref{eq:update.x.combination2}, we define $\uspxy_{ij}$ so that
\[
\mupy  \unpy_{i  j} + \uspxy_{ij} + \mumy  \unmy_{i  j} = 2 \unpx_{i j}
\]
Thus, the proof of Lemma \ref{lemma:m.h.step.1} shall apply as in 1-D under the assumption of admissibility of $\unpmy_{i  j}, \uspxy_{ij}$ and some CFL conditions. Thus, we will have admissibility of $\thetapy \in \Uad$. We obtain further simplifications because of conservative reconstructions
\begin{equation}
\uspxy_{ij} = \uspx_{ij}
\end{equation}
and thus the slope limiting for enforcing admissibility of $\uspx_{ij}$ will suffice. We note the precise slope and CFL restrictions are in Lemma~\ref{lemma:m.h.step.1.2d}.
\begin{lemma}
\label{lemma:m.h.step.1.2d}
For $\mu_{\pm x}, \mu_{\pm y}$ defined in~\eqref{eq:muxy.defn}, $\unpmx_{ij}, \unpmy_{ij}$ reconstructed in~\eqref{eq:2DMH1}, $\uspmx_{ij}, \uspmy_{ij}$ picked as in~\eqref{eq:ustar.2d}, assume
\[
\unpmx_{ij}, \unpmy_{ij}, \uspmx_{ij}, \uspmy_{ij} \in \Uad
\]
and the CFL restrictions
\begin{equation}
\begin{split}
\max_{i,j}  \frac{\lambda_{x_i}}{\mumx} \sigma_x \left(\unmx_{ij}, \uspmx_{ij} \right) \leq 1, \qquad
\max_{i,j}  \frac{\lambda_{x_i}}{\mupx} \sigma_x\left( \uspmx_{ij}, \unpx_{ij} \right)  \leq 1 \\
\max_{i,j}  \frac{\lambda_{y_j}}{\mumy} \sigma_y \left( \unmy_{ij}, \uspmx_{ij} \right) \leq 1, \qquad
\max_{i,j}  \frac{\lambda_{y_j}}{\mupy} \sigma_y \left( \uspmx_{ij}, \unpy_{ij} \right)  \leq 1 \\
\max_{i,j}  \frac{\lambda_{y_j}}{\mumy} \sigma_y  \left( \unmy_{ij}, \uspmy_{ij} \right) \leq 1, \qquad
\max_{i,j}  \frac{\lambda_{y_j}}{\mupy} \sigma_y \left( \uspmy_{ij}, \unpy_{ij} \right) \leq 1 \\
\max_{i,j}  \frac{\lambda_{x_i}}{\mumx} \sigma_x \left( \unmx_{ij}, \uspmy_{ij} \right) \leq 1, \qquad
\max_{i,j}  \frac{\lambda_{x_i}}{\mupx} \sigma_x \left( \uspmy_{ij}, \unpx_{ij} \right)  \leq 1
\label{eq:new.cfl1.2d}
\end{split}
\end{equation}
where $\lambda_{x_i} = \frac{\Delta t}{k_x \Delta x_i}, \lambda_{y_j} = \frac{\Delta t}{k_y \Delta y_j}$ for all $i,j$ and $k_x + k_y = 1$. Then, the updates $\unphpmx_{i  j}, \unphpmy_{i  j}$~\eqref{eq:2Dupdates} of the first step of 2-D MUSCL-Hancock scheme are admissible.
\end{lemma}
\subsubsection{Finite volume step}

The final update is given by
\begin{equation} \label{eq:2d.mh.final.subs}
\bw_{i  j}^{n + 1} =
\bw_{i  j}^n - \frac{\Delta t}{\Delta x_i}  ( \ff_{\iph, j}^{n + \half} -  \ff_{\imh, j}^{n + \half}) -
\frac{\Delta t}{\Delta y_j}  ( \fg_{i, \jph}^{n + \half} -  \fg_{i, \jmh}^{n + \half})
\end{equation}
where the numerical fluxes are computed as
\[
\ff_{\iph, j}^{n + \half} =  \ff \left( \unphpx_{ij}, \unphmx_{i+1,j} \right), \qquad  \fg_{i, \jph}^{n + \half} =  \fg \left( \unphpy_{i,j},\unphmy_{i,j+1} \right)
\]
As in the previous step, the expression~\eqref{eq:2d.mh.final.subs} is split \correction{into} a convex combination
\begin{align*}
\bw_{i  j}^{n + 1} = k_x \zeta_{ij}^{x} + k_y  \zeta_{ij}^{y}
\end{align*} where
\[
\zeta_{ij}^{x} := \bw_{ij}^n - \frac{\Delta t}{k_x \Delta x_i}  ( \ff_{\iph, j}^{n + \half} -  \ff_{\imh, j}^{n + \half}), \qquad
\zeta_{ij}^{y} := \bw_{ij}^n - \frac{\Delta t}{k_y \Delta y_j} ( \fg_{i, \jph}^{n + \half} -  \fg_{i, \jmh}^{n + \half})
\]
for some $k_x,k_y \ge 0$ with $k_x+k_y=1.$ The admissibility of \correction{$\zeta_{ij}^{x}$  and $ \zeta_{ij}^{y}$} will imply the admissibility of $\bw_{i  j}^{n + 1}.$ The admissibility of $\zeta_{ij}^{x}, \zeta_{ij}^{y}$ will follow exactly as from the procedure in 1-D (Lemma~\ref{lemma:muscl.step2.general}) with appropriate time step restrictions and assumption of admissibility of terms $\unphpmx_{ij}, \unphpmy_{ij}, \unphsx_{ij}, \unphsy_{ij}$ for $\unphsx_{ij}, \unphsy_{ij}$
defined as
\begin{equation} \label{eq:unphsxy}
\begin{split}
\mumx  \unphmx_{ij} + \bw_{ij}^{n+\frac{1}{2},\ast x} + \mupx  \unphpx_{ij}
& = 2 \bw_{ij}^n \\
\mumy  \unphmy_{ij}
+ \bw_{ij}^{n  + \frac{1}{2}, \ast y}
+ \mupy  \unphpy_{ij}
& = 2 \bw_{ij}^n
\end{split}
\end{equation}
The precise CFL restrictions and admissibility constraints are in the following Lemma~\ref{lemma:2d.muscl.step2.general}.
\begin{lemma}
\label{lemma:2d.muscl.step2.general}
Assume that the states $\left \{ \unphpmx_{ij}, \unphpmy_{ij}, \unphsx_{ij}, \unphsy_{ij}  \right \}_{i,j}$ belong to $\Uad$, where $ \unphsx_{ij}, \unphsy_{ij}$ are defined as in~\eqref{eq:unphsxy}. Then, the updated solution $\bw_{ij}^{n+1}$ of MUSCL-Hancock scheme is in $\Uad$ under the
CFL conditions
\begin{equation}
\label{eq:2d.new.cfl2}
\begin{gathered}
\begin{aligned}
\frac{2\lambda_{x_i}}{\mumx} \sigma_x \left( \unphmx_{ij}, \unphsx_{ij} \right)   \leq 1, \quad
2\lambda_{x_i}  \sigma_x \left(\unphsx_{ij}, \unphpx_{ij}\right) \leq 1, \quad
\frac{2\lambda_{x_i}}{\mupx} \sigma_x \left( \unphpx_{ij}, \unphmx_{i+1,j} \right) \leq 1 \\
\frac{2\lambda_{x_i}}{\mumx} \sigma_x \left( \unphpx_{i-1,j}, \unphmx_{ij} \right)\leq 1, \quad
2\lambda_{x_i}  \sigma_x \left(\unphmx_{ij}, \unphsx_{ij}\right) \leq 1, \quad
\frac{2\lambda_{x_i}}{\mupx} \sigma_x \left( \unphsx_{ij}, \unphpx_{ij} \right) \leq 1
\end{aligned} \\
\begin{aligned}
\frac{2 \lambda_{y_j}}{\mumy} \sigma_y \left( \unphmy_{ij}, \unphsy_{ij} \right)   \leq 1, \quad
2 \lambda_{y_j}  \sigma_y \left(\unphsy_{ij}, \unphpy_{ij}\right) \leq 1,  \quad
\frac{2 \lambda_{y_j}}{\mupy} \sigma_y \left( \unphpy_{ij}, \unphmy_{i,j+1} \right) \leq 1 \\
\frac{2 \lambda_{y_j}}{\mumy} \sigma_y \left( \unphpy_{i,j-1}, \unphmy_{ij} \right)\leq 1, \quad
2 \lambda_{y_j}  \sigma_y \left(\unphmy_{ij}, \unphsy_{ij}\right) \leq 1, \quad
\frac{2 \lambda_{y_j}}{\mupy} \sigma_y \left( \unphsy_{ij}, \unphpy_{ij} \right) \leq 1
\end{aligned}
\end{gathered}
\end{equation}
where $\lambda_{x_i} = \frac{\Delta t}{k_x \Delta x_i}, \lambda_{y_j} = \frac{\Delta t}{k_y \Delta y_j}$ for all $i,j$.
\end{lemma}

As in 1-D, we now show that admissibility of $\bw_{i  j}^{n + \frac{1}{2}, \ast x}, \bw_{i  j}^{n + \frac{1}{2}, \ast y}$ can also be reduced to admissibility of $\uspmx_{i  j}, \uspmy_{ij}$, similar to Lemma~\ref{lemma:muscl.step3.conservative}.
Expanding the definition of $\unphsy_{ij}$ gives us
\begin{equation}
\bw_{ij}^{\nph, \ast y} = 2 \bw_{ij}^n -(\mumy  \unmy_{i  j} + \mupy  \unpy_{ij}) - \frac{\Delta t}{\Delta x_i}  \left(\ff (\unmx_{i  j}) -  \ff (\unpx_{ij}) \right) - \frac{\Delta t}{\Delta y_j}  \left(\fg (\unmy_{i  j}) -  \fg (\unpy_{ij}) \right)
\end{equation}
If we obtain the admissibility of

\begin{equation}
\mathcal \eta_{ij}^{* {yx}}:= 2 \bw_{i  j}^n - (\mumy  \unmy_{i  j} + \mupy
\unpy_{i  j}) - \frac{\Delta t}{k_x \Delta
x_i}  \left(  \ff \left( \unmx_{i  j} \right) -
\ff \left( \unpx_{i  j} \right) \right)
\label{eq:adm2.coeff1}
\end{equation}
and
\begin{equation}
\mathcal \eta_{ij}^{* {yy}}:= 2 \bw_{i  j}^n - (\mumy  \unmy_{i  j} + \mupy
\unpy_{i  j}) - \frac{\Delta t}{k_y \Delta
y_j}  \left(  \fg \left( \unmy_{i  j} \right) -
\fg \left( \unpy_{i  j} \right) \right)
\label{eq:adm2.coeff2}
\end{equation}
for some $k_x,k_y \in [0,1]$ with $k_x+k_y=1,$ then the admissibility of $\bw_{i
j}^{\nph, \ast y}$ follows as we can write it  as  a convex combination
\begin{align*}
 \bw_{i  j}^{\nph, \ast y} = k_x \eta_{ij}^{\ast {yx}}+  k_y\eta_{ij}^{\ast {yx}}
\end{align*}
and obtain the admissibility of $\bw_{ij}^{\nph, \ast y}$. Thus, we need to limit the slope so that~(\ref{eq:adm2.coeff1},
\ref{eq:adm2.coeff2}) are admissibile. To that end, define $\bw_{i  j}^{\ast \ast {y  x}}, \bw_{ij}^{\ast \ast {y  y}}$ to satisfy
\begin{align*}
\mumx  \unmx_{i  j} + \bw_{i  j}^{\ast \ast {y
x}} + \mupx  \unpx_{i  j} = & 2 \left( 2
\bw_{i  j}^n - (\mumy  \unmy_{i  j} + \mupy
\unpy_{i  j}) \right)\\
\mumy  \unmy_{i  j} + \bw_{i  j}^{\ast \ast {y
y}} + \mupy  \unpy_{i  j} = & 2 \left( 2
\bw_{i  j}^n - (\mumy  \unmy_{i  j} + \mupy \unpy_{i  j}) \right)
\end{align*}
respectively.
Consequently,
\begin{align*}
\eta_{ij}^{* {yx}} &= \frac{1}{2} ( \mumx  \unmx_{i  j} + \bw_{i  j}^{\ast \ast {y
x}} + \mupx  \unpx_{i  j})-\frac{\Delta t}{ k_x\Delta x_i}\left(  \ff \left( \unmx_{i  j} \right) -
\ff \left( \unpx_{i  j} \right) \right)\\
\eta_{ij}^{* {yy}} & = \frac{1}{2}( \mumy  \unmy_{i  j} + \bw_{i  j}^{\ast \ast {y
y}} + \mupy  \unpy_{i  j})-\frac{\Delta t}{k_y \Delta
y_j}  \left(  \fg \left( \unmy_{i  j} \right) -
\fg \left( \unpy_{i  j} \right) \right)
\end{align*}
Then,  assuming the admissibility of $\bw_{ij}^{\ast \ast {y  x}}, \bw_{i  j}^{\ast \ast {y  y}}$ and proceeding as in the proof of Lemma \ref{lemma:muscl.step3.wss}, we can ensure that $\eta_{ij}^{* {yx}}, \eta_{ij}^{* {yy}} \in \Uad$ and thus $\bw_{i  j}^{\nph, \ast y} \in \Uad$. Furthermore, since the reconstruction is conservative
\[
\mumy  \unmy_{i  j} + \mupy  \unpy_{ij} = \mumx  \unmx_{i  j} + \mupx  \unpx_{i  j} = \bw_{i  j}^n
\]
Thus, admissibility of $\bw_{ij}^{\ast \ast {yx}}, \bw_{ij}^{\ast\ast {yy}}$ is obtained as
\[
\bw_{i  j}^{\ast \ast {yx}} = \bw_{ij}^{\ast\ast {yy}} = \bw_{i  j}^n
\]
The arguments for admissibility of $\bw_{i  j}^{\nph, \ast x}$ are similar. The admissibility criteria of $\bw_{i  j}^{\nph, \ast x}, \bw_{i  j}^{\nph, \ast y}$ are summarised in the following lemma.
\begin{lemma}
\label{lemma:2d.muscl.step3.conservative}Assume that $\correction{\bw_{ij}^n} \in \Uad$ and $\unpmx_{ij}, \unpmy_{ij} \in \Uad$ for all $i,j$ with conservative reconstruction. Also assume the CFL restrictions
\begin{equation}
\label{eq:2d.new.cfl3.conservative}
\begin{split}
\max_{i,j} \frac{\lambda_{x_i}}{\mumx} \sigma_x \left(\bw_{ij}^n, \unmx_{ij} \right) \leq 1, \qquad
\max_{i,j} \frac{\lambda_{x_i}}{\mupx} \sigma_x \left(\unpx_{ij}, \bw_{ij}^n \right) \leq 1 \\
\max_{i,j} \frac{\lambda_{y_j}}{\mumy} \sigma_y \left(\bw_{ij}^n, \unmy_{ij} \right) \leq 1, \qquad
\max_{i,j} \frac{\lambda_{y_j}}{\mupy} \sigma_y \left(\unpy_{ij}, \bw_{ij}^n \right)\leq 1
\end{split}
\end{equation}
where $\lambda_{x_i} = \frac{\Delta t}{k_x \Delta x_i}, \lambda_{y_j} = \frac{\Delta t}{k_y \Delta y_j}$ and $\mupmx, \mupmy$ are defined in~\eqref{eq:muxy.defn}. Then, $\bw_{i  j}^{\nph, \ast x}, \bw_{i  j}^{\nph, \ast y}$ defined in~\eqref{eq:unphsxy} are in $\Uad$.
\end{lemma}

Combining Lemmas~\ref{lemma:m.h.step.1.2d},~\ref{lemma:2d.muscl.step2.general},~\ref{lemma:2d.muscl.step3.conservative}, we will have the 2-D result corresponding to Theorem \ref{thm:final.condn.conservative} with the same proof.
\begin{theorem}
\label{thm:final.condn.conservative.2d}Let $\bw_{ij}^n \in \Uad$ for all $i,j$ and $\unpmx_{ij}, \unpmy_{ij}$ be the conservative reconstructions defined as
\[
\unpmx_{ij} = \bw_{ij}^n + (x_\correction{\ipmh} - x_i) \slope^x_i, \qquad
\unpmy_{ij} = \bw_{ij}^n + (y_\jpmh - y_j) \slope^y_j
\]
so that $\uspmx_{ij}, \uspmy_{ij}$~\eqref{eq:ustar.2d} are given by
\[
\uspmx_{ij} = \bw_{ij}^n + 2(x_\correction{\ipmh} - x_i) \slope^x_i,\qquad \uspmy_{ij} = \bw_{ij}^n + 2(y_\jpmh - y_j) \slope^y_j
\]
Assume that the slopes $\slope^x_i, \slope^y_j$ are chosen to satisfy $\uspmx_{ij}, \uspmy_{ij} \in \Uad$  for all $i,j$ and that the CFL restrictions~(\ref{eq:new.cfl1.2d}, \ref{eq:2d.new.cfl2}, \ref{eq:2d.new.cfl3.conservative}) are satisfied. Then the updated solution $\bw_{ij}^{n+1}$ of MUSCL-Hancock procedure is in $\Uad$.
\end{theorem}
\section{Limiting numerical flux in 2-D}
\label{sec:2d.admissibility}
Consider the 2-D hyperbolic conservation law~\eqref{eq:2d.hyp.con.law}. The Lax-Wendroff update is
\[
(\bw_{ij}^e)^{n+1} = (\bw_{ij}^e)^n
- \Delta t \left[ \frac{1}{\Delta x_e} \pd{F_h^\ee}{\xi} (\xi_i, \xi_j) + \frac{1}{\Delta y_e} \pd{G_h^\ee}{\eta}(\xi_i, \xi_j) \right ], \qquad 0 \le i,j \le N
\]
where $F^\ee_h, G^\ee_h$ are continuous time-averaged fluxes~\eqref{eq:frcontflux} in the $x,y$ directions for the grid element $\ee = (\ex, \ey)$. The reader is referred to Section 10 of~\cite{babbar2022} for more details of the 2-D Lax-Wendroff Flux Reconstruction (LWFR) scheme. Since the 2-D scheme is formed by taking a tensor product of the 1-D scheme, Theorem~\ref{thm:lwfr.admissibility} applies, i.e., the scheme will be admissibility preserving in means (Definition~\ref{defn:mean.pres}) if we choose the blended numerical flux such that the lower order updates are admissible at solution points adjacent to the interfaces. Thus, we now explain the process of constructing the numerical flux where, to minimize storage requirements and memory reads, we will perform the correction within the interface loop where only one of $x$ or $y$ flux will be available in one iteration. Thus theoretical justification for the algorithm comes from breaking 2-D lower order updates into 1-D convex combinations. The general structure of the LWFR Algorithm~\ref{alg:high.level.lw} will remain the same. Here, we justify Algorithm~\ref{alg:blended.flux.2d} for construction of blended $x$ flux with knowledge of only the $x$ flux. The algorithm for blended $y$ fluxes will be analogous.

We consider the calculation of the blended numerical flux for a corner solution point of the element, as this situation differs from 1-D, due to the fact that a corner solution point is adjacent to both $x$ and $y$ interfaces. In particular, we consider the bottom-left corner point $\bo = (0,0)$ and show that the procedure in Algorithm~\ref{alg:blended.flux.2d} ensures admissibility at such points. The same justification applies to other corner and non-corner points. For the element $\ee = (\ex, \ey)$, denoting interfaces along $x,y$ directions as $\eexpmh$, $\eeypmh$, we consider the update at the bottom left corner $\bo = (0,0)$, suppressing the local solution point index $i=0$ or $j=0$ when considering the FR interface fluxes. The lower order update is given by
\begin{align*}
\tilde{\bw}_{\bo}^{n + 1}
& =
(u_\bo^{\mathbf{e}})^n
- \frac{\Delta t}{\Delta x_e w_0} (f_{(\half, 0)}^\ee - \tilde{F}_\eexmh)
- \frac{\Delta t}{\Delta y_e w_0} (g_{(0, \half)}^\ee - \tilde{G}_\eeymh)
\end{align*}
where $\tilde{F}_\eexmh, \tilde{G}_\eeymh$ are heuristically \correction{guessed} candidates for the blended numerical flux~\eqref{eq:Fblend}. Pick $k_x, k_y > 0$ such that $k_x + k_y = 1$ and
\begin{equation}
\label{eq:low.update.2d}
\begin{split}
\tilde{u}_x^{\text{low}, n + 1} & = (u_\bo^{\mathbf{e}})^n -
\frac{\Delta t}{k_x \Delta x_e w_0} (\ff^\mathbf{e}_{(\half, 0)} - \ff_{(\exmh, \ey)} )\\
\tilde{u}_y^{\text{low}, n + 1} & = (u_\bo^{\mathbf{e}})^n -
\frac{\Delta t}{k_y \Delta y_e w_0} (\fg^\mathbf{e}_{(0, \half)} - \fg_{(\ex, \eymh)})
\end{split}
\end{equation}
satisfy
\begin{equation}\label{eq:2d.low.update.admissibility.condn}
\tilde{\bw}_x^{\text{low}, n + 1}, \tilde{\bw}_y^{\text{low}, n + 1} \in  \Uad
\end{equation}
Such $k_x, k_y$ exist because the lower order scheme with lower order flux at element interfaces is admissibility preserving. The choice of $k_x, k_y$ should be made so that~\eqref{eq:2d.low.update.admissibility.condn} is satisfied with the least time step restriction, but we have found the Fourier stability restriction imposed by~\eqref{eq:time.step.2d} to be sufficient even with the most trivial choice of $k_x = k_y = \half$. The discussion of literature for the optimal choice of $k_x,k_y$ is the same as the one made for the 2-D MUSCL Hancock scheme~\eqref{eq:kx.defn} and is not repeated here. After the choice of $k_x, k_y$ is made, if we repeat the same procedure as in the 1-D case, we can perform slope limiting to find $F_{e_x - \half, e_y}$, $F_{e_x, e_y - \half}$ such that
\begin{align}
\tilde{\bw}_x^{n + 1} & = (u_\bo^{\mathbf{e}})^n -
\frac{\Delta t}{k_x \Delta x_e w_0} (\ff_{(\half,0)}^\mathbf{e}
 - F_{(e_x - \half, e_y)}) \label{eq:2d.adm.numflux.desired.x} \\
\tilde{u}_y^{n + 1} & = (u_\bo^{\mathbf{e}})^n -
\frac{\Delta t}{k_y \Delta y_e w_0} (\fg^\mathbf{e}_{(0, \half)}
- G_{(e_x, e_y - \half)}) \label{eq:2d.adm.numflux.desired.y}
\end{align}
are also in the admissible region. Then, we will get
\begin{equation}
\label{eq:2d.xy.implies.admissibility}
k_x \tilde{\bw}_x^{n+1}
+ k_y \tilde{\bw}_y^{n+1}
= \tilde{\bw}_{\bo}^{n+1}
\end{equation}
We now justify Algorithm~\ref{alg:blended.flux.2d} as follows. Algorithm~\ref{alg:blended.flux.2d} corrects the numerical fluxes during the loop over $x$ interfaces to enforce admissibility of $\tilde{\bw}_x^{n + 1}$~\eqref{eq:2d.adm.numflux.desired.x} at all solution points neighbouring $x$ interfaces including the corner solution points, and the analogous algorithm for $y$ interfaces will ensure admissibility of $\tilde{u}_y^{n + 1}$~\eqref{eq:2d.adm.numflux.desired.y} at all solution points neighbouring $y$ interfaces including the corner points. At the end of the loop over interfaces, \eqref{eq:2d.xy.implies.admissibility} will ensure that lower order updates at all solutions points neighbouring interfaces
are admissible and Algorithm~\ref{alg:high.level.lw} will be an admissibility preserving Lax-Wendroff scheme for 2-D if we compute the blended numerical fluxes $F_{(e_x+\half, e_y)}, F_{(e_x, e_y+\half)}$ using Algorithm~\ref{alg:blended.flux.2d} and its counterpart in the $y$ direction.
\begin{algorithm}
\caption{Computation of blended flux $F_{e_x+\half, e_y, j}$ where $(e_x+\half, e_y)$ are the interface indices and $j \in \{0,\dots, N\}$ is the solution point index on the interface}
\textbf{Input:} $F_{e_x + \half, e_y, j}^\text{LW}, f_{e_x+\half, e_y, j}, f^{\ex + 1,e_y}_{\half, j}, f^\ee_{\Nmh, j}, u_{(0,j)}^{\ex + 1, \ey}, u_{(0,j)}^\ee, \alpha_\ee, \alpha_{e_x + 1, e_y}, k^{e_x,e_y}_x, k^{e_x+1,e_y}_x$ \\
\textbf{Output:} $F_{e_x+\half, e_y, j}$
\label{alg:blended.flux.2d}
\begin{algorithmic} 
\State $\overline{\alpha} = \frac{\alpha_{\ex, \ey} + \alpha_{\ex+1, \ey}}{2}$

\State $k_x^0, k_x^N = k_x^{\ex, \ey}, k_x^{\ex+1, \ey}$ \Comment{For ease of writing}

\State $F_{e_x + \half, e_y, j} \gets (1 - \overline{\alpha}) F_{e_x + \half, e_y, j}^\text{LW} + \overline{\alpha}f_{e_x+\half, e_y, j}$ \Comment{Heuristic guess to control oscillations}

\State $\tilde{\bw}_{0}^{n+1} \gets \correction{(\bw^{e_x+1, e_y}_{0,j})^n}
- \frac{\Delta t}{k_x^0 w_0 \Delta x_{e+1}} (f^{\ex+1,e_y}_{\half, j} - F_{e_x + \half, e_y, j})$ \Comment{FV inner updates with guessed $F_{e_x + \half, e_y, j}$}

\State $\tilde{\bw}_{N}^{n+1} \gets \correction{(\bw^{e_x,e_y}_{N,j})^n} - \frac{\Delta t}{k_x^N w_N \Delta x_{e}} (F_{e_x + \half, e_y, j} - f^\ee_{(\Nmh,j)})$

\State $\utilow_{0} = \correction{(\bw^{e_x+1, e_y}_{0,j})^n}
- \frac{\Delta t}{k_x^0 w_0 \Delta x_{e+1}} ( f^{\ex+1,\ey}_{\half,j} - f_{e_x+\half, e_y, j} )$ \Comment{FV inner updates with $f_{e_x+\half, e_y, j}$}

\State $\utilow_{N} = \correction{(\bw^{e_x,e_y}_{N,j})^n} - \frac{\Delta t}{k_x^N w_N \Delta x_{e}} (f_{e_x+\half, e_y, j} - f^\ee_{\Nmh,j})$

\For{$k$=1:$K$} \Comment{Correct $F_{\exph, e_y, j}$ for $K$ admissibility constraints}

\State $\theta \gets \min \left(\min_{l=0,N} \left| \frac{\epsilon_l - p_k(\tilde{\bw}_l^{n+1})}{p_k(\tilde{\bw}_l^{\text{low},n+1}) -  p_k(\tilde{\bw}_l^{n+1})} \right |, 1 \right)$

\State $F_{e_x+\half, e_y, j} \gets \theta F_{e_x+\half, e_y, j} + (1-\theta) f_{e_x+\half, e_y, j}$

\State $\tilde{\bw}_{0}^{n+1} \gets \correction{(\bw^{e_x+1, e_y}_{0,j})^n}
- \frac{\Delta t}{k_x^0 w_0 \Delta x_{e+1}} (f^{\ex+1,\ey}_{\half,j} - F_{e_x + \half, e_y, j})$ \Comment{FV inner updates with guessed $F_{e_x + \half, e_y, j}$}

\State $\tilde{\bw}_{N}^{n+1} \gets \correction{(\bw^{e_x,e_y}_{N,j})^n} - \frac{\Delta t}{k_x^N w_N \Delta x_{e}} (F_{e_x + \half, e_y, j} - f^\ee_{(\Nmh,j)})$
\EndFor
\end{algorithmic}
\end{algorithm}

\color{black}
\bibliographystyle{siam}
\bibliography{references}
\end{document}